\documentclass[11pt]{amsart}

\usepackage{amsthm,xcolor} 
\usepackage{graphicx} 
\usepackage{amssymb,latexsym,times,comment,todonotes}
\usepackage{amsmath}

\usepackage{fontenc}
\usepackage{amsfonts}
\usepackage{amsmath}
\usepackage{enumerate}

\usepackage{mathrsfs}
\usepackage{mathabx}
\usepackage{tikz}
\usepackage{marginnote}
 \usepackage[mathscr]{euscript}

\newcommand{\xdownarrow}[1]{%
  {\left\downarrow\vbox to #1{}\right.\kern-\nulldelimiterspace}
}
\def\rl{\raisebox{2pt}{\line(1,0){10}}}
\def\rll{\raisebox{2pt}{\line(1,0){25}}}
\def\occ{\mbox{occ}}\newcommand{\Lko}{L_{\kappa^+\omega}}
\newcommand{\To}[3]{#1_{#2}, \ldots, #1_{#3}}

\newcommand{\N}{{\cal N}}

\def\two{\mbox{II}}
\def\one{\mbox{I}}
\def\ba{\bar{a}}
\def\bb{\bar{b}}
\def\bc{\bar{c}}

\def\bx{\bar{x}}
\def\by{\bar{y}}

\def\cof{\mbox{cof}}

\def\ma{\mathfrak{M}}
\def\mm{\mathfrak{M}}
\def\mn{\mathfrak{N}}

\def\oN{\mathbb N}

\def\oR{\mathbb R}

\def\FS{\mbox{FS}}

\newtheorem{theorem}{Theorem}[section]
\newtheorem{definition}[theorem]{Definition}
\newtheorem{lemma}[theorem]{Lemma}
\newtheorem{proposition}[theorem]{Proposition}%
\newtheorem{claim}[theorem]{Claim}
\newtheorem{corollary}[theorem]{Corollary}

\newtheorem{remark}[theorem]{Remark}

\def\vx{\vec{x}}
\def\vy{\vec{y}}
\def\vz{\vec{z}}
\def\len{\mbox{len}}
\definecolor{forestgreen}{rgb}{0.13, 0.55, 0.13}
\begin{document}

\def\rest{\restriction}
\def\ran{\mathop{{\rm ran}}}
\def\dom{\mathop{{\rm dom}}}
\def\DG{\rm DG}
\def\U{\mathcal{U}}
\title{On some infinitary logics}
\author[Jouko V\"a\"an\"anen]{Jouko V\"a\"an\"anen}
\address[Jouko V\"a\"an\"anen]{
Department of Mathematics and Statistics\\
University of Helsinki\\
Yliopistonkatu 3\\
00014 Helsinki, Finland}
\email{jouko.vaananen@helsinki.fi}
\urladdr{https://www.mv.helsinki.fi/home/jvaanane/}
\author[Boban Veli\v{c}kovi\'c]{Boban Veli\v{c}kovi\'c}
\address[Boban Veli\v{c}kovi\'c]{
Institut de Math\'ematiques de Jussieu - Paris Rive Gauche (IMJ-PRG)\\
Universit\'e Paris Cit\'e\\
B\^atiment Sophie Germain\\
8 Place Aur\'elie Nemours \\ 75013 Paris, France}
\email{boban@math.univ-paris-diderot.fr}
\urladdr{https://webusers.imj-prg.fr/~boban.velickovic/}

%\author{Jouko V\"a\"an\"anen\\ University of Helsinki\\ \and Boban Velickovic\\ University of Paris }
%\date{}

\begin{abstract}
    We define a new class of infinitary logics $\mathscr L^1_{\kappa,\alpha}$ generalizing Shelah's logic $\mathbb L^1_\kappa$ defined in \cite{MR2869022}. If $\kappa=\beth_\kappa$ and $\alpha <\kappa$ is infinite then our logic coincides with $\mathbb L^1_\kappa$.
    We study the relation between these logics for different parameters $\kappa$ and $\alpha$. We give many examples
    of classes of structures that can or cannot be defined in these logics. Finally, we give a different version
    of Lindstr\"{o}m's Theorem for $\mathbb L^1_\kappa$ in terms of the $\phi$-submodel relation. 
    %and the Craig Interpolation theorem Just like $\mathbb L^1_\kappa$, our logics permit a Lindstr\"om Theorem and the Craig Interpolation Theorem. We argue that our logics are defined in a more canonical way than $\mathbb L^1_\kappa$. We present several examples about the expressive power of our logics.
\end{abstract}


\maketitle
\tableofcontents
%\begin{definition}

%\keywords{derived limit, additivity, strong homology, weak diamond}
%\subjclass[2010]{Primary: 03E35, 03E75, 18E25, 55Nxx}

\section{Introduction}

 The Lindstr\"om Theorem \cite{MR244013}, extending a result of Mostowski \cite{MR0089816}, characterized first order logic as the maximal logic satisfying the Compactness Theorem\footnote{The Compactness Theorem says that every theory, every finite subset of which has a model, has itself a model.} and the Downward L\"owenheim-Skolem Theorem\footnote{The Downward L\"owenheim-Skolem Theorem says that every sentence with a model has a countable model.}. One interpretation of this result is that the Compactness Theorem and the Downward L\"owenheim-Skolem Theorem are the only properties one needs to know about the model theory of first order logic, as all the rest follows from them. On the other hand, one may interpret Lindstr\"om's result as the statement that if we desire to extend the expressive power of first order logic, and there are a lot of reasons to entertain such desire, we have to give up either the Compactness Theorem or the Downward L\"owenheim-Skolem Theorem. When the result started to become more widely known in the early 1970s, a new field of logic called \emph{abstract model theory} (a.k.a. extended model theory)  was born. The new field was  advocated in particular by Barwise in  \cite{MR0376337}, one of the early seminal papers on the topic, where he wrote
 \begin{quote}
``Abstract logic is a relatively new subject which stands in relation to model theory roughly as model theory stands to modern algebra, or as functional analysis stands to classical `hard' analysis. It attempts to bring some kind of order to the whole study of generalizations of the lower predicate calculus."
\end{quote} 
By this time a host of generalizations of first order logic were known. \emph{Second order logic} as an extension of first order logic was explicitly defined  by Hilbert and Ackermann \cite{zbMATH02514643} as early as 1928. By the 1970s the power of second order logic as an extension of first order logic was widely recognised but at the same time it was considered \emph{too} powerful\footnote{``Judged by the standards of first-order logic, the model theory of second-order logic was deemed unmanageable." \cite{MR0819532} } . However, \emph{stationary logic}, a fragment  of second order logic, turned out to be  axiomatisable\footnote{A logic is axiomatizable if its sentences are finite strings of symbols and the set of (G\"odel numbers of) its valid sentences is computably enumerable.} and satisfy the Compactness Theorem in countable vocabularies \cite{MR0486629} whereupon it was deemed almost as well-behaved as first order logic itself. \emph{Generalized quantifiers} were introduced by Mostowski 
\cite{MR0089816} in 1957 who also showed that the quantifier ``there are infinitely many" is non-axiomatisable. Mostowski asked whether the similar quantifier ``there are uncountable many", which seems to avoid the obstacle that rendered ``there are infinitely many"  non-axiomatisable, is in fact axiomatisable. This question was answered positively by Vaught \cite{MR0166086}, see also \cite{MR263616}. \emph{Infinitary languages} were introduced in 1958 by Scott and Tarski \cite{MR0099914,MR0099915}. Initially model theoretic questions about infinitary languages led immediately to large cardinals. However, Karp \cite{MR0176910} and Scott \cite{MR0200133} showed that  ${\mathscr L}_{\omega_1\omega}$ offers a case where an infinitary language can have completeness theorems and model theoretic properties such as the Craig Interpolation Theorem.\footnote{The Craig Interpolation Theorem says that if $\phi\to\psi$ is a valid implication in the logic, then there is an ``interpolant" $\theta$ such that $\phi\to\theta$ is valid, $\theta\to\psi$ is valid, and, most importantly, the vocabulary of $\theta$ is included in the intersection of the vocabularies of $\phi$ and $\psi$.} 

For all the above different types of extensions of first order logic many variations were introduced over time. By the early 1970s it made sense to ask what are the common properties of all these different logics? Do some properties of logics imply each other irrespective of the logic? Do some logics have particular properties that characterize them and isolate them among all the other logics? Abstract model theory arose to answer such questions.

So, what exactly \emph{is} abstract (or extended) model theory? Barwise describes its content and basic motivation in  \cite{MR0819532} as follows: 
\begin{quote}
``In extended model theory one asks, `What is the logic of specific mathematical concepts?' More explicitly, given a particular mathematical property (like being a finite, infinite, countable, uncountable, or open set, or being a well-ordering or a continuous function, or having probability greater than some real number r), what is the logic implicit in the mathematician's use of the property? 
%What sorts of mathematical structures isolate the property most naturally? What sorts of languages best mirror the mathematician's talk about the property? What forms of reasoning about it are legitimate? Which other properties are implicit in it or are presupposed by it?

\ldots
\medbreak

%On the common sense view of logic, all the concepts we use to cope with and organize our world have their own logic. As logicians, we are perfectly entitled to delve into their logic. However, as mathematical logicians, or metamathematicians, our interest is more specialized. 
What we seek to understand is the logic of precise mathematical concepts. Extended model theory makes a frontal attack on this problem by, where appropriate, building `logics' to get answers to some of the questions listed above."

\ldots
\medbreak

Mathematicians often lose patience with logic simply because so many notions from mathematics lie outside the scope of first-order logic, and they have been told that that is logic. The study of model-theoretic logics should change that, by getting at the logic of the concepts mathematicians actually use, by finding applications, and by the isolation of still new concepts that enrich mathematics and logic."
\end{quote}

Abstract model theory was quite popular still in the 1980s. Much of the success is recorded in the handbook ``Model Theoretic Logics" by Barwise and Feferman \cite{MR819532}. On the matter of generalized quantifiers the seminal paper \cite{MR0808815} by Shelah appeared at the same time as  \cite{MR819532}, showing, among other things, that any valid implication in cofinality logic\footnote{The cofinality logic is the extension of first order logic by the so-called cofinality quantifier which says of a definable linear order that its cofinality is $\omega$. Remarkably, this extended logic satisfies the Compactness Theorem even in vocabularies of any cardinality \cite{MR0376334}.} can be interpolated in stationary logic. This suggests the problem of finding a logic between cofinality logic and stationary logic which has the Craig Interpolation Property, and even perhaps a 
Lindstr\"om Theorem. The problem is still open.
 
After the rather lively interest in abstract logic up until the 1980s, the subject fell into  oblivion, only to experience a resurgence recently. Definability in abstract logic was delineated with definability in set theory in  \cite{MR3519447} leading to a hierarchy of reflection principles in set theory, motivated by abstract logics. Abstract logics were used to construct new inner models of set theory in \cite{MR4290501}. A new infinitary logic was introduced by Shelah \cite{MR2869022}, one with the Craig Interpolation theorem and a Lindstr\"om type model theoretic characterization.
%Generalized quantifiers \cite{MR0808815}

%Abstract logic , \cite{MR2214627}

%Infinitary logic  \cite{MR3558585}\cite{MR2191236}\cite{MR2116833}, \cite{MR4646734}, \cite{MR4057944}

%Abstract logics in set theory \cite{MR3519447}, \cite{MR4290501}.

%Survays Model theoretic logics \cite{MR0819531}. ``beyond first order model theory" \cite{MR3726900}\cite{MR37269002}

% Abstract elementary classes \cite{MR4335884}

%L-1-kappa: \cite{MR4357456}\cite{MR4335884}

%Why abstract model theory? To put order into the map of extensions of FO. To emphasise the role of model theory, and sometimes set theory, too, rather than syntax, which may be accidental, in properties of extensions of FO.

Despite intensive efforts, no  model theoretic characterizations in the style of Lindstr\"om were obtained for logics with generalized quantifiers nor for infinitary logics, until \cite{MR2869022}. There is a generally held intuition, based on experience, that if a logic has a model theoretic characterization, it should satisfy the Craig Interpolation Theorem\footnote{In fact, Lindström came to his characterization of first order logic, what we nowadays call Lindström Theorem, when he tried to find a new proof of  the Craig Interpolation Theorem, as he writes in \cite{MR3409279}.}. Thus a good direction to look for such logics is to look for infinitary logics with the Craig Interpolation Theorem, which are not many. Malitz \cite{MR290943} proved that infinitary logics ${\mathscr L}_{\kappa\omega}$, for $\kappa\ge\aleph_2$, fail to have Craig Interpolation but, if $\kappa$ is regular,   any valid implication can be   interpolated  in ${\mathscr L}_{(2^{<\kappa})^+,\kappa}$. Shelah's $\mathbb L^1_\kappa$, which was defined in \cite{MR2869022}, 
has Craig Interpolation, and is actually between ${\mathscr L}_{\kappa\omega}$ and ${\mathscr L}_{\kappa\kappa}$ for strongly inaccessible $\kappa$, and also for many other $\kappa$. Moreover, $\mathbb L^1_\kappa$ has a model theoretic characterization in the style of Lindstr\"om Theorem. 

Shelah's new infinitary logic  $\mathbb L^1_\kappa$ fulfils the hopes of  those who got excited about abstract model theory in the 1970s. There are overall very few abstract logics that have the Craig Interpolation Theorem, essentially just ${\mathscr L}_{\omega_1,\omega}$ and its countable admissible fragments\footnote{There are examples among logics that do not extend first order logic or which are based on non-first order (e.g. topological) structures \cite{MR0560706}.}. Higher up in the hierarchy of infinitary logics we have nothing but counterexamples, one after another. Among logics with generalized quantifiers there is likewise not a single example of a logic with Craig Interpolation, again only counterexamples. The situation was similar with Lindstr\"om's Theorem up and until the advent of $\mathbb L^1_\kappa$. There were simply no other examples than first order logic\footnote{There are examples if we modify the concept of a structure \cite{MR0560706}.}.

A remarkable feature of $\mathbb L^1_\kappa$ is that its definition is entirely semantical. There is no discernible syntax. Until, that is, someone invents a syntax for it. The definition of an abstract logic does not presuppose a syntax and in that sense $\mathbb L^1_\kappa$ is a full-blooded abstract logic. It is abstract to the extent that no syntax is known for it. It is however totally conceivable that a syntax for $\mathbb L^1_\kappa$ will emerge in the future.

Thus $\mathbb L^1_\kappa$ satisfies the requirement of being a {\em nice logic}, but it has nevertheless attracted very little attention (this was actually predicted by Shelah himself in \cite{MR2869022}\footnote{Shelah writes in \cite[p. 398]{MR2869022} ``We feel that this is an important one among my works and will attract little attention."}).
In this paper we try to remedy this situation. The structure of the paper is the following: After some preliminaries in Section 2,  we  define in Section 3  a variation of Shelah's game by adding an extra ordinal parameter $\alpha$. Shelah's game corresponds to setting $\alpha=\omega$.
We also introduce the infinite version of the game. We study the relations between these games for various parameters. In Section 4 we give examples of properties that can or cannot be detected by our games. 
In Section 5 we consider a class of logics $\mathscr L^1_{\kappa,\alpha}$, for a cardinal $\kappa$ and an ordinal $\alpha <\kappa$. We show that if $\kappa=\beth_\kappa$, then these logics coincide with $\mathbb L^1_\kappa$, for all $\alpha <\kappa$. In Section 6
we  give a version of Lindstr\"{o}m's theorem, very much in the spirit of Shelah's \cite{MR2869022}, characterizing the logic $\mathbb L^1_\kappa$ in terms of a $\phi$-submodel relation, for all formulas $\phi$.

\section{Preliminaries}

We start by reviewing the general framework of generalized logics.  We work in von Neumann–Bernays–Gödel set theory (NBG) which allows us to speak of proper classes. 
For simplicity we restrict ourselves to notions of logics based on conventional algebraic structures with a single sorted vocabulary. For a more systematic presentation we refer the reader to  \cite{MR0819533}.
Given a vocabulary $\tau$ we denote by ${\rm Str}[\tau]$ the class of $\tau$-structures. Given a structure $\mathfrak{A}$ we write $\tau_\mathfrak{A}$ for the vocabulary of $\mathfrak{A}$. If $\sigma \subseteq \tau$ are vocabularies and $\mathfrak{A}$
is a $\tau$-structure we write $\mathfrak{A}\restriction \sigma$
for the $\sigma$-{\em reduct} of $\mathfrak{A}$, which is obtained by "forgetting" the interpretations of the symbols of $\tau$ which do not belong to $\sigma$.
Given two vocabularies $\tau$ and $\sigma$, a bijection $\rho: \tau \to \sigma$
is called a {\em renaming} if it maps constant symbols to constant symbols, 
function symbols to functions symbols of the same arity, and similarly for relation symbols. Given a renaming $\rho: \tau \to \sigma$ and a $\tau$-structure $\mathfrak{A}$
we can rename $\mathfrak{A}$ by $\rho$ in the natural way, thus obtaining a $\sigma$-structure $\mathfrak{B}= \mathfrak{A}^\rho$. 
We now recall the basic definition of a generalized logic. 

\begin{definition}\label{definition:logic}
A {\em logic} is a pair $(\mathscr L, \models_{\mathscr L})$ defined on vocabularies, where $\mathscr L[\tau]$ is the set (or proper class\footnote{Some authors insist that $\mathscr L[\tau]$ is a set, but this leaves out logics such as $\mathscr L_{\infty,\omega}$.})  of all $\mathscr L$-{\em sentences} of vocabulary $\tau$, and $\models_{\mathscr L}$ (the $\mathscr L$-{\em satisfaction relation}) is a relation between $\tau$ structures and $\mathscr L$-sentences and the following hold: 
\begin{enumerate}
    \item If $\tau \subseteq \sigma$ then $\mathscr L[\tau]\subseteq \mathscr L[\sigma]$;
    \item If $\mathfrak{A}  \models_{\mathscr L}\phi$ then $\phi \in \mathscr L[\tau_{\mathfrak{A}}]$;
    \item (Isomorphism property) If $\mathfrak{A}\models_{\mathscr L} \phi$
    and $\mathfrak{B}\simeq \mathfrak{A}$ then $\mathfrak{B}\models_{\mathscr L} \phi$;
    \item (Reduct property) If $\phi \in \mathscr L[\tau]$, $\mathfrak A$ is a structure, and $\tau \subseteq \tau_{\mathfrak{A}}$ then 
    \[
    \mathfrak{A} \models_{\mathscr L}\phi \hspace{5mm} \mbox{ iff } \hspace{5mm}
    \mathfrak{A} \restriction \tau \models_{\mathscr L} \phi.
    \] 
    \item (Renaming property) Let $\rho: \tau \to \sigma$ be a renaming of vocabularies.
    Then for each $\phi \in \mathscr L[\tau]$ there is a sentence $\phi^\rho$ in
    $\mathscr L[\sigma]$ such that 
    \[
    \mathfrak{A} \models_{\mathscr L}\phi \hspace{5mm} \mbox{ iff } \hspace{5mm} \mathfrak{A}^\rho \models_{\mathscr L} \phi^\rho. 
    \]
\end{enumerate}
\end{definition}

It is often convenient to have {\em variables} and {\em formulas} available in a logic. This is achieved in a natural way by generalizing Definition \ref{definition:logic}. 
Suppose $\tau$ is a vocabulary. Let $\{c_i: i \in I\}$ be a set of new constant symbols. Let $\tau'= \tau \cup \{ c_i : i\in I\}$ and let $\phi$ be a sentence in $\mathscr L[\tau']$. We consider $\phi$ as a formula in $\mathscr L[\tau]$ and we write $\phi(\bar x)$, where $\bar x= (x_i : i \in I)$ and the $x_i$ are new variables. 
Given a $\tau$-structure $\mathfrak{A}$ and a sequence $(a_i: i \in I)$ of elements of $A$ we write $\mathfrak{A}\models_{\mathscr L} \phi[\bar a]$
if by interpreting $c_i$ by $a_i$
we obtain a $\tau'$-structure $\mathfrak{A}'= (\mathfrak{A}, a_i ; i\in I)$ such that $\mathfrak{A}'\models_{\mathscr L} \phi$. 
In this way we can replace $\mathscr L$ by two functions ${\rm Sent}_{\mathscr L}[\tau]$ of $\mathscr L$-{\em sentences} of vocabulary $\tau$
and ${\rm Form}_{\mathscr L}[\tau]$ of $\mathscr L$-{\em formulas} of vocabulary $\tau$.
We can then define notions such as the free occurrence of variables in a canonical way. 
For further details and discussion, we refer the reader to \cite{MR0819533}.

Given a logic $\mathscr L$, a vocabulary $\tau$, and $\phi \in \mathscr L[\tau]$, we let ${\rm Mod}_{\mathscr L}^\tau(\phi)$ denote the class of all models of $\phi$, i.e. 
\[
{\rm Mod}_{\mathscr L}^\tau(\phi)= \{ \mathfrak{A}\in {\rm Str}[\tau]: 
\mathfrak{A}\models_{\mathscr L}\phi \}. 
\]
Given two sentences $\phi, \psi \in \mathscr L[\tau]$ we say that $\phi$ $\mathscr L$-{\em implies} $\psi$ and write 
\[
\phi \models_{\mathscr L} \psi
\]
if ${\rm Mod}_{\mathscr L}^\tau(\phi) \subseteq {\rm Mod}_{\mathscr L}^\tau(\psi)$, i.e. every $\tau$-structure that satisfies $\phi$ also satisfies $\psi$.

We write $\mathfrak{A}\equiv_{\mathscr L}\mathfrak{B}$ if $\mathfrak{A}$ and $\mathfrak{B}$ have the same vocabulary $\tau$ and satisfy the same ${\mathscr L}[\tau]$-sentences.

We let ${\mathscr L}_{\omega,\omega}$ denote the usual first order logic. 
We now recall some basic closure properties of a logic which ensure that it is at least as strong as ${\mathscr L}_{\omega,\omega}$.

\begin{definition}\label{definition:properties}
Let $\mathscr L$ be a logic. We have:
\begin{enumerate}
   
\item (Atom property) For all $\tau$ and all atomic $\phi \in {\mathscr L}_{\omega,\omega}[\tau]$, there is $\psi \in \mathscr L[\tau]$ such that
\[
{\rm Mod}^\tau_{\mathscr L}(\psi)= {\rm Mod}^\tau_{\mathscr L_{\omega,\omega}}(\phi).
\]
\item (Negation property) For all $\tau$ and all $\phi \in \mathscr L[\tau]$, there is a $\psi \in \mathscr L[\tau]$, denoted by $\neg \phi$, such that: 
\[
{\rm Mod}^\tau_{\mathscr L}(\psi) = {\rm Str}[\tau]\setminus {\rm Mod}^\tau_{\mathscr L}(\phi).
\]
\item (Conjunction property) For all $\tau$ and $\phi_0,\phi_1\in \mathscr L[\tau]$
there is $\psi\in \mathscr L[\tau]$, denoted by $\phi_0 \wedge \phi_1$,
such that: 
\[
{\rm Mod}^\tau_{\mathscr L}(\psi) = {\rm Mod}^\tau_{\mathscr L}(\phi_0) \cap
{\rm Mod}^\tau_{\mathscr L}(\phi_1).
\]
\item (Particularization property) If $\tau$ is a vocabulary and $c \in \tau$
is a constant symbol, then for every $\phi \in \mathscr L[\tau]$
there is $\psi \in \mathscr L[\tau \setminus \{ c\}]$ such that 
for all $(\tau \setminus \{ c \})$-structures $\mathfrak{A}$ 
we have $\mathfrak{A}\models_{\mathscr L}\psi$ iff there is $a \in A$ (the universe of $\mathfrak{A}$), such that $(\mathfrak{A},a)\models_{\mathscr L} \phi$. 
If such $\psi$ exists we denote it by $\exists c \phi$.
\end{enumerate}

\end{definition}
Let $\tau$ be a vocabulary and $\mathfrak{A}\in {\rm Str}[\tau]$.
We say that $C\subseteq A$ is $\tau$-{\em closed} if $C$ is nonempty, 
$c^{\mathfrak{A}}\in C$, for all constant symbols $c\in \tau$, and 
$C$ is closed under $f^{\mathfrak{A}}$, for all function symbols $f\in \tau$.

\begin{definition}\label{definition:relativization}
(Relativization property) We say that a logic $\mathscr L$ satisfies the {\em relativization property} if the following holds. Suppose $\sigma$ and $\tau$ are vocabularies, $c$ a constant symbol with $c\notin \sigma \cup \tau$, $\chi \in \mathscr L[\sigma \cup \{ c\}]$
and $\phi \in \mathscr L[\tau]$. Then there is $\psi \in \mathscr L[\sigma \cup \tau]$
such that if $\mathfrak{B}$ is a $(\sigma \cup \tau)$-structure and 
$\chi^{\mathfrak{B}} = \{ b\in B: (\mathfrak{B},b)\models_{\mathscr L} \chi\}$ is $\tau$-closed in $\mathfrak{B}$, then
\[
\mathfrak{B} \models_{\mathscr L} \psi \hspace{5mm} \mbox{iff} \hspace{5mm}
(\mathfrak{B}\restriction \tau) | \chi^{\mathfrak{B}} \models_{\mathscr L}\phi.
\]
Intuitively, $\psi$ is the relativization of $\phi$ to $\chi^{\mathfrak{B}}$ and is often written as 
$\phi^{\{ c | \chi(c)\}}$. If $\chi$ is the formula $P(c)$, for some unary predicate symbol $P$, we then simply write it as $\phi^P$.
\end{definition}

Obviously, ${\mathscr L}_{\omega,\omega}$ has the relativization property.

Relativization allows us to represent one $\tau$-structure in another one by relativizing to a definable subset of the bigger structure. However, if we want to be able to talk about many $\tau$-structures inside a $\sigma$-structure $\mathfrak A$ we should be able to express the parameterized version of the symbols of $\tau$ in the vocabulary of $\mathfrak{A}$ and then we should be able to say that for each suitable choice of the parameters the corresponding induced structure satisfies a given $\mathscr L[\tau]$-sentence sentence $\phi$. We only state a weaker form of substitution called {\em atomic substitution} which is sufficient for our needs. 

\begin{definition}\label{definition:substitution}
(Atomic substitution) We say that a logic $\mathscr L$ satisfies  {\em atomic substitution} if the following holds. Suppose that $\tau$ and $\sigma$ are vocabularies. Let $U\in \sigma$ be a binary relation.  Suppose for every constant symbol $c\in \tau$, we have a unary function symbol $K_c \in \sigma$, for every $n$-ary relation symbol $R\in \tau$, we have an $n+1$-ary relation symbol $K_R\in \sigma$, and for every $n$-ary function symbol $f\in \tau$, we have an $n+1$-ary function symbol $K_f\in \sigma$. Finally, let $d\in \sigma$ be a constant symbol. 
Now, let $\mathfrak{A}$ be a $\sigma$-structure and let $a=d^{\mathfrak{A}}$. 
%We want to define a $\tau$-structure $\mathfrak{B}_a$. 
Let 
\[
B_a=\{ b \in A: U^{\mathfrak{A}}(b,a)\}.
\]
If $R\in \tau$ is an $n$-ary relation symbol we can define an $n$-ary relation $R^{\mathfrak{B}_a}$ on $B_a$ by letting 
\[
R^{\mathfrak{B}_a}(b_0,\ldots, b_{n-1}) \hspace{5mm}\mbox{iff} \hspace{5mm} K_R^{\mathfrak{A}}(b_0,\ldots,b_{n-1},a).
\]
If $c\in \tau$ is a constant symbol, we let $c^{\mathfrak{B}_a}=K_c^{\mathfrak{A}}(a)$. 
If $f\in \tau$ is an $n$-ary function symbol,
we can define an $n$-ary function $f^{\mathfrak{B}_a}$ by letting $f^{\mathfrak{B}_a}(\bar{b})=K_f^{\mathfrak{A}}(\bar{b},a)$, for all $\bar{b}\in A^n$. 
%If $c^{\mathfrak{B}_a}\in B_a$, for all constant symbols $c\in \tau$, and $B_a$ is closed under the function $f^{\mathfrak{B}_a}$, for all function symbols $f\in \tau$, then we obtain a $\tau$-structure $\mathfrak{B}_a$.
If $B_a$ is closed under these interpretation of the symbols of $\tau$ we obtain a $\tau$-structure $\mathfrak{B}_a$. We require that for each $\tau$-sentence $\phi$, there is a $\sigma$-sentence $\psi$ such that 
\[
\mathfrak{A}\models_{\mathscr L}\psi \hspace{5mm} \mbox{iff \hspace{5mm} $\mathfrak{B}_a$ is a $\tau$-structure and } 
\mathfrak{B}_a \models_{\mathscr L} \phi.
\] 
\end{definition}
Let us note that the stronger form of substitution is obtained if the interpretation of the symbols of $\tau$ in the structure $\mathfrak{B}_a$ is given not by interpretations of corresponding symbols of $\sigma$, but by $\mathscr L[\sigma]$-formulas. 

Obviously, ${\mathscr L}_{\omega,\omega}$ has the atomic (and also the stronger) substitution property.

We say that a logic $\mathscr L$ is {\em small}, if for every signature $\tau$, the collection $\mathscr L[\tau]$ of $\tau$ sentences is a set. We say that a logic $\mathscr L$ is {\em regular} if  it is small and satisfies all the properties of Definition \ref{definition:logic} and Definition \ref{definition:substitution}. Note that relativization follows from atomic substitution. 

We shall also need the notion of the occurrence number. Given a logic $\mathscr L$, the {\em occurrence number} of $\mathscr L$,
if it exists, is the smallest cardinal $\gamma$ such that: 
\[
\mathscr L[\tau] = \bigcup_{\tau_0 \in [\tau]^{<\gamma}} \mathscr L[\tau_0].
\]

Note that the occurrence number of the usual first order logic ${\mathscr L}_{\omega,\omega}$ is $\omega$ and the occurrence number of the infinitary logics $\mathscr L_{\kappa,\omega}$ and 
$\mathscr L_{\kappa,\lambda}$, $\kappa$ regular, $\lambda\le\kappa$, is $\kappa$, see \cite{MR0819533}. 

There is a natural way to compare logics. 
\begin{definition}\label{definition:compare} 
Suppose $\mathscr L$ and $\mathscr L'$ are regular logics. 
\begin{enumerate}
    \item   We say that 
$\mathscr L'$ is {\em stronger} than $\mathscr L$, and write $\mathscr L \leq \mathscr L'$, if for every vocabulary $\tau$ and every $\mathscr L[\tau]$-sentence $\phi$ there is an $\mathscr L'[\tau]$ sentence $\psi$ such that ${\rm Mod}^\tau_{\mathscr L'}(\psi)= {\rm Mod}^\tau_{\mathscr L}(\phi)$. 
\item We say that $\mathscr L$ and $\mathscr L'$ are {\em equivalent} and write $\mathscr L \equiv \mathscr L'$ if $\mathscr L \leq \mathscr L'$ and $\mathscr L' \leq \mathscr L$.
\end{enumerate}
\end{definition}

We now recall some desirable properties of logics. We will restrict ourselves to regular logics to avoid some pathologies. 

\begin{definition}\label{definition:properties1} 
Let $\mathscr L$ be a regular logic and let $\kappa$ be an infinite cardinal. 
\begin{enumerate}
    \item We say that $\mathscr L$ is $\kappa$-{\em weakly compact} if for every vocabulary $\tau$ of cardinality at most $\kappa$, and every set $T$ of $L[\tau]$-sentences, if every  $T_0 \subseteq T$ of cardinality $<\kappa$ has a model, then so does $T$. We say that $\mathscr L$
    is $\kappa$-{\em compact} if this holds without any restriction on the cardinality of $\tau$.
    \item We say that $\mathscr L$ has the $\kappa$-{\em L\"owenheim-Skolem property}, if for every vocabulary $\tau$ and $\phi \in \mathscr L[\tau]$, if $\phi$ has a model then it has a model of cardinality $<\kappa$.
    \item We say that $\mathscr L$ satisfies the {\em Craig interpolation property} if for every vocabularies $\tau_0$ and $\tau_1$, and $\phi_0 \in \mathscr L[\tau_0]$ and $\phi_1 \in \mathscr L[\tau_1]$, if 
    $\phi_0\models _{\mathscr L} \phi_1$ then there is $\psi \in \mathscr L[\tau_0 \cap \tau_1]$ such that $\phi_0 \models_{\mathscr L}\psi$ and $\psi \models_{\mathscr L}\phi_1$. 
\end{enumerate}

\end{definition}

There are numerous other interesting properties of abstract logics, such as for instance the {\em Beth definability property}. We refer the reader to \cite{MR819532} and \cite{MR0344115} for details. 

One of the first results in abstract model theory is the following well-known result of Lindstr\"om.

\begin{theorem}[\cite{MR244013}]\label{theorem:Lindstrom}
Let $\mathscr L$ be a regular logic. Then the following are equivalent: 
\begin{enumerate}
    \item $\mathscr L$ is $\omega$-compact  and has the $\omega_1$-L\"owenheim-Skolem property.
    \item $\mathscr L \equiv \mathscr L_{\omega,\omega}$.
\end{enumerate}
\end{theorem}

It is possible to replace $(1)$ by a variety of other conditions. For instance, given a logic $\mathscr L$ we can define the $\mathscr L$-elementary submodel relation $\prec_{\mathscr L}$ as follows. Given two structures $\mathfrak{A}$
and $\mathfrak{B}$ in the same vocabulary $\tau$ such that $\mathfrak{A}$ is a substructure of $\mathfrak{B}$,
we enrich $\tau$ by adding a constant symbol $c_a$, for every element $a\in A$.
We then say that $\mathfrak{A}\prec_{\mathscr L} \mathfrak{B}$ 
if the expanded structures $(\mathfrak{A},a)_{a\in A}$ and 
$(\mathfrak{B}, a)_{a\in A}$ have the same $\mathscr L$ theories, i.e. satisfies the same sentences in the expanded language. 
Note that if the occurrence number of $\mathscr L$ is $\gamma$ then this is equivalent to saying
that for every $\bar{a}\in A^{<\gamma}$ we have, $(\mathfrak{A},\bar{a}) \equiv_{\mathscr L} (\mathfrak{B},\bar{a})$.

\begin{definition}\label{definition: union property}
Let $\mathscr L$ be a regular logic. We say that $\mathscr L$ satisfies the {\em Tarski union property} if, for every vocabulary $\tau$,
and every increasing sequence of $\tau$-structures $(\mathfrak{A}_n)_n$ such that $\mathfrak{A}_0\prec_{\mathscr L} \mathfrak{A}_1 \prec_{\mathscr L}\ldots$, if we let $\mathfrak{A}$ denote the union of the structures $\mathfrak{A}_n$,  then  $\mathfrak{A}_n \prec_{\mathscr L} \mathfrak{A}$, for all $n$.
\end{definition}

In \cite{MR0450023} Lindstr\"om showed that (1) in Theorem \ref{theorem:Lindstrom} can be replaced by: 
\medskip

    \begin{itemize}
    \item[($1'$)] $\mathscr L$ is $\omega$-compact and has the union property.
    \end{itemize}

It was felt that Lindstr\"om's theorem underscores the importance of first order logic as a maximal logic having some desirable properties. It then became a {\em holy grail} of Abstract Model Theory to find stronger logics which have similar characterizations by varying the defining properties. 
The search of such logics has been largely unsuccessful until Shelah introduced his logic $\mathbb L^1_\kappa$ in \cite{MR2869022}. Unfortunately, by that time there were very few people left working in Abstract Model Theory and therefore, as predicted by Shelah himself, the ideas were not picked up by other researchers. One of the goals of the current paper is to revive interest in this subject by initiating a systematic study of Shelah's logic and its variations.

\section{The similarity game}

A logic measures the similarity of two structures of the same vocabulary. The strongest similarity is if the two structures are isomorphic. Quite often we can think of a logic as a type of a game between two players, {\em Adam} and {\em Eve}. Let us suppose $\mathfrak{M}_0$ and $\mathfrak{M}_1$ are two structures in the same vocabulary $\tau$. We will write $M_i$ for the universe of $\mathfrak{M}_i$, and we will always assume that they are disjoint so that no ambiguity arises. If they are not, this can be arranged by a standard coding argument. 

Now, we think of our game as follows. Eve claims to  know that $\mathfrak{M}_0$ and $\mathfrak{M}_1$ are isomorphic, but she won't tell us the isomorphism. Instead she will answer a series of Adam's questions of a particular kind. In the process she has to match some elements of $M_0$ with elements of $M_1$, i.e. she reveals a part of the unknown isomorphism. So if at some position the matching Eve played is not  a partial isomorphism then Eve loses the game. The game may be finite or infinite, in fact we may allow games of uncountable length, although this will not be done in this paper (see e.g.
\cite{MR846604}, \cite{MR1191613}, \cite{MR1242054}).

For the logics we are interested in we will start by defining  the positions in the game and then describe the moves from one position to another. Our basic definition is a slight generalization of the one from \cite{MR2869022}. 
Throughout this section we fix two structures $\mathfrak{M}_0$
and $\mathfrak{M}_1$ in the same vocabulary $\tau$.

%
%Notation: $A(2n)=\bigcup_{i=0}^{n}A_{2i}$, $B(2n+1)=\bigcup_{i=0}^{n}B_{2i+1}$, $(A\cup B)(2n)=\bigcup_{i=0}^{n-1}(A_{2i}\cup B_{2i+1})\cup A_{2n}$, $(A\cup B)(2n+1)=\bigcup_{i=0}^{n}(A_{2i}\cup B_{2i+1})$.

Below in Definition~\ref{definition:game} we define the game which is the main concept of this paper. Since the game proceeds from position to position, we first define what we mean by a position. It is easier to explain the intuition of the game after we have actually defined the game, so we postpone giving the intuition until after Definition~\ref{definition:game}.

\begin{definition}\label{position} Given ordinals $\alpha$ and $\beta$, and a cardinal $\theta$, a $(\beta,\theta,\alpha)$-\emph{\bf position} for $(\mathfrak{M}_0,\mathfrak{M}_1)$ is a tuple
$p=(\beta_p,A_p^0,A_p^1,g_p,h_p)$, where
\begin{enumerate}
\item $\beta_p \leq \beta$.
\item $A_p^0\in[M_0]^{\leq \theta}$ and $A_p^1\in[M_1]^{\leq \theta}$.
\item $g_p:A_p^0\to A_p^1$ a partial isomorphism.
\item $h_p:A_p^0\cup A_p^1\to\alpha$.
\item $h_p^{-1}(0)\cap A_p^0\subseteq \dom(g_p)$ and $h_p^{-1}(0)\cap A_p^1\subseteq\ran(g_p)$.
\end{enumerate}
We refer to $\beta_p$ as the {\em rank} of  $p$. For $a \in \dom(h_p)$, we refer to $h_p(a)$ as the {\em height} of the element $a$ in the position $p$.
We say that a position $q$  \emph{\bf extends}  the position $p$, and write $q < p$, if
\begin{enumerate}
\item $\beta_q<\beta_p$.
\item $A_q^0\supseteq A_p^0$ and $A_q^1\supseteq A_p^1$.
\item $g_q\restriction A_p^0 =g_p$.
\item If $a\in A_p^0\cup A_p^1$, then $h_q(a)\le h_p(a)$, and if $h_p(a)>0$, then $h_q(a)<h_p(a)$.
\end{enumerate}
\end{definition}

Note that any $(\beta,\theta,\alpha)$-position is also 
a $(\beta', \theta',\alpha')$-position, for any $\beta' \geq \beta$, $\theta' \geq \theta$, and $\alpha' \geq \alpha$.

%The game is described in Figure~\ref{dgame}. 

%\todo{give intuition}

%Let $P_\beta$ be the set of all  $\alpha$-positions %$(\beta,A_0,A_1,g,h)$ for $(M_0,M_1)$. Note that %$|P_\beta|=|\alpha|^{(M_0\cup M_1)^\theta}$. The extension %relation is a partial order in the class $\bigcup_\beta %P_\beta$. 

\begin{definition}\label{definition:game}
Given ordinals $\alpha$ and $\beta$, and a cardinal $\theta$, the game $\DG^\beta_{\theta,\alpha}(\mathfrak{M}_0,\mathfrak{M}_1)$ is played between two players, Adam and Eve. 
The starting position is $(\beta,\emptyset,\emptyset,\emptyset,\emptyset)$. Eve consecutively plays $(\beta,\theta,\alpha)$-positions according to the following rule.  If the current position is $p$ then: 
\begin{itemize}
   \item First Adam chooses $\beta'<\beta_p$,   $B_0\in [M_0]^{\leq \theta}$ and $B_1\in [M_1]^{\leq \theta}$,
    \item Then Eve chooses a position $q<p$ such that $\beta_q=\beta'$,
    $B_0\subseteq A_q^0$ and $B_1\subseteq A_q^1$. 
\end{itemize}
 We declare the next position in the game to be $q$. The player who cannot move \emph{loses} and the opponent \emph{wins}. A player is said to \emph{win the game} if they have a winning strategy in the game. We  write  $A^\beta_{\theta,\alpha}(\mathfrak{M}_0,\mathfrak{M}_1)$ if Adam has a winning strategy, and 
$E^\beta_{\theta,\alpha}(\mathfrak{M}_0,\mathfrak{M}_1)$ if Eve has a winning strategy in the game $\DG^\beta_{\theta,\alpha}(\mathfrak{M}_0,\mathfrak{M}_1)$. 

\end{definition}

The intuition is the following. A position describes a partial information about an unknown isomorphism between $\mathfrak{M}_0$ and $\mathfrak{M}_1$. 
We think of $\beta_p$ as the rank of the position. This rank is chosen by Adam and it should be at most $\beta$. 
Elements of $A_0 \cup A_1$ are assigned heights by Eve and they must be ordinals $<\alpha$. Going from one position $p$ to the next one $p'$, Adam asks questions about the unknown isomorphism by requiring that the next sets $A'_0$ and $A'_1$ contain $B_0$ and $B_1$ respectively. Each time he does this he must decrease the rank of the position. Eve is not required to answer the questions immediately. Instead, she assigns heights to elements of $A_0\cup A_1$. These heights are ordinals $< \alpha$. 
Going from one position to the next one, Eve is required to decrease the heights of the elements that were present in the old position. Once the height of an element  $A_0\cup A_1$ reaches $0$, Eve must match it with an element of height $0$ in the other structure. This matching must be a partial isomorphism, otherwise Eve loses. 
Since the ranks of the positions go down the game must stop after finitely many steps. There may still be elements of $A_0 \cup A_1$ whose heights are not $0$ and Eve is not obliged to match them with elements of the other structure.

Since the game must end in finitely many moves, by the Gale-Stewart theorem \cite{Gale-Stewart} it is determined. 

In Definition~\ref{definition:theta-alpha-logic} we turn our game into a logic such that, roughly speaking, elementary equivalence with respect to the logic corresponds to Eve having a winning strategy in our game.

The main difference to the game of Shelah's \cite{MR2869022} is that our positions have \emph{height} $h_p$. This makes our games in a sense slower than Shelah's games. This will be important in Proposition~\ref{proposition:natural-sum}, where we prove a kind of transitivity property for our game. No such transitivity is known for Shelah's game.

We first observe that if Eve has a winning strategy, she has a \emph{positional} winning strategy, i.e. a strategy which at every move only depends on the position of the game, not on the history how the game evolved into that position. In the below lemma the set $K$ is what can be called a \emph{positional} winning strategy of Eve.

\begin{lemma}\label{lemma-positional}
 Eve has a winning strategy in the game $\DG^\beta_{\theta,\alpha}(\mathfrak{M}_0,\mathfrak{M}_1)$
if and only if there is a set $K$ of $(\beta,\theta,\alpha)$-positions such that:

\begin{enumerate}
\item $(\beta,\emptyset,\emptyset,\emptyset,\emptyset)\in K$.
\item If $p\in K$, then for all $\beta'<\beta_p$, for all $B_0\supseteq A_p^0$ and $B_1\supseteq A_p^1$ such that $B_0\in [M_0]^{\leq \theta}$ and $B_1\in [M_1]^{\leq \theta}$, there is a position $q\in K$ with $q<p$ such that $\beta_q=\beta'$, 
$A_q^0\supseteq B_0$ and  $A_q^1\supseteq B_1$.
\end{enumerate} 
\end{lemma}

\begin{proof}
Let us first suppose such a set $K$ exists. Eve's winning strategy is to play so that the position of the game stays in $K$. Conversely, suppose $\rho$ is a winning strategy for Eve. Let $K$ be the set of all positions arising when Eve plays according to $\rho$ against all possible legal move combinations of Adam, including the starting position $(\beta,\emptyset,\emptyset,\emptyset,\emptyset)$.
This $K$ clearly satisfies the requirements. \end{proof}

Let us make some comments on the effect of the parameters. First, if $\alpha=1$, then the game 
$\DG^{\beta}_{\theta,1}(\mathfrak{M}_0,\mathfrak{M}_1)$ is essentially the Ehrenfeucht-Fra\"{i}ss\'e game of the usual infinitary logic $\mathscr L_{\infty,\theta^+}$ for formulas of quantifier rank $\beta$. We will write 
$\DG^{\beta}_{\theta}(\mathfrak{M}_0,\mathfrak{M}_1)$ for the game $\DG^{\beta}_{\theta,1}(\mathfrak{M}_0,\mathfrak{M}_1)$.
For $1 < n < \omega$, Eve has a winning strategy in $\DG^{n\cdot\beta}_{\theta,n}(\mathfrak{M}_0,\mathfrak{M}_1)$ iff she has one in $\DG^\beta_{\theta}(\mathfrak{M}_0,\mathfrak{M}_1)$.
Therefore, in order to get something new we should allow $\alpha$ to be infinite. The case $\alpha=\omega$ is Shelah's game from \cite{MR2869022}. 

\begin{proposition}\label{proposition:big-alpha} 
If $\beta\le\alpha$, then Eve has a winning strategy in $\DG^\beta_{\theta,\alpha}(\mathfrak{M}_0,\mathfrak{M}_1)$,
for any $\theta$ and any two structures $\mathfrak{M}_0$ and $\mathfrak{M}_1$ in the same vocabulary.
\end{proposition}

\begin{proof}
The winning strategy is to keep playing positions $p$ such that $h_p(a)\geq \beta_p$, for all $a \in A^0_p \cup A^1_p$, 
and $g_p$ is the empty function.  Given such a position $p$, Adam plays some $\beta'<\beta_p$,  
$B_0\in [M_0]^{\leq \theta}$ and $B_1\in [M_1]^{\leq \theta}$. Eve defines the position $q$ by letting: 
$A^0_q=A^0_p \cup B_0$, $A^1_q= A^1_p \cup B_1$, and $h_q$ is the constant function on $A^0_q\cup A^1_q$ taking the value $\beta'$.
Given that the starting position is $(\beta,\emptyset, \emptyset, \emptyset, \emptyset)$ and $\beta \leq \alpha$,
she can clearly do this throughout the game. 
\end{proof}

\begin{proposition}\label{proposition:parameters}
Let $\mathfrak{M}_0$ and $\mathfrak{M}_1$ be two structures in the same vocabulary. If Eve has a winning strategy in $\DG^\beta_{\theta,\alpha}(\mathfrak{M}_0,\mathfrak{M}_1)$ and
$\beta'\le\beta$, $\theta'\le\theta$ and $\alpha'\ge\alpha$, then she has one in the game $\DG^{\beta'}_{\theta',\alpha'}(\mathfrak{M}_0,\mathfrak{M}_1)$.
\end{proposition}
\begin{proof}
We may assume that Eve's winning strategy, say $\sigma$, in  $\DG^\beta_{\theta,\alpha}(\mathfrak{M}_0,\mathfrak{M}_1)$
plays sets $A^0_p$ and $A^1_p$ of the least possible size. 
Namely, suppose we are in a position $p$ in the game $\DG^\beta_{\theta,\alpha}(\mathfrak{M}_0,\mathfrak{M}_1)$
and Adam plays some $\beta' < \beta_p$, and sets $B_0 \in [M_0]^{\leq \theta}$ and $B_1 \in [M_1]^{\leq \theta}$.
Then $\sigma$ responds by playing a position $q$ such that:
\[
|A^0_q| \leq \max (| A^0_p|, | B_0|, \aleph_0),
\]
and similarly for $A^1_q$. Then $\sigma$ is also a winning strategy for Eve in $\DG^{\beta'}_{\theta',\alpha'}(\mathfrak{M}_0,\mathfrak{M}_1)$.
\end{proof}

We now define the infinite version of the game $\DG^\beta_{\theta,\alpha}(\mathfrak{M}_0,\mathfrak{M}_1)$. 

\begin{definition}\label{definition:infinite} 
Given two structures $\mathfrak{M}_0$ and $\mathfrak{M}_1$ in the same vocabulary, the game $\DG_{\theta,\alpha}(\mathfrak{M}_0,\mathfrak{M}_1)$  is like $\DG^\beta_{\theta,\alpha}(\mathfrak{M}_0,\mathfrak{M}_1)$ except that Adam does not make moves in $\beta$. The game lasts $\omega$ moves. 
\end{definition}

Note that $\DG_{\theta,\alpha}(\mathfrak{M}_0,\mathfrak{M}_1)$  is closed for Eve, for if she loses in any position, she has lost in all future positions of the game, too. So by the Gale-Stewart theorem \cite{Gale-Stewart} it is determined.
 
\begin{proposition}\label{proposition:very-big-beta}
Suppose $\mathfrak{M}_0$ and $\mathfrak{M}_1$ are two structures in the same vocabulary. Let $\kappa = (|M_0|+ |M_1|)^\theta$ and let $\alpha <\kappa^+$. 
If Eve has a winning strategy in $\DG^\beta_{\theta,\alpha}(\mathfrak{M}_0,\mathfrak{M}_1)$, for all $\beta <\kappa^+$, then she has a winning strategy in $\DG_{\theta,\alpha}(\mathfrak{M}_0,\mathfrak{M}_1)$. 
\end{proposition}

\begin{proof} By Lemma \ref{lemma-positional}, we know that Eve has a positional winning strategy in the game $\DG^\beta_{\theta,\alpha}(\mathfrak{M}_0,\mathfrak{M}_1)$, for all $\beta$. Let $K_\beta$ be such a strategy. We define a set $K$ as follows.
Let $(A^0,A^1,g,h)\in K$ iff the set
\[
 \{ \beta : \exists \beta' >\beta \; (\beta,A^0,A^1,g,h)\in K_{\beta'}\} 
\]
is unbounded in $\kappa^+$.
We claim that $K$ is a positional winning strategy for Eve in 
$\DG_{\theta,\alpha}(\mathfrak{M}_0,\mathfrak{M}_1)$.
To see this, suppose $p=(A^0_p,A^1_p,g_p,h_p)\in K$ is a position, and let $B_0\in [M_0]^{\leq \theta}$ and $B_1 \in [M_1]^{\leq \theta}$ be given. For each $\beta <\kappa^+$ there are $\beta'$ and $\beta''$ such that $\beta < \beta' < \beta''$
and $(\beta',A^0_p,A^1_p,g_p,h_p) \in K_{\beta''}$. 
Then $(\beta,B_0,B_1)$ is a legitimate move for Adam in 
$\DG^{\beta''}_{\theta,\alpha}(\mathfrak{M}_0,\mathfrak{M}_1)$
in position $(\beta',A^0_p,A^1_p,g_p,h_p)$.
Pick $q_\beta \in K_{\beta''}$ which extends $(\beta',A^0_p,A^1_p,g_p,h_p)$, and such that $\beta_q=\beta$, $B_0 \subseteq A^0_{q_\beta}$ and $B_1 \subseteq A^1_{q_\beta}$. Since there are at most $\kappa$ possibilities for $(A^0_{q_\beta},A^1_{q_\beta},g_{q_\beta},h_{q_\beta})$,
there must be a single quadruple $(A^0,A^1,g,h)$ which appears 
for unboundedly many $\beta < \kappa^+$. It follows that
 $q= (A^0,A^1,g,h)\in K$ and is a legitimate reply to Adam's move $(B_0,B_1)$ in the game $\DG_{\theta,\alpha}(\mathfrak{M}_0,\mathfrak{M}_1)$.
\end{proof}

While the game $\DG^\beta_{\theta,\alpha}(\mathfrak{M}_0,\mathfrak{M}_1)$ is obviously symmetric, the existence of a winning strategy for Eve may not be transitive, as far as we know (see our Open Problem (3) at the end of the paper). In other words, if Eve wins $\DG^\beta_{\theta,\alpha}(\mathfrak{M}_0,\mathfrak{M}_1)$ and $\DG^\beta_{\theta,\alpha}(\mathfrak{M}_1,\mathfrak{M}_2)$, she may conceivably not win $\DG^\beta_{\theta,\alpha}(\mathfrak{M}_0,\mathfrak{M}_2)$. Nevertheless, there is a weak version of transitivity that we now explain. In order to do this we will use natural sums of ordinals, see \cite{Sierpinski}.
Recall that for ordinals $\alpha$ and $\beta$, the {\em natural sum} of $\alpha$ and $\beta$,
denoted by $\alpha \oplus \beta$, is defined by simultaneous induction on $\alpha$ and $\beta$
as the smallest ordinal greater than $\alpha \oplus \gamma$, for all $\gamma < \beta$, 
and $\gamma \oplus \beta$, for all $\gamma <\alpha$. 
Another way to define the natural sum  of two ordinals $\alpha$ and $\beta$
is to use the Cantor normal form: one can find a sequence of ordinals $\gamma_0 > \ldots >\gamma_{n-1}$ 
and two sequences ($k_0,\ldots, k_{n-1})$ and $(j_0,\ldots, j_{n-1})$
of natural numbers (including zero, but satisfying $k_i + j_i > 0$, for all $i$) such that
$\alpha = \omega^{\gamma_0}\cdot k_0 + \cdots +\omega^{\gamma_{n-1}}\cdot k_{n-1}$
and $\beta = \omega^{\gamma_0}\cdot j_0 + \cdots +\omega^{\gamma_{n-1}}\cdot j_{n-1}$
and define
\[
\alpha \oplus \beta = \omega^{\gamma_0}\cdot (k_0+j_0) + \cdots +\omega^{\gamma_{n-1}}\cdot 
(k_{n-1}+j_{n-1}). 
\]

What is important for us is that the natural sum is associative and commutative. It is always greater or equal 
to the usual sum, but it may be strictly greater.

\begin{proposition}\label{proposition:natural-sum} Suppose $\mathfrak{M}_0$, $\mathfrak{M}_1$, 
and $\mathfrak{M}_2$ are structures in the same vocabulary. 
Let $\beta$, $\alpha$ and $\alpha'$ be ordinals and $\theta$ a cardinal. Suppose Eve wins the games 
$\DG^\beta_{\theta,\alpha}(\mathfrak{M}_0,\mathfrak{M}_1)$ and
$\DG^\beta_{\theta,\alpha'}(\mathfrak{M}_1,\mathfrak{M}_2)$.
Then she wins the game
$\DG^\beta_{\theta,\alpha\oplus \alpha'}(\mathfrak{M}_0,\mathfrak{M}_2)$.
\end{proposition}

\begin{proof}
We may assume that $M_0$, $M_1$, and $M_2$ are pairwise disjoint. Let $K$ and $L$ be positional winning strategies for Eve in $\DG^\beta_{\theta,\alpha}(\mathfrak{M}_0,\mathfrak{M}_1)$ and
$\DG^\beta_{\theta,\alpha'}(\mathfrak{M}_1,\mathfrak{M}_2)$, respectively. Let 
\[
M = \{ (p,q)\in K\times L :  \beta_p=\beta_q \mbox{ and } A^1_p\subseteq A^0_q\}.
\]

Given $(p,q)\in M$ we define a position $r=p\circ q$ in 
$\DG^\beta_{\theta,\alpha\oplus \alpha'}(\mathfrak{M}_0,\mathfrak{M}_2)$ as follows.  First, we let $\beta_r=\beta_p$, $A^0_r=A^0_p$, $A^1_r=A^1_q$. 
We now define $h_r$. If $a \in A^0_p$ and $h_p(a)>0$, we let $h_r(a)= h_p(a) \oplus \alpha'$. If $h_p(a)=0$ then $g_p(a)$ is defined. We then let $h_r(a)=h_q(g_p(a))$. If $a\in A^1_q$ the definition is symmetric. Namely, if $h_q(a) >0$ we let $h_r(a)= \alpha \oplus h_q(a)$. If $h_q(a)=0$, then $g_q^{-1}(a)$ is defined. So we let $h_r(a)=h_p(g_q^{-1}(a))$. Finally, we let $g_r = g_q \circ g_p$. By this we mean $g_r(a)$ is defined if both $h_p(a)=0$ and $h_q(g_p(a))=0$ in which case we let 
$g_r(a)= g_q(g_p(a))$. Finally, we let 
\[
K \circ L = \{ p \circ q: (p,q) \in M \}.
\]
We claim that $K\circ L$ is a positional winning strategy for Eve. To see this, suppose $r\in K\circ L$ and fix $(p,q)\in M$
such that $r= p \circ q$. Fix $\beta' < \beta_r$, and 
$B_0 \in [M_0]^{\leq \theta}$ and $B_1\in [M_2]^{\leq \theta}$.
We need to show that there is $r' < r$ with $r'\in K\circ L$
such that $\beta_{r'}=\beta'$, $B_0\subseteq A^0_r$ and $B_1\subseteq A^1_r$. Since $p \in K$ and $K$ is a positional winning strategy for Eve in $\DG^\beta_{\theta,\alpha}(\mathfrak{M}_0,\mathfrak{M}_1)$,
we can find $p' < p$ such that $\beta_{p'}=\beta$, $B_0\subseteq A^0_{p'}$ and $A^0_q \subseteq A^1_{p'}$. 
Now, since $L$ is a positional winning strategy for Eve in 
$\DG^\beta_{\theta,\alpha'}(\mathfrak{M}_1,\mathfrak{M}_2)$
we can find $q' < q$ with $\beta_{q'}=\beta'$ and such that 
$A^1_{p'}\subseteq A^0_{q'}$ and $B_1\subseteq A^1_{q'}$.
Let $r' =p'\circ q'$. It is clear that $r' < r$, $\beta_{r'}=\beta'$, $B_0 \subseteq A^0_{r'}$, and $B_1\subseteq A^1_{r'}$. It follows that $K \circ L$ is a positional winning strategy for Eve in $\DG^\beta_{\theta,\alpha\oplus \alpha'}(\mathfrak{M}_0,\mathfrak{M}_2)$.
\end{proof}

Let us use $\DG^\beta_{\theta,<\alpha}(\mathfrak{M}_0,\mathfrak{M}_1)$ to denote the game in which Eve first chooses $\alpha'<\alpha$ and then the game $\DG^\beta_{\theta,\alpha'}(\mathfrak{M}_0,\mathfrak{M}_1)$
is played. We obtain from Proposition~\ref{proposition:natural-sum} the following transitivity result:

\begin{proposition}\label{proposition:natural-sum1} Suppose $\mathfrak{M}_0$, $\mathfrak{M}_1$, 
and $\mathfrak{M}_2$ are structures in the same vocabulary. 
Let $\beta$ and  $\alpha$ be ordinals, $\alpha$ additively indecomposable (i.e. $\gamma<\alpha$ and $\delta<\alpha$ imply $\gamma+\delta<\alpha$), and $\theta$ a cardinal. Suppose Eve wins the games 
$\DG^\beta_{\theta,<\alpha}(\mathfrak{M}_0,\mathfrak{M}_1)$ and
$\DG^\beta_{\theta,<\alpha}(\mathfrak{M}_1,\mathfrak{M}_2)$.
Then she wins the game
$\DG^\beta_{\theta,<\alpha}(\mathfrak{M}_0,\mathfrak{M}_2)$.
\qed
\end{proposition}

We now consider the question of reducing the ordinal $\alpha$ in the game. More precisely, given $\beta,\alpha$ and $\theta$, we would like to find  $\beta^*$ and $\theta^*$ which are not too big relative to $\beta$ and $\theta$ and such that 
$E^{\beta^*}_{\theta^*,\alpha}(\mathfrak{M}_0,\mathfrak{M}_1)$ implies 
$E^{\beta}_{\theta,\omega}(\mathfrak{M}_0,\mathfrak{M}_1)$.
As a warm-up we first consider infinite games. 

\begin{proposition}\label{proposition:infinite}
Let $\theta$ be a cardinal and let $\alpha <(2^\theta)^+$ be an ordinal. Suppose $\mathfrak{M}_0$ and $\mathfrak{M}_1$ are two structures in the same vocabulary. If Eve wins $\DG_{2^\theta,\alpha}(\mathfrak{M}_0,\mathfrak{M}_1)$ (the big game)  then she wins $\DG_{\theta,\omega}(\mathfrak{M}_0,\mathfrak{M}_1)$ (the small game).
\end{proposition}

\begin{proof}
 The two games are determined. Suppose  that Eve has a winning strategy $\tau$ in the big game  and  a  strategy $\sigma$ of Adam is given  in the small game. 
 Let $\lambda$ be a sufficiently large regular cardinal such that $\mathfrak{M}_0$, $\mathfrak{M}_1$,
 $\sigma$ and $\tau$ belong to $H(\lambda)$, where $H(\lambda)$ denotes the set of all sets whose transitive closure is $<\lambda$. We fix an increasing $\in$-chain $(H_n)_n$ of elementary submodels of $H(\lambda)$ of size $2^\theta$, 
 containing all the relevant parameters and such that $H_n$ is closed under $\theta$-sequence, for all $n$. 
 We can also arrange that $2^\theta \subseteq H_0$. 
 Since $\alpha < (2^{\theta})^+$ and $\alpha \in H_0$, this implies that $\alpha \subseteq H_0$, as well. 
 Let $H= \bigcup_n H_n$.
 We take the role of Adam and play one master play in the big game against $\tau$ as follows: 
 
 We start by playing $(M_0\cap H_0, M_1\cap H_0)$. Let $p_0$ be the response of $\tau$. 
 Since $\tau, M_0,M_1, H_0 \in H_1$, we have that $p_0\in H_1$. We then play $(M_0\cap H_1,M_1\cap H_1)$, and let $p_1$ be the reply of $\tau$. 
 By a similar consideration, we have that  $p_1\in H_2$.
 We continue in this way: at stage $n$ we play 
 $(M_0\cap H_n,M_1\cap H_n)$. Let $(p_n)_n$ be the sequence of $\tau$'s responses.  Let us say $p_n=(A^0_n, A^1_n,g_n,h_n)$, for all $n$. Notice that $g_n \in H_{n+1}$, for all $n$.   
 Let $g= \bigcup_n g_n$. Then $g$ is an isomorphism between $M_0 \cap H$ and $M_1\cap H$.
  Notice that if $a\in (M_0\cup M_1)\cap H_n$ 
 then $a\in \dom(h_k)$, for all $k \geq n$. 
 The sequence of values $\{ h_k(a): k \geq n\}$ is decreasing until it reaches $0$. When $h_k(a)=0$ then $a\in \dom(g_k)$, if $a\in M_0$, and $a\in \ran(g_k)$, if $a\in M_1$. 
For $a\in (M_0\cup M_1)\cap H$, let:
 \[
 h(a) = \min \{ n : h_n(a)=0\}.
 \]
 %Suppose $h_n(a)=0$. Note that if $a\in M_0$  then $a\in \dom(g_n)$, and if $a\in M_1$ then $a\in \ran(g_n)$.  

 We now describe how to defeat Adam's strategy $\sigma$ in the small game. 
 Suppose the first move of $\sigma$ is $(B^0_0,B^1_0)$.
 Since $\sigma \in H_0$ it follows that $(B^0_0,B^1_0)\in H_0$.
 Since $p_0\in H_1$, we can find $C^0_0,C^1_0\in H_1$ of size $\leq \theta$ such that 
 \begin{enumerate}
     \item $C^i_0\in [A^i_0]^{\leq \theta}$, for $i=0,1$,
     \item $B^i_0\subseteq C^i_0$, for $i=0,1$,
     \item $g_0[C^0_0] \subseteq C^1_0$ and $g_0^{-1}[C^1_0]\subseteq C^0_0$.
 \end{enumerate}
To get $C^1_0$ we simply add to $B^1_0$ the images under $g_0$ of elements of $B^0_0$ which are in $\dom(g_0)$, and to get $C^0_0$
we add to $B^0_0$ the preimages of elements of $B^1_0$ in the range of $g_0$. Let $g_0^* = g_0\restriction C^0_0$ and 
 $h_0^*= h_0\restriction (C^0_0\cup C^1_0)$. Since $C^0_0$ and $C^1_0$ are of size $\leq \theta$ and $H_1$ is closed under $\theta$ sequence, we have that $g_0^*, h_0^*\in H_1$. 
Therefore, the tuple $q_0=(C^0_0,C^1_0,g_0^*,h_0^*)$ is an element of $H_1$ and is a legitimate response of Eve to Adam's play $(B^0_0,B^1_0)$ in the small game. 
Suppose the move of $\sigma$ in position $q_0$ 
is $(B^0_1,B^1_1)$. 
Since $\sigma \in H_1$ we also have that $(B^0_1,B^1_1)\in H_1$.
We now repeat the above procedure, this time in $H_2$, to get sets $C^0_1$ and $C^1_1$. Let $g_1^*=g_1\restriction C^0_1$.
Define $h_1^*$ by letting, for $a\in C^0_1\cup C^1_1$,
\[
h_1^*(a)= \begin{cases}
			h(a)-1, & \text{if $h(a)>0$}\\
            0, & \text{if $h(a)=0$}
		 \end{cases}
\]
We get a tuple $q_1=(C^0_1,C^1_1,g_1^*,h_1^*)$ which is a legitimate reply of Eve to the move $(B^0_1,B^1_1)$ of $\sigma$ in the position $q_0$. As before, we argue that $q_1\in H_2$.
 Given the current position $q_n\in H_{n+1}$ we let $(B^0_n,B^1_n)$ be the move of $\sigma$. We then move to the next model $H_{n+2}$ to define Eve's next move $q_{n+1}$. 
 
 At each step we decrease the previous heights of elements in $C^0_n\cup C^1_n$ by $1$, unless they are already $0$, and 
 for those elements we use the function $g_n$ to define the function $g_n^*$. 
By continuing like this for $\omega$ moves we defeat $\sigma$.
%, which contradicts our assumption that $\sigma$ is a winning strategy for Adam in the small game. 
\end{proof}

\subsection{Komjath-Shelah lemma}

We now turn to a partition theorem due to Komjath-Shelah \cite{MR2037074} that we will need in the main theorem of
this section. Given an ordinal $\lambda$ let $\FS(\lambda)$ denote the set of finite decreasing sequences
of ordinals from $\lambda$. So, elements of $\FS(\lambda)$ are sequences $s$ of the form $s(0)s(1)\ldots s(n-1)$
with $\lambda > s(0) > s(1) > \ldots s(n-1)$. Here $n=| s |$ is the length of $s$. 
Note that we can identify $s$ with a finite subset of $\lambda$ given in the decreasing ordering. 
An immediate extension of a sequence $s$ with one ordinal $\gamma$ is denoted by $s\gamma$.
If $\alpha$ is an ordinal, then an $\alpha$-{\em system} in $\lambda$ is a system of ordinals $\{ e(s): s \in \FS(\alpha)\}$
from $\lambda$ such that:
\begin{enumerate}
    \item $e(\gamma) < e(\gamma') <\lambda$, for $\gamma < \gamma' < \alpha$,
    \item $e(s\gamma) < e(s\gamma') < e(s)$,  for  $\gamma < \gamma' < {\rm min}(s)$, if $s\neq \emptyset$.
\end{enumerate}

%for all $s$ that is non-empty: 
%\[
%e(s\gamma) < s(s\gamma') < e(s), \mbox{ for } \gamma < \gamma' < {\rm min}(s). 
%\]
%\begin{enumerate}
%\item $e$ is order preserving
%\item $e$ preserves lengths, i.e. $| e (s) | = | s |$, for all $s$,
%\item ${\rm min}(e(s\gamma)) < {\rm min}(e(s\gamma')) < {\rm min} (e(s))$, for all $s\in \FS(\alpha)$, and $\gamma < \gamma ' < %{\rm min}(s)$.
%\end{enumerate}

If $e$ is an $\alpha$-system in $\lambda$, 
we can define an embedding $\pi_e: \FS(\alpha)\to \FS(\lambda)$
by letting $\pi_e(\emptyset)=\emptyset$, and if $s\in \FS(\alpha)$ is non-empty, 
\[
\pi_e(s)=e(s(0)), e(s(0)s(1)), \ldots e(s).
\]

We will need the following result of Komjath-Shelah \cite{MR2037074}.
    
\begin{theorem}[\cite{MR2037074}]\label{ks}
Assume that $\alpha$ is an ordinal and $\mu$ a cardinal.
Set $\lambda=(|\alpha|^{\mu^{\aleph_0}})^{+}$. Suppose $F:\FS(\lambda^+) \to \mu$. Then there is $e$, an 
$\alpha$-system in $\lambda$, and a function $c:\omega \to \mu$ such that $F(\pi_e(s))= c(| s|)$, for all  $s\in \FS(\alpha)$.
\end{theorem}

\begin{proof}
    
For completeness and for the benefit of the reader we recall the proof of this result from \cite{MR2037074}.  
Suppose $\alpha$, $\mu$, $\lambda$ and $F$ are as in the statement of the lemma. 
For  $c: \omega\to \mu$, we define the ranking function $r_c$ on $\FS(\lambda^+)$ as follows. 
First, if $s\in \FS(\lambda^+)$ and there is $i \leq | s |$ such that $F(s\rest i) \neq c(i)$,
then we let $r_c(s)=-1$. Otherwise we say that $r_c(s)\geq 0$.
We define what it means that $r_c(s)\geq \xi$, by induction on $\xi$.
We say that $r_c(s)\geq \xi$ iff for every $\nu < \xi$ we have
\[
{\rm otp}\{ \gamma < {\rm min}(s): r_c(s\gamma) \geq \nu \} \geq \lambda.
\]
Here we consider that ${\rm min}(\emptyset)=\lambda^+$.
Finally, we let $r_c(s)=\xi$ if $r_c(s)\geq \xi$, but $r_c(s)\geq \xi +1$ is not true. 

\begin{lemma}\label{fixing}
    There is $c:\omega \to \mu$ such that $r_c(\emptyset) \geq \alpha$.
\end{lemma}
\begin{proof}
Let $n$ be an integer. Let $\lambda^{\leq n}$ denote the tree of all sequences of ordinals from $\lambda$
of length at most $n$. Let $T$ be a subtree of $\lambda^{\leq n}$. We say that a node $s\in T$
is a $\lambda$-{\em splitting node} if $| \{ \xi : s\xi \in T\} | = \lambda$.
We say that $T$ is {\em fully splitting} if every non maximal node $s$ of $T$ is $\lambda$-splitting.

\begin{claim}\label{stabilize} Let $n$ be an integer and $T$ a fully splitting subtree of $\lambda^{\leq n}$.
Suppose $I$ is a set of cardinality $|\alpha|^{\mu^{\aleph_0}}$  and $F: T\to I$ is a function. 
Then there is a fully splitting subtree $R$ of $T$ and a function $d: n +1 \to I$
such that $F(s)=d(| s|)$, for all $s\in R$.

\end{claim}

\begin{proof}
    We prove this by induction on $n$. For $n=1$, since $\lambda = (|\alpha|^{\mu^{\aleph_0}})^+$
    this follows by the Pigeon Hole Principle.  
Suppose the claim is true for fully splitting subtrees of $\lambda^{\leq n}$ and let $T$ be a fully splitting subtree of
$\lambda^{\leq n+1}$. Let $L$ be the first level of $T$. 
%= \{ \gamma < \lambda: (\gamma) \in T \}$. 
For each $\gamma \in L$  let 
\[
T_\gamma= \{ s \in \lambda^{\leq n} : \gamma s \in T \}. 
\]
Then $T_\gamma$ is a fully splitting subtree of $\lambda^{\leq n}$. 
By inductive assumption we can find a fully splitting subtree $R_\gamma$ of $T_\gamma$
and $d_\gamma: n +1 \to I$ such that 
\[
F(\gamma s) = d_\gamma(| s|), \mbox{ for all } s \in T_\gamma.
\]
Since $| I ^\omega | = | I |$, we can find a subset $L^*$ of $L$ of cardinality $\lambda$
and a fixed $d^*:n+1 \to I$ such that $d_\gamma=d^*$, for all $\gamma \in L^*$.
By shrinking $L^*$ we can moreover arrange that $F(\gamma)$ has a fixed value say $\rho$, for $\gamma \in L^*$.
Let
\[
R = \{ \gamma s : \gamma \in L^* \mbox{ and } s \in R_\gamma \}
\]
and let $d: n+2 \to I$ be defined by: $d(0)=\rho$ and $d(k+1)=d^*(k)$, for $k\leq n$.
It is clear that $R$ and $d$ are as required. 

 \end{proof}

Let us say that a  $(\lambda,n)$-{\em assignment} is a function $e: \lambda^{\leq n}\to \lambda^+$
such that, for all $s\in \lambda ^{< n}$,
\[
e(s\gamma) < e (s\gamma') < e(s), \mbox{ for all } \gamma < \gamma' < \lambda.
\]
If $e$ is a $(\lambda,n)$-assignment we can define an embedding
$\pi_e: \lambda^{\leq n}\to \FS(\lambda^+)$
by letting
\begin{enumerate}
    \item $\pi_e(\emptyset)=\emptyset$,
    \item $\pi_e(s)=e(s(0)), e(s(0)s(1)), \ldots e(s)$, if $s\in \lambda^{\leq n}$ is non-empty. 
\end{enumerate}

\noindent Given a $(\lambda,n)$-assignment $e$ and $\gamma <\lambda^+$
we say that $e$ is {\em above } $\gamma$ if $e(s) > \gamma$, for all $s\in \lambda^{\leq n}$.
Note that if $e$ is a $(\lambda,n)$-assignment which is above $\gamma + \lambda$ then 
$e$ can be extended to a $(\lambda,n+1)$-assignment which is above $\gamma$.
Finally, observe that $d$ is an $(\lambda,n)$-assignment and $T$ is a fully splitting subtree of $\lambda^{\leq n}$,
then if we let $\sigma$ be the natural isomorphism between $\lambda^{\leq n}$ and $T$
then $e\circ \sigma$ is a $(\lambda,n)$-assignment as well.

Now, back to the proof of Lemma \ref{fixing}. 
Let us assume towards contradiction that $r_c(\emptyset)< \alpha$, for all $c: \omega \to \mu$.
We are going to construct by induction on $n$, ordinals $d(n)$, and for each $c: \omega \to \mu$,
ordinals $\xi(c,n)$ with $-1 \leq \xi (c,n) < \alpha$ such that 
    for every $\gamma < \lambda^+$ there is a $(\lambda,n)$-assignment $e$ which is above $\gamma$ and 
    such that, if $\pi_e$ be the embedding of $\lambda^{\leq n}$ derived from $e$, we have: 
    \begin{enumerate}
        \item $F(\pi_e(s))=d(| s |)$, for all $s\in \lambda ^{\leq n}$,
        \item $r_c(\pi_e(s))=\xi(c,|s|)$, for all $s\in \lambda ^{\leq n}$ and all $c: \omega \to \mu$.
    \end{enumerate}
To start we let $d(0)=F(\emptyset)$ and $\xi(c,0)=r_c(\emptyset)$, for all $c:\omega \to \mu$.
For $\xi <\lambda^+$ we have: $F(\xi)\in \mu$ and $r_c(\xi) < \alpha$, for all $c: \omega \to \mu$. 
By the Pigeon Hole Principle we can find an ordinal $d(1)<\mu$
and, for each $c:\omega \to \mu$, an ordinal $\xi(c,1)$ such that the set $W$ of ordinals $\xi <\lambda^+$
such that: 
\begin{enumerate}
    \item[(a)] $F(\xi)=d(1)$,
    \item[(b)] $r_c(\xi)= \xi(c,1)$, for all $c: \omega \to \mu$.
\end{enumerate}
\noindent is cofinal in $\lambda^+$. Now, for $\gamma < \lambda^+$ we construct
a $(\lambda,1)$-assignment satisfying (1) and (2). Fix  $\gamma < \lambda^+$.
Since $W$ is cofinal in $\lambda^+$ we can find an increasing function $f:\lambda\to W \setminus (\gamma+1)$.
We now define $e:\lambda^{\leq 1}\to \lambda^+$ by letting $e(\emptyset)$ be any ordinal above ${\rm ran}(f)$ 
and letting $e(\eta)=f(\eta)$, for all $\eta < \lambda$. It is clear that $e$ satisfies (1) and (2)
and is above $\gamma$.

Now suppose we have constructed $d(k)$, and $\xi(c,k)$, for all $k \leq n$ and $c:\omega\to \mu$.
For each $\gamma < \lambda^+$ pick a $(\lambda,n)$-assignment $e_\gamma$ above $\gamma + \lambda$
 which witnesses the inductive hypothesis. 
We can extend $e_\gamma$ to a $(\lambda,n+1)$-assignment $e_\gamma^*$ above $\gamma$.
By Claim \ref{stabilize} we can then find a fully splitting subtree $T_\gamma$ of $\lambda^{\leq n+1}$,
an ordinal $d_\gamma(n+1) <\mu$, and for $c:\omega \to \mu$, $-1 \leq \xi_\gamma(c,n+1) <\alpha$
such that, for every $s\in \lambda^{n+1}$, we have 
\begin{itemize}
    \item $F(\pi_{e_\gamma}(s))=d_\gamma(n+1)$, for all  $s\in \lambda^{n+1}$,
    \item $r_c(\pi_{e_\gamma}(s))= \xi_\gamma(c,n+1)$, for $s\in \lambda^{n+1}$ and all $c: \omega \to \mu$.
\end{itemize}
Now, by the Pigeon Hole Principle again, we can find an unbounded subset $W$ of $\lambda^+$,
a fixed ordinal $d(n+1) <\mu$, $-1 \leq \xi(c,n+1) <\alpha$, for all $c:\omega \to \mu$, 
such that $d_\gamma(n+1)=d(n+1)$ and $\xi_\gamma(c,n+1)=\xi(c,n+1)$, for all $\gamma\in W$ and $c:\omega \to \mu$.
Therefore the $(\lambda,n+1)$-assignments $e_\gamma^*$, witness (1) and (2) for the values $d(n+1)$
and $\xi(c,n+1)$, for $c:\omega \to \mu$. This completes the inductive step.  
Once the whole construction is completed we get a function $d:\omega \to \mu$.
However, note that then we have
\[
\xi(d,0) > \xi (d,1) > \xi (d,2) > \ldots 
\]
which is a contradiction. 
\end{proof}

We now return to the proof of Theorem \ref{ks}. Fix $c:\omega\to \mu$ such that $r_c(\emptyset) \geq \alpha$.
We construct an $\alpha$-system $e$ with the additional property that for every $s\in \FS(\alpha)$
of length $n>0$,
\[
r_c(\pi_e(s)) \geq s(n-1).
\]
To show this we have to show that if we are given an $t\in \FS(\lambda^+)$ with $r_c(t)\geq \beta$
then we can choose ordinals $\{ \xi_\gamma : \gamma < \beta \}$
\begin{enumerate}
    \item if $\gamma < \gamma' < \beta$ then $\xi_\gamma < \xi_{\gamma '}< {\rm min}(t)$,
    \item $r_c(t\xi_\gamma)\geq \gamma$, for all $\gamma <\beta$.
\end{enumerate}

To see this let $\delta_\gamma$ be the supremum of the $\lambda$ first ordinals $\xi$
such that $r_c(t\xi)\geq \gamma$. 
Note that, if $\gamma < \gamma'$ then $\delta_\gamma \leq \delta_{\gamma'}$.
 We are going to select by transfinite
recursion the elements $\xi_\gamma < \delta_\gamma$ as needed. 
At step $\gamma$ we have already selected the ordinals $\{\xi_{\gamma'} : \gamma' < \gamma\}$.
Since $\sup(\{ \xi_{\gamma'}: \gamma' < \gamma\}) \leq \delta_\gamma$, $\delta_\gamma$ has cofinality $\lambda$ and $\gamma < \lambda$
we have that $\sup(\{ \xi_{\gamma'}: \gamma' < \gamma\}) < \delta_\gamma$.
Hence we can choose $\xi_\gamma$ as required. 

Now we can define $e(s)$, for $s\in \FS(\alpha)$, by induction on the length of $s$,
and satisfy the requirement that $r_c(\pi_e(s))\geq s(n-1)$, where $n=|s|$.
Recall that $r_c(\pi_e(s))\geq 0$ implies that $F(\pi_e(s\rest i))=c(i)$, for all $i\leq | s|$.
Hence we must have that $F(\pi_e(s))=c(| s |)$, for all $s\in \FS(\alpha)$.
This completes the proof of Theorem \ref{ks}.

\end{proof}

\subsection{The Main Theorem}
We now turn back to the question of reducing the ordinal $\alpha$ in the game. 
Our plan is to combine the ideas of Proposition \ref{proposition:infinite}
and Theorem \ref{ks}.
Suppose $\theta$ is a cardinal, $\alpha$ and $\beta$ are ordinals. 
Let $\mathfrak{M}_0$ and $\mathfrak{M}_1$ be two structures in the same vocabulary. 
We would like to find cardinals $\theta^*$ and $\beta^*$ 
such that if Eve wins the game $\DG^{\beta^*}_{\theta^*,\alpha}(\mathfrak{M}_0, \mathfrak{M}_1)$ (the big game),
then she wins $\DG_{\theta,\omega}^\beta(\mathfrak{M}_0,\mathfrak{M}_1)$ (the small game).
%We are going to set $\theta^*=2^\theta$. 
We define a sequence of ordinals $(\Gamma(\xi): \xi \leq \beta)$ as follows:
\begin{enumerate}
    \item $\Gamma(0)=2$,
    \item $\Gamma(\xi +1) = ((\Gamma(\xi))^{|\alpha|^\theta})^{++}$,
    \item $\Gamma(\nu)=\sup_{\xi<\nu}\Gamma (\xi)$, if $\nu$ is limit. 
\end{enumerate}
%Suppose we have two models $A$ and $B$. We define two games, the Small Game $DG^\gamma_{\theta,\omega+1}(A,B)$, and the Big Game
%$DG^\Gamma_{2^\theta,\alpha}(A,B)$. Here $\theta$ is an arbitrary infinite cardinal and $\alpha$ is an arbitrary ordinal. The ordinals $\gamma$ and $\Gamma$ have the following relationship with each other and also with $\theta$ and $\alpha$:

%\[
%\Gamma_0=2, \Gamma_{\xi+1}=((\Gamma_\xi)^{|\alpha|^\theta})^{++}, \Gamma_\nu=\sup_{\xi<\nu}\Gamma_\xi,  \Gamma=\Gamma_\gamma.
% \]

Note that if $\kappa=\beth_\kappa$ and $\alpha,\theta,\beta<\kappa$, then also $\Gamma(\beta)<\kappa$.

\begin{theorem}\label{theorem:main-thm}
Let $\theta$ be a cardinal,  $\alpha$ and $\beta$ be ordinals with $\alpha < (2^\theta)^+$ and $\beta$ limit. 
Let $(\Gamma(\xi): \xi \leq \beta)$ be defined as above, and let $\Gamma= \Gamma(\beta)$.
 Suppose $\mathfrak{M}_0$ and $\mathfrak{M}_1$ are two structures in the same vocabulary. If Eve wins $\DG^{\Gamma}_{2^\theta,\alpha}(\mathfrak{M}_0,\mathfrak{M}_1)$ (the big game)  then she wins the game $\DG^\beta_{\theta,\omega}(\mathfrak{M}_0,\mathfrak{M}_1)$ (the small game).
% If Player II has a winning strategy in  $DG^\Gamma_{2^\theta,\alpha}(A,B)$, then Player II has a winning strategy  in  $DG^\gamma_{\theta,\omega+1}(A,B)$.
\end{theorem}

\begin{proof} Note that both games are determined. We assume that Adam has a winning strategy $\sigma$
in the small game, and Eve has a winning strategy $\tau$ in the big game, and we derive a contradiction. 
Recall that for an ordinal $\lambda$, $\FS(\lambda)$ denotes the set of descending sequences in $\lambda$,
ordered by extension. %$\FS(\lambda)$ is a tree with no infinite branches. 
%We have a function $\beta\mapsto \beta^*=\Gamma_\beta$ which maps $\FS(\gamma)$ strict order preservingly into $\FS(\Gamma)$. 
If $f,g$ are  ordinal valued functions, we write $f<g$ if 
\begin{enumerate}
    \item ${\rm dom}(g)\subseteq {\rm dom}(f)$
    \item $f(a)\leq g(a)$, for all $a\in {\rm dom}(g)$,
    \item if $a \in {\dom g}$ and  $g(a)\ne 0$, then $f(a)< g(a)$.
\end{enumerate}
Let $\mu$ be a sufficiently big regular cardinal. 
We pick models $H_s\prec H(\mu), s\in \FS(\Gamma)$, such that:
\begin{enumerate}
\item $\mathfrak{M}_0, \mathfrak{M}_1, \sigma, \tau\in H_\emptyset$,
\item $s\in H_s$, for all $s\in \FS(\Gamma)$,
\item   If $\lambda < {\rm min}(s)$, then $H_s\cup\{H_s\}\subseteq H_{s\lambda}$,  
\item $|H_s|=2^\theta$,
\item $H_s^\theta\subseteq H_s$.
\end{enumerate}

For each $s\in \FS(\Gamma)$ we describe a $(\Gamma, 2^\theta,\alpha)$-position $p_s$ for $(\mathfrak{M}_0,\mathfrak{M}_1)$
as follows. We let $p_\emptyset = (\Gamma,\emptyset, \emptyset, \emptyset, \emptyset)$.
If $s$ is non-empty, $p_s$ will be of the form  $p_s=({\rm min}(s), A^0_s,A^1_s,g_s,h_s)$, 
for some $A^0_s \in [M_0]^{2^\theta}, A^1_s\in [M_1]^{2^\theta}$, 
$g_s$ a partial isomorphism between $A^0_s$ and $A^1_s$, and $h_s: A^0_s \cup A^1_s\to \alpha$ 
is the height function for this position.  We require that, for all non-empty $s\in \FS(\Gamma)$:
\begin{enumerate}
   % \item $\beta_{p_s}= {\rm min}(s)$, for $s\neq \emptyset$,
    \item $p_s\in H_s$, 
    \item $p_s$ is a winning position for Eve in the big game,
    \item if $\lambda < {\rm min}(s)$ then $p_{s\lambda} < p_s$,
    \item if $\lambda < {\rm min}(s)$ then $H_s \cap M_0 \subseteq A^0_{s\lambda}$ 
    and $H_s\cap M_1 \subseteq A^1_{s\lambda}$.
\end{enumerate}
%Here we consider that ${\rm min}(\emptyset)=\Gamma$.
Such a family of positions is easy to construct by induction on the length on $s$.
Let $s\in \FS(\Gamma)$, and assume $p_s$ has been chosen and satisfies the above conditions. 
For each $\lambda < \min{s}$, we go to the model $H_{s\lambda}$
and let Adam play $(\lambda, H_s\cap M_0, H_s\cap M_1)$ in position $p_s$.
We then let $p_{s\lambda}$ be the reply of Eve's strategy $\tau$ to this move. 
Since $H_s \cup \{ H_s\} \subseteq H_{s\lambda}$ this move by Adam belongs to $H_{s\lambda}$.
Moreover, since $\tau \in H_{s\lambda}$, we
have that $p_{s\lambda} \in H_{s\lambda}$, and $p_{s\lambda}$ is a winning position for Eve in the big game. 

We now describe how to play against Adam's strategy $\sigma$ in  $\DG^\beta_{\theta,\omega}(\mathfrak{M}_0,\mathfrak{M}_1)$. Suppose $\sigma$ starts by playing some $(\beta_0,B^0_0,B^1_0)$,
with $B^0_0 \in [M_0]^{\theta}$, $B^1_0\in [M_1]^{\theta}$, and $\beta_0 < \beta$.
Let $B^0= B^0_0\cup B^1_0$, and let $I_0$ be the set of all functions $h: B^0\to \alpha$. 
Since $\sigma \in H_\emptyset$ and $H_\emptyset$ is closed under $\theta$-sequences,
we have that $B^0\subseteq H_\emptyset$. Consider some $\lambda < \Gamma$.
Then $p_\lambda$ is of the form $p_\lambda= (\lambda, A^0_\lambda,A^1_\lambda,g_\lambda,h_\lambda)$,
where $H_\emptyset \cap M_0 \subseteq A^0_\lambda$ and $H_\emptyset \cap M_1\subseteq A^1_\lambda$.
Since $\sigma \in H_\emptyset$ we have that $B^0_0$ and $B^1_0\in H_\emptyset$. 
Moreover $H_\emptyset$ is closed under $\theta$-sequences, so $B^0_0, B^1_0\subseteq H_\emptyset$ as well. 
It follows that $B^0 \subseteq {\rm dom}(h_\lambda)$. 
Therefore, for all $s\in \FS(\Gamma)$ which are non-empty, we have that $B^0 \subseteq {\rm dom}(h_s)$.
 We define the function $F_0: \FS(\Gamma)\to I_0$ by: 
\[
F_0(s)= h_{s}\rest B^0.
\]
By Theorem \ref{ks},  there is a $(\Gamma(\beta_0)+1)$-system $e_0$ and a function $c_0:\omega \to I_0$
such that:
\[
F_0(\pi_{e_0}(s)) = c_0(| s|), \mbox{ for all } s\in \FS(\Gamma(\beta_0)+1).
\]
Then $c_0(0)$ is the empty function, and $c_0(n)$ is a total function from $B^0$ to $\alpha$, for $n>0$. 
Moreover,  $c_0(n+1) < c_0(n)$, for all $n>0$, meaning that:
\begin{enumerate}
    \item $c_0(n+1)(a) \leq c_0(n)(a)$, for all $a\in B^0$,
    \item if $c_0(n)(a)\neq 0$ then $c_0(n+1)(a) < c_0(n)(a)$. 
\end{enumerate}
It follows that for every $a\in B^0$, there is $k$ such that $c_0(k)(a)=0$. 
Let $h_0'(a)$ be the least such $k$. This defines the height function on $B^0$.

Let $\lambda_0= e_0(\Gamma(\beta_0))$. Then $p_{\lambda_0}\in H_{\lambda_0}$.
Let
\begin{enumerate}
    \item $C^0_0=B^0_0\cup g_{\lambda_0}^{-1}[B^1_0]$,
    \item $C^1_0=B^1_0\cup g_{\lambda_0}[B^0_0]$,
    \item $g^*_0= g_{\lambda_0}\rest C^0_0$.
\end{enumerate}
Finally, define the function $h^*_0:C^0_0\cup C^1_0 \to \omega$ by:
\[ 
h^*_0(a) = \begin{cases}
    h'_0(a) & \text{ if $a\in B^0_0 \cup B^1_0$,}\\
    0    & \text{otherwise}
\end{cases}
\]
Let $q_0= (\beta_0,C^0_0,C^1_0, g^*_0,h^*_0)$.
We let Eve play $q_0$ as her first move in the small game. 
Note that $q_0$ is a legitimate response to Adam's  move $(\beta_0,B^0_0, B^1_0)$ in this game.

 Suppose in the next move Adam's strategy $\sigma$ responds by playing some $(\beta_1,B^0_1,B^1_1)$.
 Note that $q_0\in H_{\lambda_0}$, and hence $(\beta_1,B^0_1,B^1_1)\in H_{\lambda_0}$, as well. 
We have $\beta_1 < \beta_0$, $B^0_1 \in [M_0]^{\theta}$ and $C^0_0\subseteq B^0_1$,
  $B^1_1\in [M_1]^{\theta}$ and $C^1_0\subseteq B^1_1$. Let $B^1=B^0_1\cup B^1_1$
  and let $I_1$ be the set of all functions $h:B^1\to \alpha$.

  Consider some $\lambda < \Gamma(\beta_0)$.
Then $p_{\lambda_0e_0(\lambda)}$ is of the form 
\[
p_{\lambda_0e_0(\lambda)}= (\lambda, A^0_{\lambda_0e_0(\lambda)},A^1_{\lambda_0 e_0(\lambda)},g_{\lambda_0 e_0(\lambda)},h_{\lambda_0 e_0(\lambda)}),
\] 
where $H_{\lambda_0} \cap M_0 \subseteq A^0_{\lambda_0 e_0(\lambda)}$ 
and $H_{\lambda_0} \cap M_1\subseteq A^1_{\lambda_0 e_0(\lambda)}$.
Since $B^1\in H_{\lambda_0}$ and $H_{\lambda_0}$ is closed under $\theta$-sequences
we have that $B^1\subseteq H_{\lambda_0}$ and hence $B^1 \subseteq {\rm dom}(h_{\lambda_0 e_0(\lambda)})$.
To simplify notation, if $s\in \FS(\Gamma(\beta_0))$ let us write $t_0(s)$ for  $\pi_{e_0}(\beta_0 s )$.
By the above, for all such $s$, we have that $B^1 \subseteq {\rm dom}(h_{t_0(s)})$.

We define the function $F_1:\FS(\Gamma(\beta_0))\to I_1$ by: 
  \[
  F_1(s)= h_{t_0(s)}\rest B^1.
  \]
By Theorem \ref{ks},
there is a $(\Gamma(\beta_1)+1)$-system $e_1$ in $\Gamma(\beta_0)$ and a function $c_1:\omega \to I_1$
such that:
\[
F_1(\pi_{e_1}(s)) = c_1(| s|), \mbox{ for all } s\in \FS(\Gamma(\beta_1)+1).
\]
Then $c_1(n)$ is a function from $B^1$ to $\alpha$, for all $n$. 
Moreover, $c_1(n+1) < c_1(n)$, for all $n$, meaning that:
\begin{enumerate}
    \item $c_1(n+1)(a) \leq c_1(n)(a)$, for all $a\in B^1$,
    \item if $c_1(n)(a)\neq 0$ then $c_1(n+1)(a) < c_1(n)(a)$. 
\end{enumerate}
It follows that for every $a\in B^1$, there is $k$ such that $c_1(k)(a)=0$. 
Let $h_1'(a)$ be the least such $k$. This defines the height function on $B^1$.
The key point is that, for all $n$,
\[
c_1(n)\rest B^0=c_0(n+1).
\]
To see this fix $a\in B^0$ and some $s\in \FS(\Gamma(\beta_1))$ of length $n$.
Then $F_1(s)(a)=h_{t_0(s)}(a)$, and since the length of $t_0(s)$ is $n+1$,
we have 
\[
h_{t_0(s)}(a)=c_0(n+1)(a).
\]
It follows that if $a\in B^0$ and $h'_0(a)\neq 0$, then $h'_1(a)=h'_0(a)-1$.

Let $\lambda_1= e_0(e_1(\Gamma(\beta_1)))$. Then $p_{\lambda_0 \lambda_1}\in H_{\lambda_0\lambda_1}$.
Let
\begin{enumerate}
    \item $C^0_1=B^0_1\cup g_{\lambda_0\lambda_1}^{-1}[B^1_1]$,
    \item $C^1_1=B^1_1\cup g_{\lambda_0\lambda_1}[B^0_1]$,
    \item $g^*_1= g_{\lambda_0\lambda_1}\rest C^0_1$.
\end{enumerate}
Finally, define the function $h^*_1:C^0_1\cup C^1_1 \to \omega$ by:
\[ 
h^*_1(a) = \begin{cases}
    h'_1(a) & \text{ if $a\in B^0_1 \cup B^1_1$,}\\
    0    & \text{otherwise}
\end{cases}
\]
Let $q_1=(\beta_1,C^0_1,C^1_1, g^*_1,h^*_1)$ and note that $q_1$ is a legitimate response by Eve
to the second move played by the strategy $\sigma$ in the small game.
Eve can continue playing in this way indefinitely; referring to the big game in order to decide what to play in the small game. Since Adam has to play a decreasing sequence of ordinals $\beta_0 > \beta_1 > \ldots$, 
the game has to end after finitely many moves, and so Eve wins this run of the game. 
Therefore, $\sigma$ is not a winning strategy for Adam in the small game,  a contradiction.

\end{proof}

%\begin{corollary}
%   $L^{1,\beta}_{\le\theta,\alpha}\le L^{1,\beta}_{\le\theta,\omega}\le L^{1,\beta^*}_{\le 2^\theta,\alpha}.$ At the limit, if $\alpha<\kappa=\beth_\kappa$: $L^{1}_{\kappa,\omega}\equiv L^{1}_{\kappa,\alpha}.$
% \end{corollary}

\section{Examples}

We give some examples on properties that can or cannot be detected by our games. 

%Let $\mathscr A$ be a class of structures in a fixed vocabulary $\tau$.  Given a cardinal $\theta$ and ordinals $\alpha$ and $\beta$, we say that  the game $\DG^\beta_{\theta,\alpha}$ {\em detects membership} in $\mathscr A$, if whenever $\mathfrak{M}_0$ and $\mathfrak{M}_1$ are two $\tau$-structures such that $\mathfrak{M}_0 \in \mathscr A$ while $\mathfrak{M}_1 \notin \mathscr A$,  then Adam has a winning strategy in $\DG^{\beta}_{\theta,\alpha}(\mathfrak{M}_0,\mathfrak{M}_1)$. We now give some examples of properties that can or cannot be detected by our games. 

\subsection{Cardinality}

% Let $\mathscr Q_\theta$ be the class of structures $(M,R)$, $R\subseteq M$, where  $|R|\geq\theta$.

\begin{proposition}\label{size} Suppose $\cof(\theta)>\omega$,
 $\alpha$ is a countable ordinal and $\beta > \alpha$. Let $M_0$ and $M_1$
be two structures in the empty vocabulary such that $|M_0| \geq \theta$
and $| M_1| <\theta$. Then Adam has a winning strategy in $\DG^{\beta}_{\theta,\alpha}(M_0,M_1)$. 
%Mention also, when does Eve have.} $\DG^{\beta}_{\theta,\alpha}(M_0,M_1)$. detects membership in $\mathscr Q_\theta$.
%$Q_\theta$ is definable in $L^1_{\eta\alpha}$ for every $\alpha<\omega_1$ and every $\eta>\theta$.
\end{proposition}

% \todo{drop the phrase "definable in"}
\begin{proof}
%Let $\mathfrak{M}_0=(M_0,R_0)$ and $\mathfrak{M}_0=(M_1,R_1)$ be  such that $\mathfrak{M}_0\in \mathscr Q_\theta$, but $\mathfrak{M}_1\notin \mathscr Q_\theta$. 
%Let $\alpha$ be a countable ordinal.  
We describe a winning strategy for Adam.
%in $\DG^{\alpha+1}_{\theta,\alpha}(\mathfrak{M}_0,\mathfrak{M}_1)$.
We start by picking a subset $U$ of $M_0$ of size $\theta$ and play as our first move $(\alpha,U,\emptyset)$.
Eve then plays some $h_0:U \to \alpha$. Let $Z_0=h_0^{-1}(0)$. 
Eve also has to play an injective function $g_0: Z_0\to M_1$.
 Note that $Z_0$ has to be of size $<\theta$. 
Since ${\rm cof}(\theta)>\omega$ and $\alpha$ is countable, there is an ordinal $\alpha_0 < \alpha$
such that $U_0= h_0^{-1}(\alpha_0)$ has cardinality $\theta$. 
We then play $(\alpha_0,\emptyset,\emptyset)$. 
Eve responds by playing some $h_1:U\to \alpha$ such that $h_1 < h_0$, meaning that for all $a\in U$,
$h_1(a)\leq h_0(a)$, and if $h_0(a)\neq 0$ then $h_1(a)< h_0(a)$. Let $Z_1=h_1^{-1}(0)$. 
Eve also has to play an injective function $g_1: Z_1\to M_1$ such that $g_1$ extends $g_0$.
We find an ordinal $\alpha_1 < \alpha_0$ such that $U_1= h_1^{-1}(\alpha_1)$ has
cardinality $\theta$. We continue like this until at some stage we reach $\alpha_n=0$.
At that point $U_n=Z_n$ and this set has cardinality $\theta$. Since $M_1$ has
cardinality smaller than $\theta$, Eve cannot play an injective function $g_n:Z_n \to M_1$,
so she loses the game. 
\end{proof}

% Suppose $(M,R),(M',R')$ are given and $(M,R)\in Q_\theta$. Suppose  $2\uparrow\DG^{\beta}_{\eta,1}((M,R),(M',R'))$, where $\beta=\alpha\cdot 4$. Let us  prove that $|R'|<\theta$. Suppose the contrary: $|R'|\ge\theta$. Let $X\subseteq R'$ with $|X|=\theta$. We start $\DG^{\beta}_{\eta,\alpha}((M,R),(M',R'))$. Player 2 plays $\alpha_1<\alpha$. Player 1 plays $$\beta_0=\alpha_1\cdot 3+1, A_0=\emptyset, \beta_1=\alpha_1\cdot 3, B_0=X.$$ Player 2 plays $$h_0=\emptyset, g_0=\emptyset, h_1: X\to\alpha_1.$$  For some $\alpha^*<\alpha_1$ the set $h_1^{-1}(\alpha^*)$ is of size $\theta$, since $\cof(\theta)>\omega$.  The corresponding part in $M$ is of size $<\theta$. Now Player 1 wins.
%\end{proof}

%  Here is a sketch of a sentence of $L^1_{\theta^+\alpha}$ defining $Q_\theta$:
% $$\neg\exists\langle x_i:i<\theta\rangle\bigwedge_{h:\theta\to\omega}\bigwedge_{n\in\ran(h)}
% \bigwedge_{h(i)=n}(R(x_i)\wedge\bigwedge\{x_i\ne x_j : h(j)=n, i\ne j\}).
%. $$
% Note that this is of the form
%   $$\neg\exists\langle x_i:i<\theta\rangle\bigwedge_{h:\theta\to\omega}\bigwedge_{n\in\ran(h)}
%. \phi_{h,n}(\langle x_i:h(i)=n\rangle).
%% $$

\subsection{Covering every small set by countably many sets from a fixed  list}
Let $\mathscr C$ be a class of structures in a given vocabulary $\tau$. For a cardinal $\theta$ and ordinals $\alpha$ and $\beta$,
we say that the game $\DG^\beta_{\theta,\alpha}$ {\em detects membership} in $\mathscr C$, if for any two $\tau$-structures 
$\mathfrak{M}$ and $\mathfrak{N}$ such that $\mathfrak{M}\in \mathscr C$ and $\mathfrak{N}\notin \mathscr C$,
Adam has a winning strategy in $\DG^\beta_{\theta,\alpha}(\mathfrak{M}, \mathfrak{N})$.

 Suppose $\theta$ is an infinite cardinal. 
 Let $\mathscr S_\theta$ be the class of structures $(M,R)$, $R\subseteq M^2$, where, for all $a\in M$, 
 \[
 |\{b\in M: R(a,b)\}|<\theta.
 \]
Let $\mathscr T_\theta$ be the class of structures $(M,R)$, $R\subseteq M^2$, where  for all $X\in [M]^{<\theta}$ there is a countable set $Y$ such that 
\[ X\subseteq \{b\in M : \exists a\in Y R(a,b)\}. \]
In this case we say that $Y$ is a {\em cover} of $X$. Let $\mathscr K_\theta=\mathscr S_\theta\cap \mathscr T_\theta$.  This is as in \cite[Definition 3.6]{MR2869022}.
If $(M,R)\in \mathscr K_\theta$, then $R$ codes subsets of size $<\theta$ such that every subset of size $<\theta$ is covered by countably many of them.
The following is straightforward. 

\begin{proposition}
    If $\lambda=\lambda^\omega$, then there is a model in $\mathscr K_\theta$ of cardinality $\lambda$  if and only if $\lambda^{<\theta}=\lambda$.
  \qed  
\end{proposition}

\begin{proposition}
    Suppose $\theta$ is a cardinal with $\cof(\theta)>\omega$, $\alpha$ is a countable ordinal and $\beta > \alpha \cdot 2$.
    Then $\DG^\beta_{\theta,\alpha}$ detects membership in $\mathscr S_\theta$. 
  %  Let $\eta \geq \theta$. Then $\DG^{\alpha+1}_{\theta,\alpha}$ detects membership in  $\mathscr S_\theta$.
\end{proposition} 

%Suppose $\cof(\theta)>\omega$ and $\alpha$ is a countable ordinal.  Let $\eta \geq \theta$. Then $\DG^{\alpha+1}_{\theta,\alpha}$ detects membership in  $\mathscr S_\theta$.
%is definable in $L^1_{\eta\alpha}$ for every $\alpha<\omega_1$ and every $\eta>\theta$.
% \end{lemma}

\begin{proof}
Suppose $R_0$ is a binary relation on $M_0$ and $R_1$ a binary relation on $M_1$. Let $\mathfrak{M}_0=(M_0,R_0)$ and $\mathfrak{M}_1=(M_1,R_1)$.
Suppose that $\mathfrak{M}_0\in \mathscr S_\theta$ and $\mathfrak{M}_1\notin \mathscr S_\theta$.
We describe a winning strategy for Adam in $\DG^\beta_{\theta,\alpha}(\mathfrak{M}_0, \mathfrak{M}_1)$.
Let  $a_1\in M_1$ be such that $\{b \in M_1: R_1(a_1,b) \}$ has size at least $\theta$. 
Adam lets  $\beta_0= \alpha \cdot 2$ and plays $(\beta_0, \emptyset , \{ a_1\})$. Eve must respond by playing some height function $h_0$. 
The partial isomorphism $g_0$ played by Eve is not important. The only important point is that $a_1\in {\rm dom}(h_0).$
In the next move Adam plays $\beta_1=\alpha + h_0(a_1)$. Eve then plays a lower height function $h_1$ and Adam responds by playing $\beta_2= \alpha + h_1(a_1)$.
In finitely many steps we reach some $n$ such that $h_n(a_1)=0$. At that point Eve must have $a_1 \in {\rm ran}(g_n)$. Let $a_0= g^{-1}_n(a_1)$.
We also have that $\beta_n > \alpha$. 
Since $\{ b \in M_0: R_0(a_0,b)\}$ has size $<\theta$, from this point Adam can use the strategy as in the proof of Proposition \ref{size} to win the game. 
\end{proof}

\begin{proposition}
    Suppose $\cof(\theta)>\omega$, $\alpha$ is a countable ordinal and $\beta > \alpha \cdot 2$. 
Then $\DG^{\beta}_{\theta,\alpha}$ detects membership in  $\mathscr T_\theta$.
\end{proposition} 
%$T_\theta$ is definable in $L^1_{\theta\alpha}$ for every $\alpha<\omega_1$.
%\end{proposition}

\begin{proof}
Suppose $\mathfrak{M}_0=(M_0,R_0) \in \mathscr T_\theta$ and $\mathfrak{M}_1=(M_1,R_1)\notin \mathscr T_\theta$,
where $R_0$ is a binary relation on $M_0$ and $R_1$ is a binary relation on $M_1$.
If $A,B\subseteq M_0$ we say that $A$ is a {\em cover} of $B$, if for every $b \in B$ there is $a\in A$ such that $R_0(a,b)$.
We define a similar notion for subsets of $M_1$. 
By our assumption every subset of $M_0$ of cardinality $<\theta$ has a countable cover, but this is not true for $M_1$.
We show that Eve does not have a winning strategy in $\DG^\beta_{\theta,\alpha}(\mathfrak{M}_0, \mathfrak{M}_1)$.
Since this game is determined this means that Adam has a winning strategy. 

Now, suppose towards contradiction that $\sigma$ is a winning strategy for Eve. 
Let $U$ be a subset of $M_1$ of cardinality $<\theta$ that does not have a countable cover. 
Adam sets $\beta_0= \alpha \cdot 2$ and plays as his first move $(\beta_0,\emptyset, U)$.
Eve follows her strategy $\sigma$ and responds by playing a position $(\beta_0,A^0_0,A^1_0,g_0,h_0)$. 
%We may assume that $A^1_0=U$. 
Since $U$ does not have a countable cover and $\alpha$ is countable, 
there is $\alpha_0 < \alpha$ such that $h_0^{-1}(\alpha_0)\cap U$ does not have a countable cover. 
In his next move Adam plays $\beta_1=\alpha + \alpha_0$. 
In her next move following $\sigma$ Eve chooses a lower height function $h_1$. 
Adam then finds an ordinal $\alpha_1 < \alpha_0$ such that $h_1^{-1}(\alpha_1)\cap U$ does not have a countable cover,
and plays $\beta_2=\alpha + \alpha_1$.

By playing like this, in finitely many stages they reach an integer $n$ such that $\alpha_n=0$.
At this stage we have that the set $B_1= h_n^{-1}(0)\cap U$ does not have a countable cover. 
We also have that $B_1 \subseteq {\rm ran}(g_n)$. Let $B_0= g_n^{-1}[B_1]$.
Since $B_0$ is of cardinality $<\theta$ and $\mathfrak{M}_0\in \mathscr T_\theta$,
$B_0$ has a countable cover, say $A_0$. 
Now, for each $a\in A_0$, Adam runs a separate run of the game in which Eve is forced to reveal the image of $a$ in $M_1$.
Namely, in the next stage of the game Adam sets $\beta_n=\alpha$ and plays $(\beta_n,\{ a\}, \emptyset)$.
In her next move by $\sigma$, Eve has to play a height function $h_{n+1}$ with $a\in {\rm dom}(h_{n+1})$.
In his next move Adam plays $\beta_{n+1}=h_{n+1}(a)$.
Eve then plays a lower height function $h_{n+2}$ and Adam responds by playing $\beta_{n+2}=h_{n+2}(a)$. 

In finitely many stages they reach an integer $k$ such that $h_k(a)=0$. 
At that point we must have $a\in {\rm dom}(g_k)$. Let us denote by $f(a)$ the value $g_k(a)$, for this integer $k$.
Therefore we obtain a function $f:A_0 \to M_0$. 
Note that, for each $a\in A_0$ we had a separate run of the game in which Eve follows her strategy $\sigma$,
so we do not know that the function $f$ is injective. Nevertheless the set $A_1=f[A_0]$ is countable. 
We claim that $A_1$ is a cover of $B_1$. Indeed, given $b_1 \in B_1$ let $b_0 \in B_0$ be such that $g_n(b_0)=b_1$.
Since $A_0$ is a cover of $B_0$ there is $a\in A_0$ such that $R_0(a,b_0)$. 
By following the run of the game corresponding to $a$ we must have that $R_1(f(a),b_1)$. 
Therefore, $A_1$ is a countable cover of $B_1$, which is a contradiction.

\end{proof}

%We now have the following. 
\begin{corollary}
Suppose $\theta$ is a cardinal with $\cof(\theta)>\omega$, $\alpha$ is a countable ordinal and $\beta > \alpha \cdot 2$.
Then $\DG^\beta_{\theta,\alpha}$ detects membership in $\mathscr K_\theta$. 
\qed
\end{corollary}

\subsection{The models $\Phi(A)$}

%We let $\eta$ denote the order type of the rationals. 
Suppose $A\subseteq\omega_1$ is non-empty. Let 
\[
 \Phi(A)= A \times ( [0,1) \cap \mathbb Q)  \cup (\omega_1 \setminus A)\times ( (0,1) \cap \mathbb Q).
%\alpha<\omega_1}r_\alpha,
\]

%\[
%r_\alpha=\left\{\begin{array}{ll}
%1+\eta&\mbox{ if $\alpha\in A$}\\
%\eta&\mbox{ if $\alpha\notin A$}
%\end{array}\right.
%\]

We consider the structure $(\Phi(A),<_{\rm lex})$,
where $<_{\rm lex}$ is the lexicographical ordering. 
It is well known that $\Phi(A)\cong\Phi(B)$ iff $A\Delta B$ is non-stationary (see e.g. \cite[p. 57]{MR0462942}).

\begin{proposition}\label{phiA} Let $\alpha$ and $\beta$ be ordinals with $\alpha$ countable and $\beta >\alpha \cdot 3$.
Then  $\Phi(A)\cong\Phi(B)$ iff Eve has a winning strategy %\todo{Mention, when does Eve win.} 
in $\DG^\beta_{\aleph_1,\alpha}(\Phi(A),\Phi(B))$.
\end{proposition}

\begin{proof}
   We only have to prove the implication from right to left. 
   Suppose $\Phi(A)$ and $\Phi(B)$ are not isomorphic. So the set $A \Delta B$ is stationary.
   We describe a winning strategy for Adam in $\DG^\beta_{\aleph_1,\alpha}(\Phi(A),\Phi(B))$.
   Adam first forces Eve to play an isomorphism between an uncountable subset of $\Phi(A)$ and a subset of $\Phi(B)$.
   More precisely, Adam starts by setting $\beta_0= \alpha \cdot 3$ and
   plays $(\alpha \cdot 3, \Phi(A),\Phi(B))$.
   Eve assigns a height function $h_0 : \Phi(A) \cup \Phi(B)\to \alpha$ and plays a partial isomorphism between 
   $h_0^{-1}(0) \cap \Phi(A)$ and $h_0^{-1}(0) \cap \Phi(B)$. 
   Let $\alpha_0$ be the least such that $h_0^{-1}(\alpha_0) \cap \Phi(A)$ is uncountable. 
   In the next move Adam plays $\beta_1= \alpha \cdot 2 + \alpha_0$. Eve then plays some $h_1: \Phi(A)\cup \Phi(B) \to \alpha$,
   which is lower than $h_0$, namely such that, for all $x$,  
   \begin{enumerate}
       \item $h_1(x)\leq h_0(x)$,
       \item if $h_0(x)\neq 0$ then $h_1(x) <h_0(x)$.
   \end{enumerate}
   Let $\alpha_1$ be the least such that $h_1^{-1}(\alpha_1)\cap \Phi(A)$ is uncountable.
   Note that if $\alpha_0 \neq 0$ then $\alpha_1 < \alpha_0$. 
   In this case Adam plays $\alpha_1$ as his next ordinal. He continues playing in this way. 
   After finitely many steps the ordinal $\alpha_n$ reaches $0$. 
   The ordinal $\beta_n$ is  $\alpha \cdot 2$.
   At that stage Eve has played a partial isomorphism $g_n: D \to R$,
   where $D$ is an uncountable subset of $\Phi(A)$ and $R$ is an uncountable subset of $\Phi(B)$. 
   Let $\pi_0$ be the projection to the first coordinate and let $C$ be the set of ordinals $\delta <\omega_1$
   such that:
   \begin{enumerate}
       \item if $g_n(x)=y$ then, $\pi_0(x) < \delta$ iff $\pi_0(y) < \delta$,
       \item $\delta$ is a limit point of $\{ \pi_0(x): x \in D\}$ and $\{ \pi_0(y): y \in R\}$.
   \end{enumerate}
   Then $C$ is a club in $\omega_1$. Since $A \Delta B$ is stationary, there is $\delta \in C \cap A \Delta B$.
   For simplicity let us assume $\delta \in A$ and $\delta \notin B$.
   Note that the set $D_0=\{ x \in D: \pi_0(x) < \delta\}$ has the supremum in $\Phi(A)$,
   namely $(\delta,0)$, while the set $R_0=\{ y \in R : \pi_0(y) <\delta\}$ does not have a supremum in $\Phi(B)$.

    In the next block of moves Adam forces Eve to commit to the image of $(\delta,0)$ in $\Phi(B)$.
    More precisely, Adam plays 
$\beta_{n+1} = \alpha + h_n((\delta,0))$.
    Eve then plays some $h_{n+1}$ which is lower than $h_n$. Adam plays 
    $\beta_{n+2}= \alpha + h_{n+1}((\delta,0))$.
    He continues like this until at some stage $m\geq n$, we have $h_m ((\delta,0))=0$. 
    At this point we must have $(\delta,0)\in {\rm dom}(g_m)$. The ordinal $\beta_m$ then equals $\alpha$.
    Suppose $g_m((\delta,0))=(\xi,q)$. 
    Note that we must have $\xi \geq \delta$, for otherwise $g_m$ could not be a partial isomorphism. 
    Moreover, $(\xi,q)$ has to be above all the elements of $R_0$. Since $R_0$ does not have
    a supremum in $\Phi(B)$, Adam can find some element $y\in \Phi(B)$ which is above all elements of $R_0$,
    but is also below $(\xi,q)$.
    In the final block of moves Adam forces Eve to reveal the preimage of $y$. 
    Namely, he starts this block by playing  $\beta_{m+1}= h_m(y)$. In the next move Eve plays the height function $h_{m+1}$
    which is lower than $h_m$. Adam plays $\beta_{m+2}=h_{m+1}(y)$, etc. 
    After finitely many steps we must have $h_k(y)=0$, so there must be some $x\in {\rm dom}(g_k)$
    such that $g_k(x)=y$. However, we would have $x <_{\rm lex} (\delta,0)$, so there is some element $x'\in D_0$
    with $x <_{\rm lex} x'$, but then $g_k(x') \in R_0$ and also $g_k(x') >_{\rm lex} y$, which means
    that at this stage the function $g_k$ played by Eve could not be a partial isomorphism, so she loses the game.

\end{proof}
% What is $\beta$? For every countable $\alpha$ there is $\beta$ such that the above holds.

\subsection{The irrationals vs. the reals}

We consider the two structures $\mathbb I = \mathbb R \setminus \mathbb Q$ and $\mathbb R$ and  with the usual ordering. As dense linear orders without endpoints, these structures are $\mathscr L_{\infty,\omega}$-equivalent.

\begin{proposition} Let $\alpha$ and $\beta$ be ordinals with $\alpha$ countable and $\beta > \alpha \cdot 2$. 
Then Adam %
%\todo{Mention, when does Eve win.} 
has a winning strategy in $\DG^{\beta}_{2^\omega,\alpha}(\mathbb I, \mathbb R)$.
%\setminus\oQ,<),(\oR<))$.
\end{proposition}

\begin{proof}
For simplicity, we consider that $\mathbb I$ and $\mathbb R$ are disjoint, 
and describe a winning strategy for Adam. He start by setting $\beta_0 = \alpha \cdot 2$
and playing $(\beta_0,\mathbb I, \mathbb R)$.
Eve then plays a height function  $h_0:\mathbb I \cup \mathbb R \to \alpha$ 
and a partial isomorphism from $h_0^{-1}(0)\cap \mathbb I$ to $h_0^{-1}(0)\cap \mathbb R$.
Since $\alpha$ is countable there is an ordinal $\alpha_0 < \alpha$ such 
that $h_0^{-1}(\alpha_0) \cap \mathbb I$ is non meager. 
In the next move Adam plays $\beta_1=\alpha + \alpha_0$. Eve plays a lower height function $h_1$
and extends the partial isomorphism $g_0$ to $g_1$.
Again, there is $\alpha_1< \alpha_0$ such that $h_1^{-1}(\alpha_1)$ is non meager. 
Adam then plays $\beta_2=\alpha + \alpha_1$. 
He continues like this until at some stage $\alpha_n=0$.
At that point the set $D={\rm dom}(g_n)$ must be non meager, and hence dense in some open interval $J$.
Adam fixes a rational $q\in J$. Consider the sets $R_0=g_n [D \cap (-\infty, q)]$ and $R_1=g_n[D \cap (q,\infty)]$.
Since, for every $x\in R_0$ and $y\in R_1$ we have $x < y$, there must be a real $r$
such that $x < r < y$, for every $x\in R_0$ and $y\in R_1$. 
Adam then plays $\beta_{n+1}= h_n(r)$. Eve then plays a lower height function $h_{n+1}$
and Adam responds by playing $\beta_{n+2}= h_{n+1}(r)$. He continues in this way until we reach 
$m$ such that $h_m(r)=0$. At that point the real $r$ must be in the range of $g_m$,
but its preimage could only be $q$, which is not in $\mathbb I$.
In other words, Eve cannot play $g_m$ which is a partial isomorphism, and so Adam wins the game.

%and a partial isomorphism $g_0$ between $h_0^{-1}(0)Look at $h^{-1}(n)$, $n<\omega$. By Baire Category Theorem, this set is dense on an interval $I$ for some $n$. Now I waits for $n$ moves. Then he picks a rational $r$ which is on the interval $I$. Let us look at the images of elements of $h^{-1}(n)$. Since $h^{-1}(n)$ is dense on $I$, there is a real $s$ which corresponds in the order of $\oR$ to $r$. Now I plays this element of $\oR$ and wins.
\end{proof}

\subsection{The reals vs. the reals minus a point}\label{4.5}

We now present two models $\mathfrak M_0$ and $\mathfrak M_1$ such that, for $\beta= \omega_1+1$ and $\theta=2^{\aleph_0}$,
Eve has a winning strategy in  $\DG^\beta_{\theta,\alpha}(M_0,M_1)$ if $\alpha > \omega$ and Adam has a winning strategy if $\alpha =\omega$.  
Thus the models $M_0$ and $M_1$ are equivalent in our sense but not in Shelah's sense, which corresponds to the parameter setting $\alpha =\omega$.
In \cite[p. 57]{MR0462942} the models $\mathfrak{M}_0=(\oR,<)$ and $\mathfrak{M}_1=(\oR\setminus\{0\},<)$ were considered and it was observed that these are non-isomorphic but can be made isomorphic by adding a Cohen-real. 
In particular, the models are partially isomorphic.
We will consider that the domains $M_0$ and $M_1$ of $\mathfrak M_0$ and $\mathfrak M_1$ are disjoint. 
This can be easily achieved by identifying $M_0$ with $\mathbb R \times \{ 0 \}$
and $M_1$ with $(\mathbb R \setminus \{ 0 \})\times \{ 1 \}$ and copying the ordering accordingly. 

\begin{proposition}\label{minus a point1} Let $\mathfrak{M}_0$ and $\mathfrak{M}_1$ be the structures described above. 
    Let $\theta=2^{\aleph_0}$ and $\beta= \omega_1+1$. Then we have the following. 
    \begin{enumerate}
        \item Eve has a winning strategy in $\DG^\beta_{\theta,\omega+1}(\mathfrak{M}_0,\mathfrak{M}_1)$. 
        \item Adam has a winning strategy in $\DG^\beta_{\theta,\omega}(\mathfrak{M}_0,\mathfrak{M}_1)$. 
    \end{enumerate}
\end{proposition}
%\begin{proposition}\label{rmz}Let\footnote{We define $M_0$ and $M_1$  as cartesian products in order to make them disjoint. Of course, this is unessential but helps in notation.} $M_0=(\oR\times\{0\},<)$ and $M_1=(\oR\times\{1\}\setminus\{(0,1)\},<)$ with the  linear order $(a,b)<(a',b)\iff a<a'$.
%Let $\beta=\omega_1+1$ and $\theta=2^{\aleph_0}$.
%\begin{enumerate}
%\item Player 2 wins\todo{Has a winning strategy.} $\DG^\beta_{\theta,\omega+2}(M_0,M_1)$ 
%\item Player 1 wins $\DG^\beta_{\theta,\omega+1}(M_0,M_1)$ 
%\end{enumerate}
%\end{proposition}

%\todo{drop the $\times\{0\}$ part}
\begin{proof} (1):
We describe a winning strategy for Eve. 
We may assume that Adam starts by letting $\beta_0=\omega_1$ and playing $(\beta_0, M_0,M_1)$. 
Eve responds by giving all elements of $M_0 \cup M_1$ height $\omega$.
%The winning strategy of player II is based on her starting with  $\alpha=\omega+1$. Next player I chooses $\beta_0<\beta=\omega_1+1$. W.l.o.g.  $\beta_0=\omega_1$. In addition player I plays an element of $[M_0]^{\le\theta}$ and an element of $[M_1]^{\le\theta}$. W.l.o.g. he plays the entire sets $M_0$ and $M_1$ since, after all, $|M_0\cup M_1|=\theta$. Player II gives now all elements in $M_0\cup M_1$ height $\omega$. Then we move to the second round of the game. 
Now Adam plays $\beta_1<\beta_0=\omega_1$. If $\beta_1=0$, Eve  gives all elements height $1$, and wins the game immediately, as Adam cannot move any more. 
Let us therefore assume $\beta_1\ne 0$. Now Eve has to give $h(x)<\omega$ for each $x\in M_0\cup M_1$. Before defining the function $h$ we make some preparations.

%We use a method due to C. Karp \cite[page 411]{MR0209132}.  
Let $\delta=\omega^{\beta_1}$,  and let $L$ be $\delta\times(1+\omega^*)$ with the lexicographical order.
We use the important fact of infinitary logic, proved by C. Karp \cite[page 411]{MR0209132}, that in the  game $\DG^{\beta_0}_{1,1}(\delta,\mbox{\it L})$, which is just the usual Ehrenfeucht-Fra\"{i}ss\'e game for formulas of quantifier rank $\beta_0$, Eve has a winning strategy. The winning strategy of Eve, nicely presented in \cite{MR0209132}, is based on Eve maintaining finite unions of disjoint intervals of length $\omega^{\beta_n}$ in $\delta$ and isomorphic copies in $L$. 

Let $(a_\xi :\xi < \delta)$ be a strictly increasing continuous sequence  of elements of $M_0$ such that $\sup_\xi a_\xi=\infty$.  
\begin{center}
\begin{tabular}{cccr}
\rll\rll$a_0$\rl$a_1$\rll\rll$a_\xi$\rll\rll&$(\xi<\delta)$\\
\end{tabular}
\end{center}Let $(b_\xi: \xi < \delta)$ be a strictly increasing  sequence of length $\delta$ of elements of $M_1$ converging to $0$,
and $(c_n: n < \omega)$ a strictly decreasing  sequence of  elements of $M_1$ converging to $0$. 
 %For each $n$ let $(b^n_\xi)$ be a strictly increasing continuous sequence of length $\delta$ in 
%$[b_n,b_{n+1})$ with $b^n_\xi=b_n$ and $\sup_\xi b^n_\xi=b_{n+1}$. 
For each $n$, let $(c^n_\xi: \xi < \delta)$ be a strictly increasing sequence of length $\delta$ of elements of $M_1$ 
with $c^n_0=c_n$, $\sup_\xi c^{n+1}_\xi=c_n$ and $\sup_\xi c^{0}_\xi=\infty$.
\begin{center}
\begin{tabular}{c}
\hspace{5mm}$\raisebox{-2pt}{0}$\hspace{77mm}$\empty$\\
$\underbrace{\rl b_0 \rl b_1 \rl b_\xi \rll}_{(-\infty,0)}$
$\bullet$ $\underbrace{\rll}_{\ldots}$
$\underbrace{c^n_0 \rl c^n_1\rll c^n_\xi\rll}_{[c_n,c_{n-1})}$
$\underbrace{c^1_0 \rl c^1_1\rll c^1_\xi\rll}_{[c_1,c_0)}$
$\underbrace{c^0_0 \rl c^0_1\rll c^0_\xi\rl}_{[c_0,\infty)}$\\
\end{tabular}
\end{center}%in 
%$[c_{n+1},c_{n})$ with infimum $c_{n+1}$. The family of intervals
%$$\{(-\infty,a_0) \}\cup \{[a_\xi,a_{\xi+1}):\xi<\delta\}
%$$ covering $M_0$, together with the family of intervals
%$$\{(-\infty,b_0^0)\}\cup\{[b^n_\xi,b^n_{\xi+1}):n<\omega,\xi<\delta\}$$
%$$\cup\{[c^{n}_{\xi+1},c^n_\xi):n<\omega,\xi<\delta\}\cup\{[c^0_0,\infty)\}$$
%covering $M_1$, is  countable. 
The points $a_\xi$ divide the linear order $M_0$ into  disjoint intervals $(-\infty, a_0)$ and $[a_\xi,a_{\xi+1})$. 
Let us enumerate these intervals as $\{ I_n: n < \omega \}$. 
The points $b_\xi$ and $c^n_\xi$ divide the linear order  $M_1$ into  disjoint intervals  $(-\infty,b_0)$, $[b_n,b_{n+1})$ and 
$[c^n_\xi,c^n_{\xi+1})$, for $\xi < \delta$ and $n<\omega$. Let us enumerate these intervals as $\{ J_n : n < \omega \}$. 
We let $h(x)$ denote the unique $n$ such that $x\in I_n\cup J_n$. 
Eve plays this function $h$ as the height function on this second round.
 The rest of the game proceeds as follows: Adam plays a descending sequence of elements of $\beta_1$. Eve decreases the heights of all elements of $M_0 \cup M_1$ in each step by one until they become zero. 
Whenever the height of an element $a$ becomes zero, Eve extends the partial isomorphism she gave in the previous round so that $a$ is in the domain or in the range of the mapping. 
Each interval $I_n$ corresponds in a canonical way to an element of $\delta$. Respectively, each interval $J_n$ corresponds to an element of $L$. This reduces the game 
$\DG^{\beta_0}_{\theta,\omega+1}(\mathfrak{M}_0,\mathfrak{M}_1)$ to the game 
$\DG^{\beta_0}_{1,1}(\delta,L)$. In the latter game Eve has a winning strategy, as noted above.
This winning strategy translates canonically to a winning strategy for Eve in the former game.

\medskip

(2): This is similar to the proof of (2) in the next section. 
%(For details, see the proof in the next section.)
%The sets played by Player 2 in $M_1$ has to stay away from 0. If a dense set is played, we may consider the entire interval having been played. May assume sets of a fixed height are closed. Can code such sets by countable sets.
\end{proof}

\subsection{The trees $\kappa^{\le \omega}$ and $\kappa^{\le \omega}\setminus\{b\}$}\label{4.6}
Fix a cardinal $\kappa$. Let $\mathfrak{M}_0$ be the tree $\kappa^{\le\omega}$, where  $s\le t$ iff $s$ is an initial segment of $t$.
Fix some $b\in \kappa^\omega$ and let $\mathfrak{M}_1$ be $\kappa^{\leq \omega} \setminus \{ b \}$ with the same ordering. 
We will again consider that the universes of $\mathfrak{M}_0$ and $\mathfrak{M}_1$ are disjoint. 
\medskip

\begin{proposition}\label{minus a point} Let $\mathfrak{M}_0$ and $\mathfrak{M}_1$ be the structures described above. 
    Let $\theta=\kappa^\omega$ and $\beta= \kappa^++1$. Then we have the following. 
    \begin{enumerate}
        \item Eve has a winning strategy in $\DG^\beta_{\theta,\omega+1}(\mathfrak{M}_0,\mathfrak{M}_1)$. 
        \item Adam has a winning strategy in $\DG^\beta_{\theta,\omega}(\mathfrak{M}_0,\mathfrak{M}_1)$. 
    \end{enumerate}
\end{proposition}
%\begin{enumerate}
%\item For $\beta \le \kappa^+ +1$, Player II has a winning strategy in  $\DG^\beta_{{\kappa}^\omega,\omega+2}(A,B)$. 
% 
%\item  For $\beta\ge \kappa^+ +1$ Player I wins $\DG^\beta_{{\kappa}^\omega,\omega+1}(A,B)$. 
%
%\end{enumerate} \end{proposition}

\begin{proof} (1): This is similar to the proof of (1) in Proposition~\ref{minus a point1}.
\medskip

(2): %For simplicity, let $<^*$ be a well-order of  $\kappa^{\le\omega}$. 
For any $s\in\kappa^{<\omega}$, let $N(s)=\{s'\in\kappa^{\le\omega}:s\le s'\}$.
We call $N(s)$ the {\em cone} of $s$. We put a topology on $\kappa^{\leq \omega}$ by letting all elements of $\kappa^{<\omega}$ be isolated. 
If $f \in \kappa^\omega$ we let the neighbourhoods of $f$ be the cones $N(f\rest n)$, for $n<\omega$.
If $s,t\in \kappa^{\leq \omega}$, we write $s \wedge t$ for their meet i.e. the largest common initial segment of $s$ and $t$.
We also write $\Delta(s,t)$ for the length of $s\wedge t$. 
Before we start, let us observe that if $(\gamma,A,B,g,h)$ is a winning position for Eve, and $\gamma \geq \omega$,
then $g$ must preserve the length of nodes, and if $s,t\in {\rm dom}(g)$ then 
\[ \Delta(g(s),g(t))=\Delta(s,t).
\]
Suppose now that Eve has a winning strategy $\tau$ in $\DG^\beta_{\theta,\omega}(\mathfrak{M}_0,\mathfrak{M}_1)$. 
Adam challenges $\tau$ as follows. He lets $\beta_0=\kappa^+$ and plays $(\beta_0,M_0,M_1)$ as his first move. 
In future moves he only has to specify a decreasing sequence of ordinals below $\kappa^+$.
Suppose $\tau$ responds by playing a position $(\beta_0, M_0,M_1,g_0,h_0)$.
Let $A_n=h_0^{-1}(n)\cap M_0$ and $B_n=h_0^{-1}(n)\cap M_1$, for all $n$. 
Then $g_0$ has to be a partial isomorphism between $A_0$ and $B_0$. 

Observe that the missing branch $b$ cannot be in the closure $\bar{B}_0$ of $B_0$.
Otherwise we  find a sequence $(b_n)_n$ of nodes in $B_0$ converging to $b$ such that 
\[
\Delta(b_0,b_1) < \Delta(b_1,b_2) < \Delta(b_2,b_3) < \ldots 
\]
Let $a_n=g_0^{-1}(b_n)$ and observe that the sequence $(a_n)_n$ must be convergent as well.
Let $a$ be the limit, and let $n=h_0(a)$. Adam then plays $\beta_1=\omega+ n$ and in $n$ steps forces Eve to decide the image of $a$ in $M_1$,
while ensuring that $\beta_{n+1} \geq\omega$. But then the only possible value for $g_{n+1}(a)$ would be $b$, which is not in $M_1$, a contradiction. 

Let $s_0$ be the shortest initial segment of $b$ such that $N(s_0)\cap B_0= \emptyset$. 
Let $m_0$ be such that $s_0\in B_{m_0}$, i.e. $m_0=h_0(s_0)$. 
For each ordinal $\gamma$ with $\omega < \gamma < \kappa^+$, there is a winning position for Eve
of the form $(\gamma,M_0,M_1,g_1^\gamma,h_1^\gamma)$ in which Eve is obliged to decide the preimage of $s_0$ in $M_0$, i.e. such that $s_0 \in {\rm ran}(g_1^\gamma)$.
Indeed, for each $\gamma$, Adam can play $\gamma + m_0$ in the position $(\beta_0, M_0,M_1,g_0,h_0)$ as his next move and then keep decreasing this ordinal by $1$ until the height of $s_0$
reaches $0$ at which point Eve is obliged to decide the preimage of $s_0$. Let $t_0^\gamma$ be such that $g_1^\gamma(t_0^\gamma)=s_0$.
Let $C^\gamma_1={\rm ran}(g^\gamma_1)$. If $\gamma \geq \omega$, by a similar argument as above we must have that $b\notin \bar{C}^\gamma_1$.
Let $s^\gamma_1$ be an initial segment of $b$ such that $N(s^\gamma_1)\cap C^\gamma_1=\emptyset$.
Now, by the Pigeon Hole Principle there are fixed node $t_0,s_1\in \kappa^{<\omega}$ and an unbounded subset $X_1$ of $\kappa^+$
such that $t_0^\gamma=t_0$, and $s^\gamma_1=s_1$, for all $\gamma \in X_1$. 
%Let $C^\gamma_1 = {\rm ran}(g^\gamma_1)$.
%Now, by a similar argument as above, Adam observes that if $\gamma\in X_0 \setminus \omega$ then $b$ cannot be in $\bar{C}^\gamma_1$, 
%otherwise he could win the game from the position $(\gamma,M_0,M_1,g_1^\gamma,h_1^\gamma)$.
%Let $s^\gamma_1$ be an initial segment of $b$, extending $s_0$, such that $N(s^\gamma_1)\cap C^\gamma_1=\emptyset$.

Let $m_1$ be such that $s_1 \in B_{m_1}$, i.e. $m_1=h_0(s_1)$. 
As above, we argue that, for all $\gamma < \kappa^+$, there is 
a winning position for Eve of the form $(\gamma, M_0,M_1,g_2^\gamma,h_2^\gamma)$ such that $s_1 \in {\ran}(g_2^\gamma)$
and which extends a position of the form $(\gamma',M_0,M_1,g_1^{\gamma '},h_1^{\gamma '})$, for some $\gamma' \in X_1$ with $\gamma < \gamma'$.
Namely, Adam picks some $\gamma' \in X_1$ with $\gamma + m_1 \leq \gamma'$. Since $(\gamma',M_0,M_1,g_1^{\gamma '},h_1^{\gamma '})$ is a winning
position for Eve so is $(\gamma + m_1,M_0,M_1,g_1^{\gamma '},h_1^{\gamma '})$. 
Adam then plays $m_1$ moves from this position
each time decreasing the height of the position by $1$ until the height of $s_1$ reaches $0$ at which point Eve is obliged to decide the preimage of $s_1$.
Now, given such a position $(\gamma, M_0,M_1,g_2^\gamma,h_2^\gamma)$
let $t_1^\gamma$ be such that $g_2^\gamma(t_1^\gamma)=s_1$. Let $C^\gamma_2={\rm ran}(g^\gamma_2)$.
As before, if $\gamma \geq \omega$ we must have $b \notin \bar{C}^\gamma_2$. Let $s^\gamma_2$ be an initial segment of $b$ such that
$N(s^\gamma_2)\cap C^\gamma_2=\emptyset$. By the Pigeon Hole Principle again, there is an unbounded subset $X_2$ of $\kappa^+$ 
and fixed nodes $t_1,s_2 \in \kappa^{<\omega}$ such that $t_1^\gamma=t_1$, and $s^\gamma_2=s_2$, for all $\gamma \in X_2$. 

By repeating this process we can find  increasing sequences $(s_n)_n$ and $(t_n)_n$ of nodes in $\kappa^{<\omega}$, unbounded subsets $X_n$ of $\kappa^+$,
and, for each $\gamma \in X_n$, a winning position for Eve $p_n^\gamma=(\gamma,M_0,M_1,g^\gamma_n,h^\gamma_n)$ such that, letting $C^\gamma_n={\rm ran}(g^\gamma_n)$, for all $n$
we have:
\begin{enumerate}
    \item  $N(s_n)\cap C^\gamma_n=\emptyset$,
    \item  $s_n$ are initial segments of $b$,
    \item if $\gamma \in X_{n+1}$ then $t_n \in {\rm dom}(g^\gamma_{n+1})$ and $g^\gamma_{n+1}(t_n)=s_n$,
    \item if $\gamma \in X_{n+1}$ there is $\gamma' \in X_n$ such that $\gamma < \gamma'$ and $p_{n+1}^\gamma$ extends $p_n^{\gamma '}$. 
    %$(\gamma,M_0,M_1,g^\gamma_{n+1},h^\gamma_{n+1})$ extends $(\gamma', M_0, M_1, g^{\gamma '}_n,h^{\gamma '}_n)$.
\end{enumerate}

%such that the $s_n$ are initial segments of $b$, and for each $n$ and $\gamma \in X_{n+1}$, 
%there is a winning position for Eve $(\gamma,M_0,M_1,g^\gamma_{n+1},h^\gamma_{n+1})$ such that $t_n \in {\rm dom}(g^\gamma_{n+1})$ 
%and $g^\gamma_{n+1}(t_n)=s_n$.  
Let $a= \bigcup_n t_n$ and let $k=h_0(a)$.  Pick some $\gamma \in X_{k+1}$ with $\gamma \geq \omega$ and let $\gamma' \in X_k$ be such that $\gamma < \gamma'$
and $p^\gamma_{k+1}$ extends $p^{\gamma '}_k$. Notice that $a\in {\rm dom}(g^{\gamma '}_k)$. % and that $g^\gamma_{k+1}$ extends $g^{\gamma '}_k$.
Since $t_k\leq a$ and $g^\gamma_{k+1}(t_k)=s_k$ we must have that $s_k \leq g^\gamma_{k+1}(a)$. However, $g^\gamma_{k+1}$ extends $g^{\gamma '}_k$
and hence $s_k\leq g^{\gamma '}_k(a)$, which contradicts the fact that $N(s_k)\cap C^{\gamma '}_k = \emptyset$.
\end{proof}

\subsection{Prikry forcing and the game $\DG^\beta_{<\kappa,\alpha}$}

Elementary equivalence in infinitary logic is closely related to isomorphism in a forcing extension. This is particularly evident in the case of $L_{\infty\omega}$ where elementary equivalence is equivalent to isomorphism in \emph{some} forcing extension (\cite{MR485316}). For stronger infinitary logics, we have to limit ourselves to special classes of forcings and we do not in general get an equivalence as in the case of $L_{\infty\omega}$ (\cite{MR0462942}). Given two structures $\mathfrak{M}_0$ and $\mathfrak{M}_1$ in the same vocabulary, we write $\DG^\beta_{<\kappa,\alpha}(\mathfrak{M}_0,\mathfrak{M}_1)$ 
for the variation of the game $\DG^\beta_{\theta,\alpha}(\mathfrak{M}_0,\mathfrak{M}_1)$  in which the two players play subsets of $M_0$ and $M_1$ of size $<\kappa$
instead of subsets of size $\leq \theta$.
We demonstrate below a connection between the game $\DG^\beta_{<\kappa,\alpha}$ and Prikry-forcing, introduced in \cite{MR262075}. 

Let $\kappa$ be a measurable cardinal and let $\U$ be a normal ultrafilter on $\kappa$. 
We recall the definition of Prikry forcing $\mathbb P_\U$.
Conditions in $\mathbb P_\U$ are pairs $p=(s_p,Y_p)$, where $s_p$ is a finite increasing sequence of elements of  $\kappa$ and  $Y_p$ is an element  of $\U$ with $\max(s_p)< \min(Y_p)$. 
We will simultaneously think of $s_p$ either as a finite increasing sequence and as a finite set which is the range of the sequence.
We say that a condition $q$ extends $p$ and write $q\leq p$ if:
\begin{enumerate}
    \item $s_p$ is an initial segment of $s_q$, 
    \item $Y_q \subseteq Y_p$,
    \item $s_q \setminus s_p \subseteq Y_p$. 
\end{enumerate}
We say that a condition $q$ is a {\em pure} extension of a condition $p$ if $q\leq p$ and $s_q=s_p$.
Given a generic filter $G$ we let $g=\bigcup \{ s_p : p \in G\}$. Then $g$ is a cofinal $\omega$-sequence in $\kappa$. 
It is well known that $\mathbb P_\U$ has the $\kappa^+$-c.c. and does not add new bounded subsets of $\kappa$. 
Moreover if $p\in \mathbb P_\U$ and $\phi$ is a sentence in the forcing language, there is a pure extension $q$ of $p$ which decides $\phi$ (see e.g. \cite[Lemma 21.11]{MR1940513}).
We refer to this fact as the Prikry lemma.

\begin{proposition}\label{4.10} Let $\mathfrak{M}_0$ and $\mathfrak{M}_1$ be two structures in the same vocabulary $\tau$. 
Suppose $\U$ is a normal ultrafilter on a measurable cardinal $\kappa$, $\alpha$ and $\beta$ are ordinals with $\alpha <\kappa$.
Suppose $G$ is a $V$-generic filter over $\mathbb P_\U$.
%Then  (1) $\Rightarrow$ (2) $\Rightarrow$  (3):
\begin{enumerate}
%\item  Eve has a winning strategy in  $\DG^\beta_{<\kappa,\alpha}(\mathfrak{M}_0,\mathfrak{M}_1)$ in $V[G]$, for any $V$-generic $G$ over $\mathbb P_\U$.
\item If Eve has a winning strategy in $\DG^\beta_{<\kappa,\alpha}(\mathfrak{M}_0,\mathfrak{M}_1)$ in $V[G]$,
then she has a winning strategy in $\DG^\beta_{<\kappa,\omega+ \alpha}(\mathfrak{M}_0,\mathfrak{M}_1)$ in $V$.
\item  If Eve has a winning strategy in  $\DG^\beta_{<\kappa,\alpha}(\mathfrak{M}_0,\mathfrak{M}_1)$ in $V$, then she has a winning strategy 
in  $\DG^\beta_{\kappa,\omega+\alpha}(\mathfrak{M}_0,\mathfrak{M}_1)$ in $V[G]$.
%\item  Eve has a winning strategy in  $\DG^\beta_{\kappa,\omega+\alpha+\omega+1}(\mathfrak{M}_0,\mathfrak{M}_1)$ in $V[G]$, for any $V$-generic $G$ over $\mathbb P_\U$.
\end{enumerate}
\end{proposition}
\begin{proof}
(1): % $\Rightarrow$ (2):
We assume for simplicity that $\tau$ has only relation symbols and that the domains of the two structures, $\mathfrak{M}_0$ and $\mathfrak{M}_1$, are disjoint. 
Let us first prove something a bit weaker. Namely, assume that the maximal condition in $\mathbb P_\U$ forces that $\check{\mathfrak{M}_0}$ and $\check{\mathfrak{M}_1}$
are isomorphic and let us show that Eve has a winning strategy in $\DG^\beta_{<\kappa,\omega}(\mathfrak{M}_0,\mathfrak{M}_1)$ in $V$.
Fix a $\mathbb P_\U$-name $\Gamma$ for the generic Prikry sequence and a $\mathbb P_\U$-name $\pi$ for the isomorphism between $\mathfrak{M}_0$ and $\mathfrak{M}_1$.
For each $a\in M_0$, let $N^1_a\in V$ be the set of possible values for $\pi(a)$ in $M_1$. By the $\kappa^+$-c.c., $|N^1_a|\le\kappa$. Let $f_a:N^1_a\to\kappa$ be an injective function in $V$.
Let $\rho(a)$ be a $\mathbb P_\U$-name for the least integer $n$ such that $f_a(\pi(a))< \Gamma(n)$.
Similarly, if $b\in M_1$ we let $N^0_b$ denote the set of possible values for $\pi^{-1}(b)$. Again by the $\kappa^+$-cc, $| N^0_b| \le \kappa$.
Let $f_b:N^0_b\to \kappa$ be an injective function in $V$, and let $\rho(b)$ be a $\mathbb P_\U$-name for the least integer $n$ such that $f_b(\pi^{-1}(b))< \Gamma(n)$.
Thus the function $\rho$ maps  $M_0 \cup M_1$ to $\mathbb P_\U$-names for integers. 
We now informally describe a winning strategy for Eve in $\DG^\beta_{<\kappa,\omega}(\mathfrak{M}_0,\mathfrak{M}_1)$. 
Along the way Eve builds a decreasing sequence $(p_n)_n$ of conditions in $\mathbb P_\U$ such that $|s_{p_n}| = n+1$.
To simplify notation we will write $s_n$ for $s_{p_n}$ and $Y_n$ for $Y_{p_n}$.

Suppose the first move by Adam is $(\beta_0,B^0_0,B^0_1)$, where $\beta_0< \beta$,  $B^0_0\in [ M_0]^{<\kappa}$ and $B^0_1\in [M_1]^{<\kappa}$. 
%is a subset of $M_0$ of size $<\kappa$, and $B^0_1$ is a subset of $M_1$ of size $<\kappa$.
To begin Eve finds $X_0\in \U$ such that $(\emptyset, X_0)$ decides the values of $\rho(x)$, for $x\in B^0_0 \cup B^0_1$. 
She can do this by the Prikry lemma and the fact that there  $<\kappa$-many decisions to be made.
%\Dot{n}_a$, for $a\in B^0_0$, and $\Dot{n}_b$, for $b\in B^0_1$.
This gives us a function $\rho_0$ defined on $B^0_0\cup B^0_1$ taking integer values. 
%Let $h_0(a)$ be the value of $\Dot{n}_a$, for $a\in B^0_0$, and $h_0(b)$ be the value of $\Dot{n}_b$, for $b\in B^0_1$.
Eve then lets $\alpha_0=\min(X_0)$ and finds $Y_0\subseteq X_0\setminus \{\alpha_0\}$ with $Y_0\in U$ such that $(\{ \alpha_0\}, Y_0)$ decides the value
of $\pi(a)$, for all $a\in B^0_0$ such that $\rho_0(a)=0$, and similarly the value of $\pi^{-1}(b)$, for all $b \in B^0_1$ such that $\rho_0(b)=0$. 
The reason she can do that is that $(\{ \alpha_0\}, X_0 \setminus \{\alpha_0\})$ forces that the value of $\pi(a)$ belongs to $f_a^{-1}[\alpha_0]$ which is a subset of $N^1_a$ of size $<\kappa$.
Similarly, $(\{ \alpha_0\}, X_0 \setminus \{\alpha_0\})$ forces that  $\pi^{-1}(b)$ belongs to $f_b^{-1}[\alpha_0]$, which is a subset of $N_b^0$ of size $<\kappa$. 
Thus she has $<\kappa$-many  decisions to make. 
In this way she obtains a partial isomorphism $g_0$ of size $<\kappa$ whose domain contains $\rho_0^{-1}(0)\cap B^0_0$ and whose range contains $\rho_0^{-1}(0)\cap B^0_1$.
%Note that $g_0$ is a partial isomorphism. 
Let $A^0_0= {\rm dom}(g_0)\cup B^0_0$ and $A^0_1 = {\rm ran}(g_0)\cup B^0_1$. Eve then extends $\rho_0$ to a function $h_0$ on $A^0_0\cup A^0_1$
by setting $h_0(x)=0$, for all $x\in A^0_0 \setminus B^0_0$ and all $x\in A^0_1 \setminus B^0_1$.
Eve then plays as her first move $(\beta_0,A^0_0,A^0_1,g_0,h_0)$. She sets her first condition $p_0$ to be $(\{ \alpha_0\}, Y_0)$. 

Suppose that at stage $n$ we are in position $(\beta_n,A^n_0,A^n_1,g_n,h_n)$ and Eve has fixed a  condition $p_n=(s_n,Y_n)$. 
Suppose that Adam then plays $(\beta_{n+1}, B^{n+1}_0,B^{n+1}_1)$, for some $\beta_{n+1} < \beta_n$, $B^{n+1}_0\in [M_0]^{<\kappa}$, and $B^{n+1}_1\in [M_1]^{<\kappa}$.
Eve first finds $X_{n+1} \subseteq Y_n$ with $X_{n+1} \in U$ such that $(s_n, X_{n+1})$ decides the value of $\rho(x)$, for $x\in B^{n+1}_0 \cup B^{n+1}_1$.
She can do this by the Prikry lemma and the fact that there are $<\kappa$ decisions to be made. 
This gives us a functions $\rho_{n+1}$ on $B^{n+1}_0 \cup B^{n+1}_1$ with integer values.   
Eve lets $\alpha_{n+1}=\min(X_{n+1})$ and finds $Y_{n+1} \subseteq X_{n+1}\setminus \{ \alpha_{n+1}\}$ with $Y_{n+1} \in \U$
such that $(s_n \cup \{ \alpha_{n+1}\}, Y_{n+1})$ decides the values of $\pi(a)$, for all $a\in A^n_0$ with $h_n(a)\leq 1$, and all $a\in B^{n+1}_0$ with $\rho_{n+1}(a)\leq n+1$,
as well as the values of $\pi^{-1}(b)$, for all $b\in A^n_1$ with $h_n(b)\leq 1$, and all $b\in B^{n+1}_1$ with $\rho_{n+1}(b)\leq n+1$.
This gives us a new partial isomorphism $g_{n+1}$ of size $<\kappa$ which extends $g_n$. 
Let $A^{n+1}_0= A^n_0 \cup B^{n+1}_0 \cup {\rm dom}(g_{n+1})$ and $A^{n+1}_1= A^n_1 \cup B^{n+1}_1\cup {\rm ran}(g_{n+1})$. 
Eve  defines the function $h_{n+1}$ on $A^{n+1}_0 \cup A^{n+1}_1$ as follows.

\[
h_{n+1}(x)=\begin{cases}
			0, & \text{if $x\in {\rm dom}(g_{n+1})\cup {\rm ran}(g_{n+1})$,}\\
            h_n(x)\dot - 1, & \text{if $x\in A^n_0\cup A^n_1$,} \\
            \rho_{n+1}(x)\dot{-} (n+1), & \text{if $x\in (B^{n+1}_0 \setminus A^n_0) \cup (B^{n+1}_1 \setminus A^n_1)$.}
		 \end{cases}
\]
Here, for integers $x$ and $y$, we use
\[
x \dot - y = \max (x-y,0).
\] 
Eve then plays  $(\beta_{n+1},A^{n+1}_0,A^{n+1}_1,g_{n+1},h_{n+1})$
and sets $p_{n+1}= (s_n \cup \{ \alpha_{n+1}\}, Y_{n+1})$.
It is clear now that by playing in this way Eve wins the game $\DG^\beta_{<\kappa,\omega}(\mathfrak{M}_0,\mathfrak{M}_1)$.

We now return to the original problem. We will play the game $\DG^{\beta}_{<\kappa,\omega+\alpha}(\mathfrak{M}_0, \mathfrak{M}_1)$ in $V$ and we will refer to it 
as the $V$-game. We will play the game $\DG^{\beta}_{<\kappa,\alpha}(\mathfrak{M}_0, \mathfrak{M}_1)$  in $V[G]$ and will refer to it as the $V[G]$-game. 
 Let $\sigma$ be a $\mathbb P_\U$-name which is forced by the maximal condition to be a winning strategy for Eve in the $V[G]$-game. 
We prove that Eve has a winning strategy  in the $V$-game. 
%$\DG^{\beta}_{<\kappa, \omega +\alpha}(\mathfrak{M}_0, \mathfrak{M}_1)$ in $V$.
Along the way Eve will build a decreasing sequence $(p_n)_n$ of conditions in $\mathbb P_\U$.
As above we will write $p_n=(s_n,Y_n)$. We will ensure that $| s_n| =n$, for all $n$.

Suppose Adam starts the $V$-game  by playing some $(\beta_0,B^0_0,B^0_1)$, 
with $\beta_0 <\beta$, $B^0_0 \in [M_0]^{<\kappa}$ and $B^0_1 \in [M_1]^{<\kappa}$.
Eve considers the same move in the $V[G]$-game. 
Let $(\check{\beta}_0, \dot{A}^0_0,\dot{A}^0_1,\dot{g_0},\dot{h}_0)$ be a name for the reply of $\sigma$ in this game. 
%The strategy $\sigma$ replies by playing some $(\beta_0,A^0_0,A^0_1,g_0,h_0)$. Since the two structures are relational we may assume that 
%$A^0_0=B^0_0 \cup {\rm dom}(g_0)$ and $A^0_1=B^0_1 \cup {\rm ran}(g_0)$. 
Since $\alpha < \kappa$ and $B^0_0 \cup B^0_1$ has size $<\kappa$, Eve can find a set $Y_0\in \U$ such that $(\emptyset, Y_0)$ decides the value of $\dot{h}_0(x)$, for all $x\in B^0_0 \cup B^0_1$.
This gives us a function $h'_0: B^0_0 \cup B^0_1\to \alpha$. Eve then defines $h^*_0$ on $B^0_0 \cup B^0_1$ by:
\[
h^*_0(x) = \omega + h'_0(x).
\]
In the $V$-game Eve replies by playing $(\beta_0,B^0_0,B^0_1,\emptyset,h^*_0)$. She also sets $p_0=(\emptyset, Y_0)$.
Suppose now Adam plays $(\beta_1,B^1_0,B^1_1)$ in the $V$-game. We may assume that $B^0_0 \subseteq B^1_0$ and $B^0_1\subseteq B^1_1$
and let  $(\check{\beta}_1, \dot{A}^1_0,\dot{A}^1_1,\dot{g_1},\dot{h}_1)$ be a name for the reply of $\sigma$ to this move of Adam in the $V[G]$-game.
We first find $X_1 \subseteq Y_0$ with $X_1\in \U$ such that $(\emptyset, X_1)$ decides the value of $\dot{h}_1(x)$, for all $x\in B^1_0\cup B^1_1$. 
Let the resulting function be $h'_1$. 
For each $a\in B^1_0$ with $h'_1(a)=0$ let $N^1_a$ be the set of possible values of $\dot{g}_1(a)$. 
Then by the $\kappa^+$-cc of $\mathbb P_\U$, $| N^1_a |\leq \kappa$. Fix an injective function $f_a:N^1_a\to \kappa$, for all such $a$.
Let $\rho_1(a)$ be the $\mathbb P_\U$-name for the least integer $n$ such that $f_a(\dot{g}_1(a)) < \Gamma(n)$.
Similarly, for each $b\in B^1_1$ such that $h'_1(b)=0$, let $N^0_b$ be the set of possible values of $\dot{g}_1^{-1}(b)$.
Then $|N^0_b| \leq \kappa$ as well. Fix an injective function $f_b:N^0_b\to \kappa$, for all such $b$ and let $\rho_1(b)$ be a $\mathbb P_\U$-name
for the least integer $n$ such that $f_b(\dot{g}^{-1}_1(b)) < \Gamma (n)$. 
By shrinking the set $X_1$ we may assume that $(\emptyset, X_1)$ decides the value of $\rho_1(x)$, for all $x\in B^1_0 \cup B^1_1$ such that $h'_1(x)=0$. Let us call this value $r_1(x)$.
 Let $\alpha_0=\min(X_1)$. 
We can now find $Y_1\subseteq X_1\setminus \{ \alpha_0\}$ with $Y_1\in \U$ such that $(\{ \alpha_0\}, Y_1)$ decides
the value of $\dot{g}_1(a)$, for all $a\in B^1_0$ such that $h'_1(a)=0$ and $r_1(a)=0$,
and similarly decides the value of $\dot{g}^{-1}_1(b)$, for all $b\in B^1_1$ such that $h'_1(b)=0$ and $r_1(b)=0$.
This is because there are $<\kappa$-many possible values, so we can apply the Prikry lemma.
This gives us a partial isomorphism $g^*_1$. Let $A^1_0= B^1_0 \cup {\rm dom}(g^*_1)$ and $A_1^1=B^1_1\cup {\rm ran}(g^*_1)$. 
Define the function $h^*_1$ on $A^1_0 \cup A^1_1$ by: 
\[
h^*_{1}(x)=\begin{cases}
			0, & \text{ if $x\in {\rm dom}(g^*_1)\cup {\rm ran}(g^*_1)$,}\\
            r_1(x), & \text{ if $x\in B^1_0 \cup B^1_1$ and $h'_1(x)=0$,}\\
            \omega + h'_1(x),  & \text{ if $x\in B^1_0 \cup B^1_1$ and $h'_1(x) >0$.} 
		 \end{cases}
\]
Now, in the $V$-game Eve plays $(\beta_1,A^1_0,A^1_1,g^*_1,h^*_1)$ and she sets $p_1=(\{ \alpha_0\}, Y_1)$.
It should now be clear how to define a winning strategy for Eve in the $V$-game by using the construction from the first part of the proof.

(2):
We now consider  $\DG^\beta_{<\kappa,\alpha}(\mathfrak{M}_0,\mathfrak{M}_1)$ in $V$ and call it the $V$-game, 
and we consider  $\DG^\beta_{\kappa,\alpha+\omega}(\mathfrak{M}_0,\mathfrak{M}_1)$ in $V[G]$ and call it the $V[G]$-game. 
In $V[G]$, let $(\gamma_n)_n$ be the Prikry generic sequence. Thus $(\gamma_n)_n$ is an increasing sequence cofinal in $\kappa$.
Since $\mathbb P_\U$ has the $\kappa^+$-cc, every set in $[M_0]^{\leq \kappa}$ is covered by a set of size $\kappa$ which belongs to $V$, and similarly for sets in $[M_1]^{\leq \kappa}$.
We can therefore assume that in the $V[G]$-game Adam plays only sets of size $\leq \kappa$ which are in $V$.
For each $N\in V$ of size $\kappa$, fix an injective function $f_N: N\to \kappa$ which is in $V$.
In $V[G]$, we define $\rho_N: N\to \omega$ by letting $\rho_N(x)$ be the least integer $n$ such that $f_N(x) < \gamma_n$.
For each integer $n$, let 
\[
N(n) = \{ x \in N : \rho_N(x) < \gamma_n \}.
\]
Note that $N(n)\in V$ and has size $<\kappa$. 
Now, fix a winning strategy $\sigma$ for Eve in the $V$-game. We describe a winning strategy for Eve in the $V[G]$-game. 
Suppose Adam starts by playing $(\beta_0, C^0_0, C^0_1)$, for some $\beta_0 < \beta$, and $C^0_0 \in [M_0]^{\leq \kappa}$ and $C^0_1 \in [M_1]^{\leq \kappa}$.
We may assume that $C^0_0$ and $C^0_1$ are in $V$.
Let $C_0= C^0_0 \cup C^0_1$. Eve defines the function $h^*_0: C_0\to \omega+\alpha$ by letting, for all $x$, 
\[ 
h^*_0(x)= \omega + \rho_{C_0}(x).
\]
Eve then plays $(\beta_0,C^0_0,C^0_1,\emptyset, h^*_0)$ as a reply to Adam's move in the $V[G]$-game. 
Suppose now Adam plays some $(\beta_1, C^1_0, C^1_1)$, with $\beta_1 < \beta_0$, and $C^1_0\in [M_0]^{\leq \kappa}$ and $C^1_1 \in [M_1]^{\leq \kappa}$.
We may assume that $C^1_0$ and $C^1_1$ are in $V$, and $C^0_0 \subseteq C^1_0$ and $C^0_1 \subseteq C^1_1$. Let $C_1=C^1_0 \cup C^1_1$.
Let $B^0_0 = C^0_0(0)$ and $B^0_1= C^0_1(0)$. 
Eve goes to the $V$-game and lets Adam play $(\beta_1,B^0_0, B^0_1)$ as his first move in this game. 
Let the answer of the strategy $\sigma$ be $(\beta_1,A^0_0,A^0_1,g_0,h_0)$. Let $A_0= A^0_0 \cup A^0_1$.
Eve then sets $D^1_0= A^0_0 \cup C^1_0$ and $D^1_1=A^0_1\cup C^1_1$. Let $D_1= D^1_0 \cup D^1_1$.
She then defines a function $h^*_1: D_1 \to \alpha + \omega$ as follows: 

\[
h^*_{1}(x)=\begin{cases}
			h_0(x), & \text{ if $x\in A_0$,} \\  
            \alpha + \rho_{C_0}(x)-1, & \text{ if $x\in C_0\setminus A_0$,}\\
            \alpha + \rho_{C_1}(x),  & \text{ if $x\in C_1 \setminus (A_0 \cup C_0)$.} 
		 \end{cases}
\]
Let $g^*_1=g_0$. In the $V[G]$ Eve plays $(\beta_1,D^1_0,D^1_1,g^*_1,h^*_1)$ as her reply to Adam's previous move.

In the $n$-move, we have a position $(\beta_n,D^n_0,D^n_1,g^*_n,h^*_n)$ in the $V[G]$-game and a corresponding position 
$(\beta_n,A^{n-1}_0,A^{n-1}_1,g_{n-1},h_{n-1})$ in the $V$-game. We have that 
\[
A^{n-1}_0=\{ x \in D^n_0: h^*_n(x) < \alpha\},
\]
and similarly for $A^{n-1}_1$. We have that $g^*_n= g_{n-1}$. If we let $A_{n-1}=A^{n-1}_0\cup A^{n-1}_1$ we have
$h_{n-1}=h^*_n \restriction A_{n-1}$. Moreover, ${h_n^*}^{-1}(\alpha +m)$ has size $<\kappa$, for every integer $m$.
Suppose Adam then plays some $(\beta_{n+1}, C^{n+1}_0,C^{n+1}_1)$ in the $V[G]$-game. 
Without loss of generality $C^n_0 \subseteq C^{n+1}_0$ and $C^n_1\subseteq C^{n+1}_1$. 
Eve takes the elements of $D^n_0$ whose current height is $\alpha$ and calls the resulting set $B^n_0$.
She defines similarly $B^n_1$. She then lets Adam play $(\beta_{n+1},B^n_0,B^n_1)$ in the  corresponding position of the $V$-game. 
Let the response of $\sigma$ be $(\beta_{n+1},A^n_0,A^n_1,g_n,h_n)$.
Back in the $V[G]$-game she lets $D^{n+1}_0=D^n_0\cup C^{n+1}_0$ and $D^{n+1}_1=D^n_1\cup C^{n+1}_1$.
She has to define the height function $h^*_{n+1}$ on $D_{n+1}=D^{n+1}_0 \cup D^{n+1}_1$.
To elements $x$ of $A_n = A^n_0 \cup A^n_1$ she assigns height $h_n(x)$. 
She lowers the heights of the elements of $C_n \setminus A_n$ by $1$. To the elements $x$ of $C_{n+1}\setminus (A_n \cup C_n)$ 
she assigns heights $\alpha +\rho_{C_{n+1}}(x)$. This gives the height function $h^*_{n+1}$.
She lets $g^*_{n+1}=g_n$ and plays 
$(\beta_{n+1},D^{n+1}_0,D^{n+1}_1,g^*_{n+1},h^*_{n+1})$ in the $V[G]$-game. 
This describes a winning strategy for Eve in the $V[G]$-game. 
\end{proof}
%JV\todo{Can we find an application? Undefinability of $P=Q^\omega$, of cofinality $\omega$, of singular cardinality, of ...?}

%Assume then II has a w.str. in $V$. Note: every set (of ordinals) of size $\kappa$ in $V[G]$ is covered by a set of size $\kappa$ in $V$ by $\kappa^+$-c.c. and hence is the union of $\omega$ sets smaller than $\kappa$ in $V$ (using the $\kappa_n$s. In $V$ we assume w.str. for moves $<\kappa$ and height $\alpha$. In $V[G]$ moves are of size $\kappa$ but heights are $\alpha+\omega+1$.
%
%(2)$\to$(1):
%
%\end{proof}

\subsection{The linear orders $\oplus_n\Phi$ and $\Phi$}

Suppose $\theta$ is a cardinal and $\Phi$ is a $\theta^+$ saturated linear ordering. Let $\oplus_n \Phi$ denote the sum of $\omega$ many copies of $\Phi$,
i.e. $\omega \times \Phi$ with the lexicographical ordering. 
%and let $\Psi$ be isomorphic to $\oplus_n \Phi$, but disjoint from $\Phi$.
%Let $\kappa$ be inaccessible (not necessary) and $\Phi$ the lexicographic order of binary eventually zero sequences in $2^\kappa$. Note that $\Phi$ is saturated.

\begin{proposition}
  Eve has a winning strategy in $\DG_{\theta,\omega}(\oplus_n \Phi, \Phi)$.
\end{proposition}
%Adam has a winning strategy in $\DG_{\theta,\omega}(\oplus_n \Phi, \Phi)$.
%If $\theta<\kappa$ and $\beta\le\infty$, then Player II wins the game 
%$DG^\beta_{\theta,\omega}(\oplus_n\Phi,\Phi)$.
\begin{proof}
We may consider that $\oplus_n \Phi$ and $\Phi$ are disjoint. This can be easily achieved by replacing $\Phi$ by $\{x \} \times \Phi$, for some $x\notin \omega$.
For an integer $n$, let $\Phi_n$ denote the $n$-th copy of $\Phi$, i.e. $\{ n \} \times \Phi$, 
and let $\Phi[n]$ denote the sum of the first $n$ copies of $\Phi$, i.e. $n \times \Phi$.
Note that $\Phi[n]$ is $\theta^+$-saturated, for all $n>0$. 
%sum of $n$ copies of $\Phi$, i.e. $n \times \Phi$ with the lexicographical ordering.  Note that $\Phi_n$ is also $\theta^+$-saturated. 
For a subset $X$ of $\oplus_n \Phi$ we let $X(n)$ denote $X \cap \Phi_n$ and  $X[n]= X \cap  \Phi[n]$. 

We now describe a winning strategy for Eve in $\DG_{\theta,\omega}(\oplus_n \Phi, \Phi)$.
Suppose Adam starts by playing $(B^0_0,B^0_1)$ with $B^0_0\in [\oplus_n \Phi]^{\leq \theta}$ and $B^0_1\in [\Phi]^{\leq \theta}$.
Eve finds a partial isomorphism $g_0$ of size $\leq \theta$ from  $\Phi[1]$ to  $\Phi$ such that $B^0_0[1] \subseteq {\rm dom}(g_0)$ and $B^0_1\subseteq {\rm ran}(g_0)$.
She can do that since $\Phi[1]$ and $\Phi$ are both $\theta^+$-saturated.  
Let $A^0_0=B^0_0 \cup {\rm dom}(g_0)$ and $A^0_1={\rm ran}(g_0)$. Let $h_0: A^0_0 \cup A^0_1\to \omega$ be defined by: 
\[
h_0(x)=\begin{cases}
			0, & \text{ if $x\in {\rm dom}(g_0) \cup {\rm ran}(g_0)$,} \\  
            n,   & \text{ if $x\in B^0_0(n) \setminus {\rm dom}(g_0)$.}
		 \end{cases}
\]
She then plays $(A^0_0,A^0_1,g_0,h_0)$.
Suppose in the next move Adam plays some $(B^1_0,B^1_1)$. We may assume that $B^0_0 \subseteq B^1_0$ and $B^0_1 \subseteq B^1_1$.
Eve then extends $g_0$ to a partial isomorphism $g_1$ of size $\leq \theta$ from $\Phi[2]$ to $\Phi$ such that
$B^1_0[2] \subseteq {\rm dom}(g_1)$ and $B^1_1 \subseteq {\rm ran}(g_1)$. 
Let $A^1_0= B^1_0\cup {\rm dom}(g_1)$ and $A^1_1= {\rm ran}(g_1)$.  Let $h_1: A^1_0\cup A^1_1 \to \omega$ be defined by: 
\[
h_1(x)=\begin{cases}
			0, & \text{ if $x\in {\rm dom}(g_1) \cup {\rm ran}(g_1)$,} \\  
            n-1,  & \text{ if $x\in B^1_0(n) \setminus {\rm dom}(g_1)$, for $n>1$.}
		 \end{cases}
\]
It should be clear that by continuing in this way Eve wins the game $\DG_{\theta,\omega}(\oplus_n \Phi, \Phi)$. 
\end{proof}

\begin{corollary}
 The property of having countable cofinality for a linear order  cannot be detected by the game $\DG_{\theta,\omega}$.
 \qed
\end{corollary}

For more on cofinality $\omega$, see Section~\ref{openproblems}, problem (8).

%\subsection{Aronszajn and related lines}

%We show that the classes of Aronszajn lines, Souslin lines, Countryman lines, and real lines (i.e. suborders of the order of the real numbers) are all definable in $L^1_{\le\aleph_1}$.

%A linear order is called\begin{enumerate}
%    \item an \emph{Aronszajn line} if it does not contain a copy of $\omega_1$, $\omega_1^*$ or an uncountable separable order.
%    \item a \emph{Souslin line} if it has the countable chain conditions but is not separable.
%    \item a \emph{Countryman line} if it is uncountable but its square is the union of a countable family of chains.
%    \item a \emph{real line} it it is isomorphic to a suborder of the canonical oreder of the real numbers.
%\end{enumerate}

\section{A class of infinitary logics}

Let $\theta$ be a cardinal, and $\alpha$ and $\beta$ ordinals. 
Recall that for two structures $\mathfrak{M}$ and $\mathfrak{N}$ in the same vocabulary, we write
 $E^\beta_{\theta,\alpha}(\mathfrak{M}, \mathfrak{N})$
if Eve has a winning strategy in $\DG^\beta_{\theta,\alpha}(\mathfrak{M},\mathfrak{N})$.
This relation is reflexive and symmetric by definition, but we do not know if it is transitive. 
Therefore, we make the following definition.

 \begin{definition}\label{definition:beta-theta-alpha-equivalent}
Let $\theta$ be a cardinal and $\alpha$ and $\beta$ ordinals. We let $\equiv^\beta_{\theta,\alpha}$ be the transitive closure of the relation $E^\beta_{\theta,\alpha}$. 
 \end{definition}

%JV
Note that if we write
 $E^\beta_{\theta,<\alpha}(\mathfrak{M}, \mathfrak{N})$
when Eve has a winning strategy in $\DG^\beta_{\theta,<\alpha}(\mathfrak{M},\mathfrak{N})$, then for additively indecomposable $\alpha$, e.g. if $\alpha=\omega^2$, the relation  $E^\beta_{\theta,<\alpha}(\mathfrak{M}, \mathfrak{N})$ is, by Proposition~\ref{proposition:natural-sum1}, transitive, and hence an equivalence relation.

\begin{remark}
    Note that if $\beta' \geq \beta$, $\theta' \geq \theta$ and $\alpha' \leq \alpha$
    then $\equiv^{\beta'}_{\theta',\alpha'}$ refines $\equiv^{\beta}_{\theta,\alpha}$. 
\end{remark}

% \begin{definition}\label{definition:beta-theta-alpha-equivalent}
% Suppose $\beta$ and $\alpha$ are ordinals, and $\theta$ is a cardinal. Let $\mathfrak{M}$ and $\mathfrak{N}$ be two structures in the same vocabulary $\tau$. We say that $\mathfrak{M}$ are $(\beta,\theta,\alpha)$-{\em equivalent} and write 
%  \[
% \mathfrak{M}\equiv^{\beta}_{\theta,\alpha} \mathfrak{N},
% \]
% if for every sub-vocabulary $\tau_0$ of $\tau$ of size $\leq \theta$,  there is a finite sequence of $\tau_0$-structures,  $\mathfrak{M}_0, \ldots, \mathfrak{M}_n$, such that $\mathfrak{M}_0 = \mathfrak{M}\restriction \tau_0$, $\mathfrak{M}_n = \mathfrak{N}\restriction \tau_0$, and $E^{\beta}_{\theta,\alpha}(\mathfrak{M}_i,\mathfrak{M}_{i+1})$, holds for all $i<n$. We let $\mathscr L^{\beta}_{\theta,\alpha}$ denote the derived logic, namely sentences of $\mathscr L^{\beta}_{\theta,\alpha}$ are simply the $\equiv^{\beta}_{\theta,\alpha}$-equivalence classes and a structure $\mathfrak{M}$ satisfies a sentence if it belongs to this equivalence class.

% \end{definition}

\begin{definition}\label{definition:theta-alpha-logic}
    Suppose $\kappa$ is a cardinal and $\alpha$ is an ordinal. For any signature $\tau$  we define the following.
    \begin{enumerate}
        \item An $\mathscr L^1_{\kappa,\alpha}[\tau]$-sentence $\phi$ is a class of $\tau_\phi$-structures closed under $\equiv^{\beta}_{\theta,\alpha}$,
        for some cardinal $\theta <\kappa$, a vocabulary $\tau_\phi\in [\tau]^{\leq \theta}$, and an ordinal $\beta < \kappa$.
        \item If $\phi$ is a $\tau$-sentence we let $\theta_\phi$ be the least $\theta$ as above and let $\beta_\phi$ be the least
        ordinal for which $\phi$ is closed under $\equiv^\beta_{\theta_\phi,\alpha}$.
        \item For a $\tau$-structure $\mathfrak{M}$ and $\phi\in\mathscr L^1_{\kappa,\alpha}$, we let:
        \[
           \mathfrak{M} \models \phi \hspace{3mm} \mbox{ if and only if } \hspace{3mm}  \mathfrak{M} \rest \tau_\phi \in \phi. 
        \]
    \end{enumerate}
\end{definition}

It may seem awkward that the sentences of $\mathscr L^1_{\kappa,\alpha}$ are proper classes. However, as we will prove in Proposition~\ref{prop:between-logics},  if $\kappa$ is a regular cardinal and $\alpha <\kappa$ is an ordinal, then $\mathscr L^1_{\kappa,\alpha} \leq \mathscr L_{\hspace{1pt}\beth_\kappa,\kappa}$. Note also the forthcoming Proposition~\ref{theorem:equality-thm}, which connects our logics to Shelah's logics. 

\begin{proposition}\label{closure-properties}
    Let $\kappa$ be an uncountable cardinal and $\alpha$ an ordinal with $0 <\alpha <\kappa$. The logic $\mathscr L^1_{\kappa,\alpha}$ is closed under atomic sentences, 
    negation, conjunction, and particularization. 
\end{proposition}
\begin{proof}
   
\noindent (1) (Atom property) Assume $\phi\in \mathscr L_{\omega,\omega}[\tau]$. We may assume that $\tau$ is finite. 
We need to find $\psi\in \mathscr L^1_{\kappa,\alpha}$, i.e. a model class closed under $\equiv^\beta_{\theta,\alpha}$, for some cardinal $\theta < \kappa$ and ordinal $\beta <\kappa$, 
such that, for every $\tau$-structure $\mathfrak M$,
\[
\mathfrak{M} \models \phi \hspace{5mm} \mbox{$\Longleftrightarrow$} \hspace{5mm} \mathfrak{M} \rest \tau_{\psi} \in \psi.
\]
We let $\psi$ be the class of models satisfying $\phi$. In particular, $\tau_\psi=\tau$. 
Note that $\psi$ is closed under the relation $\equiv^\beta_{\omega,\alpha}$ by the original simple argument
of Ehrenfeucht \cite{MR126370}, for any $\alpha <\beta <\kappa$.

\noindent (2)  (Negation) Let $\tau$ be a vocabulary and $\phi\in \mathscr L^1_{\kappa,\alpha}$. Then $\neg \phi$ can be taken to be the complement of $\phi$.

\noindent (3) (Conjunction) Suppose $\tau$ is a vocabulary and $\phi_0,\phi_1\in \mathscr L^1_{\kappa,\alpha}[\tau]$. We need to find an $\mathscr L^1_{\kappa,\alpha}[\tau]$
sentence $\psi$ such that 
\[
{\rm Mod}^\tau_{\mathscr L^1_{\kappa,\alpha}}(\psi) = {\rm Mod}^\tau_{\mathscr L^1_{\kappa,\alpha}}(\phi_0) \cap
{\rm Mod}^\tau_{\mathscr L^1_{\kappa,\alpha}}(\phi_1).
\]
Let $\psi$ be the class of $\tau_{\phi_0} \cup \tau_{\phi_1}$-structures $\mathfrak{M}$ such that $\mathfrak{M}\rest \tau_{\phi_0}\in \phi_0$ 
and $\mathfrak{M}\rest \tau_{\phi_1}\in \phi_1$. Let $\theta=\max \{ \theta_{\phi_0}, \theta_{\phi_1}\}$ and let $\beta=\max \{ \beta_{\phi_0},\beta_{\phi_1}\}$. 
Note that $\psi$ is closed under $\equiv^\beta_{\theta,\alpha}$. Hence  $\psi\in \mathscr L^1_{\kappa,\alpha}[\tau]$, and is as required.

\noindent (4) (Particularization property) Let $\tau$ be a vocabulary and $c$ a new constant symbol. Suppose $\phi$ is an $\mathscr L^1_{\kappa,\alpha}[\tau \cup \{ c \}]$-sentence. 
Let 
\[ 
\psi = \{ \mathfrak{M} \in {\rm Str}[\tau] \ : \exists a \, (\mathfrak{M},a) \in \phi \}.
\]
Let $\theta=\theta_\phi$ and $\beta= \beta_{\phi} + \alpha +1$.  Let us check that $\psi$ is closed under $\equiv^\beta_{\theta,\alpha}$.
First of all, if this class is empty, then it is an $\mathscr L^1_{\kappa,\alpha}$-sentence. 
Suppose now $\mathfrak{M}\in \psi$ and suppose $\mathfrak N$ is another $\tau$-structure such that Eve has a winning strategy $\sigma$ in $\DG^\beta_{\theta,\alpha}(\mathfrak{M}, \mathfrak{N})$.
We need to show that $\mathfrak{N} \in \psi$. We describe a partial play in ${\DG}^{\beta}_{\theta,\alpha}(\mathfrak{M}, \mathfrak{N})$ in which Eve uses $\sigma$.
Let $a \in M$ be such that $(\mathfrak{M},a) \in \phi$. Adam starts by letting $\beta_0= \beta_\phi + \alpha$ and playing $(\beta_0, \{ a\},\emptyset)$.
Eve follows her strategy $\sigma$ and assigns some height $\alpha_0 <\alpha$ to $a$. Adam then sets $\beta_1= \beta_\phi + \alpha_0$. 
In finitely many steps the height $\alpha_n$ assigned to $a$ by Eve must reach $0$ while $\beta_n \geq \beta_\phi$. 
At this stage Eve must map $a$ to some element of $N$. Since Eve can follow $\sigma$ from this position on and win the game ${\DG}^{\beta_\phi}_{\theta,\alpha}((\mathfrak{M},a), (\mathfrak{N},b))$
it follows that $(\mathfrak{M},a)\equiv^{\beta_\phi}_{\theta,\alpha} (\mathfrak{N},b)$. Since $\phi$ is closed under this relation $(\mathfrak{N},b)\in \phi$, 
hence $\mathfrak{N}\in \psi$.
\end{proof}

\begin{proposition}\label{proposition-atomic-substitution}
    Let $\kappa$ be an uncountable cardinal and $\alpha$ an ordinal with $0 <\alpha <\kappa$. The logic $\mathscr L^1_{\kappa,\alpha}$ has the atomic substitution property. 
\end{proposition}

\begin{proof}
Suppose that $\tau$ and $\sigma$ are vocabularies. Let $U\in \sigma$ be a binary relation.  
Suppose for every constant symbol $c\in \tau$, we have a unary function symbol $K_c \in \sigma$, for every $n$-ary relation symbol $R\in \tau$, 
we have an $n+1$-ary relation symbol $K_R\in \sigma$, and for every $n$-ary function symbol $f\in \tau$, we have an $n+1$-ary function symbol $K_f\in \sigma$. Finally, let $d\in \sigma$ be a constant symbol. 
Now, let $\mathfrak{A}$ be a $\sigma$-structure and let $a=d^{\mathfrak{A}}$. 
%We want to define a $\tau$-structure $\mathfrak{B}_a$. 
Let 
\[
U_a^{\mathfrak{A}}=\{ b \in A: U^{\mathfrak{A}}(b,a)\}.
\]
If $R\in \tau$ is an $n$-ary relation symbol we can define an $n$-ary relation $R^{\mathfrak{U}_a^{\mathfrak{A}}}$ on $U_a^{\mathfrak{A}}$ by letting 
\[
R^{\mathfrak{U}^{\mathfrak{A}_a}}(b_0,\ldots, b_{n-1}) \hspace{5mm}\mbox{iff} \hspace{5mm} K_R^{\mathfrak{A}}(b_0,\ldots,b_{n-1},a).
\]
If $c\in \tau$ is a constant symbol, we let $c^{\mathfrak{U}^{\mathfrak{A}}_a}=K_c^{\mathfrak{A}}(a)$. 
If $f\in \tau$ is an $n$-ary function symbol,
we can define an $n$-ary function $f^{\mathfrak{U}^{\mathfrak{A}}_a}$ by letting $f^{\mathfrak{U}^{\mathfrak{A}}_a}(\bar{b})=K_f^{\mathfrak{A}}(\bar{b},a)$, for all $\bar{b}\in A^n$. 
%If $c^{\mathfrak{B}_a}\in B_a$, for all constant symbols $c\in \tau$, and $B_a$ is closed under the function $f^{\mathfrak{B}_a}$, for all function symbols $f\in \tau$, then we obtain a $\tau$-structure $\mathfrak{B}_a$.
If $U_a^{\mathfrak{A}}$ is closed under these interpretation of the symbols of $\tau$ we obtain a $\tau$-structure $\mathfrak{U}^{\mathfrak{A}}_a$. 
Suppose $\phi$ is a $\tau$-sentence. We need to find a $\sigma$-sentence $\psi$ such that 
\[
\mathfrak{A}\models_{\mathscr L^1_{\kappa,\alpha}}\psi \hspace{5mm} \mbox{iff \hspace{5mm} $\mathfrak{U}^{\mathfrak{A}}_a$ is  $\tau$-closed and } 
\mathfrak{U}^{\mathfrak{A}}_a \models_{\mathscr L^1_{\kappa,\alpha}} \phi.
\] 
We simply let:
\[
\psi = \{ \mathfrak{A}\in {\rm Str}[\sigma]: \mathfrak{U}_{d^{\mathfrak{A}}} \mbox{ is $\tau$-closed and } \mathfrak{U}^{\mathfrak{A}}_{d^{\mathfrak{A}}} \in \phi\}.
\] 
In order to show that $\psi \in \mathscr L^1_{\kappa,\alpha}$ it suffices to find a cardinal $\theta< \kappa$ and an ordinal $\beta < \kappa$
such that $\psi$ is closed under $\equiv^\beta_{\theta,\alpha}$. 

Let $\theta=\theta_\phi$ and $\beta=\beta_\phi+ \alpha +1$. Assume that $\mathfrak{M}$ and $\mathfrak{N}$ are two $\sigma$-structures, $\mathfrak{M}\in \psi$,
and Eve has a winning strategy in ${\DG}^\beta_{\theta,\alpha}(\mathfrak{M},\mathfrak{N})$. We have to show that $\mathfrak{N}\in \psi$. 
Let $a=d^{\mathfrak{M}}$ and $b=d^{\mathfrak{N}}$. 

First we show that  $U_a^{\mathfrak{M}}$ is $\tau$-closed if and only if $U_b^{\mathfrak{N}}$ is $\tau$-closed. 
Assume towards contradiction that $U_a^{\mathfrak{M}}$ is $\tau$-closed, but $U_b^{\mathfrak{N}}$ is not $\tau$-closed.
Let $B$ be a finite subset of $U_b^{\mathfrak{N}}$ whose $\tau$-closure is not contained in $U_b^{\mathfrak{N}}$.
In the first move Adam sets $\beta_0=\beta_\phi+ \alpha$ and plays $(\beta_0, \emptyset, B)$. 
Let us say that Eve follows her winning strategy and plays $(\beta_0, A^0_0,A^0_1,g_0,h_0)$. 
Let $\alpha_0= \max \{ h_0(x): x \in B\}$. 
Adam then sets $\beta_1= \beta_\phi + \alpha_0$, and plays $(\beta_1,\emptyset, \emptyset)$.
At stage $n$ Eve plays $(\beta_n, A^n_0,A^n_1,g_n,h_n)$. Let $\alpha_n= \max \{ h_n(x): x\in B\}$.
Adam sets $\beta_{n+1}=\beta_\phi + \alpha_n$ and plays $(\beta_{n+1},\emptyset, \emptyset)$ as his next move.
At some point the ordinal $\alpha_n$ reaches $0$  and then Eve is obliged to have $B \subseteq {\rm ran}(g_n)$. 
Since $a$ and $b$ are interpretations of $d$ in $\mathfrak{M}$ and $\mathfrak{N}$ we must have $a \in {\rm dom}(g_n)$ and $b \in {\rm ran}(g_n)$.
But then $g_n$ cannot be a partial isomorphism, a contradiction. 

In case $U_a^{\mathfrak{M}}$  and $U_b^{\mathfrak{N}}$ are both $\tau$-closed, it suffices to show that Eve has a winning strategy in the game
\[
{\rm DG}^{\beta_\phi}_{\theta,\alpha}(\mathfrak{U}^{\mathfrak{M}}_a,\mathfrak{U}^{\mathfrak{N}}_b).
\]
But  this is clear; if Adam plays inside $U_a^{\mathfrak{M}} \cup U_b^{\mathfrak{N}}$ then Eve's winning strategy in the game ${\DG}^\beta_{\theta,\alpha}(\mathfrak{M},\mathfrak{N})$
restricted to $U_a^{\mathfrak{M}} \cup U_b^{\mathfrak{N}}$ gives rise to a winning strategy for Eve in this game as well. 
This shows that $\psi$ is an $\mathscr L^1_{\kappa,\alpha}$-sentence, and therefore that $\mathscr L^1_{\kappa,\alpha}$ has the atomic substitution property. 
\end{proof}
 
The way we have defined the logic $\mathscr L^1_{\kappa,\alpha}$, given a vocabulary $\tau$ each $\tau$-sentence is a proper class.
In order to remedy this situation we show that each $\tau$-sentence has a model of bounded cardinality. 
This is done as in \cite{MR2869022} by using the results of Karp \cite{MR0209132} on $\mathscr L_{\infty,\omega}$, as generalized by Benda \cite{MR250860} to $\mathscr L_{\infty,\gamma}$ (see also 
\cite[Chapter~3]{MR0539973}).

\begin{definition}\label{def:back-and-forth}
    Let $\theta$ be a cardinal and  $\mathfrak{M}$ and $\mathfrak{N}$  two structures in the same vocabulary $\tau$ of cardinality $\leq \theta$.
    We define a sequence $\{ \mathcal I^\alpha_\theta(\mathfrak{M},\mathfrak{N}) :\alpha \in {\rm ORD}\}$ of families of partial isomorphisms from $\mathfrak{M}$ to $\mathfrak{N}$ as follows.
    \begin{itemize}
        \item  Let ${\mathcal I}^0_\theta(\mathfrak{M},\mathfrak{N})$ be the family of all partial isomorphisms from $\mathfrak{M}$ to $\mathfrak{N}$ of size at most $\theta$. 
        \item  Suppose ${\mathcal I}^\alpha_\theta(\mathfrak{M},\mathfrak{N})$ has been defined for all $\alpha <\beta$. We say that a partial isomorphism $f$ of size at most $\theta$ is in 
        ${\mathcal I}^\beta_\theta(\mathfrak{M},\mathfrak{N})$ if for every $A \in [M]^{\leq \theta} \cup [N]^{\leq \theta}$ and every $\alpha < \beta$
    there is $g\in {\mathcal I}^\alpha_\theta(\mathfrak{M},\mathfrak{N})$ which extends $f$ and covers $A$.
     \end{itemize}
    Here $g$ covers $A$ means that if $A \in [M]^{\leq \theta}$ then $A \subseteq {\rm dom}(g)$ and if $A\in [N]^{\leq \theta}$
    then $A \subseteq {\rm ran}(g)$.
    If ${\mathcal I}^\beta_\theta(\mathfrak{M},\mathfrak{N})$ is non-empty we say that  
  $\mathfrak{M}$ and $\mathfrak{N}$ are $\theta$-{\em back and forth equivalent up to} $\beta$ and write 
    \[ \mathfrak{M} \equiv^\beta_\theta \mathfrak{N}.
    \]
\end{definition}
Suppose now that $\theta$ is a cardinal and that $\mathfrak{M}$ is a $\tau$-structure, for some vocabulary $\tau$ of cardinality at most $\theta$. 
Observe that if $\beta$ is an ordinal and $f\in {\mathcal I}^{\beta}_\theta(\mathfrak{M},\mathfrak{M})$, and $g\subseteq f$ is a partial isomorphism then $g\in {\mathcal I}^{\beta}_\theta(\mathfrak{M},\mathfrak{M})$.
Moreover, if $f,g\in {\mathcal I}^{\beta}_\theta(\mathfrak{M},\mathfrak{M})$ with ${\rm ran}(f)\subseteq {\rm dom}(g)$ then $g\circ f \in {\mathcal I}^{\beta}_\theta(\mathfrak{M},\mathfrak{M})$.
Given two sequence $\bar{a}=(a_i)_{i <\theta}$ and $\bar{b}=(b_i)_{i<\theta}$ of elements of $M$ we write 
\[
\bar{a}\sim_{\beta} \bar{b}
\]
if there is $f\in {\mathcal I}^{\beta}_\theta(\mathfrak{M},\mathfrak{M})$ such that $f(a_i)=b_i$, for all $i <\theta$. It follows that $\sim_\beta$ is an equivalence relation.
We now want to estimate the number of $\sim_\beta$-equivalence classes. 
Consider first the case $\beta=0$. Note that if $\bar{a}$ and $\bar{b}$ are  two sequences of length $\theta$, then $\bar{a}\sim_0\bar{b}$
iff the map $a_i \mapsto b_i$ extends to a partial isomorphism. 
Let ${\rm tp}_0(\bar{a})$ denote the isomorphism type of the sequence $\bar{a}$. 
Then $\bar{a}\sim_0 \bar{b}$ iff ${\rm tp}_0(\bar{a})= {\rm tp}_0(\bar{b})$. 
Since the vocabulary $\tau$ has cardinality at most $\theta$, the cardinality of the set of  $\sim_0$-equivalence classes is at most  $2^\theta$. 

Now, if $\bar{a},\bar{d}\in M^\theta$ and $f:\theta\to \theta$ is such that $a_i=d_{f(i)}$, for all $i <\theta$, we write $\bar{a} \sqsubseteq_f \bar{d}$. 
Suppose we have defined ${\rm tp}_\alpha(\bar{a})$, for all $\alpha <\beta$ and all $\bar{a}\in M^\theta$. Given $\bar{a}\in M^\theta$ we let
\[
{\rm tp}_\beta (\bar{a}) = \{ ({\rm tp}_\alpha(\bar{d}),f): \alpha <\beta, \bar{d}\in M^\theta, f\in \theta^\theta, \mbox{ and } \bar{a}\sqsubseteq_f \bar{d}\}.
\]
It should be clear that if ${\rm tp}_\beta(\bar{a})= {\rm tp}_\beta(\bar{b})$ then $\bar{a}\sim_\beta \bar{b}$, so it suffices to estimate the number of $\beta$-types of $\theta$-sequences
of elements of $M$. Let $T_\beta$ be the set of $\beta$-types of $\theta$-sequences in $M$ and let $\lambda_\beta$ be its cardinality. 
Then by the above analysis we have that
\[
\lambda_\beta \leq \prod_{\alpha < \beta}\exp({\lambda_\alpha \times 2^\theta}).
\]
From this we conclude that $\lambda_\beta \leq \beth_{\beta+1}(\theta)$. We now have the following well-known result, due in one form or another to
\cite[Theorem 10.3.8]{MR0176910} and \cite[Theorem 3.4.1]{MR0539973}:

\begin{theorem} Let $\theta$ be an infinite cardinal and $\beta$ an ordinal. 
Suppose $\mathfrak{M}$ is a structure in a vocabulary $\tau$ of cardinality at most $\theta$. 
Then there is a substructure $\mathfrak{N}$ of $\mathfrak{M}$ of cardinality at most $\beth_{\beta+1}(\theta)$
such that $\mathfrak{M} \equiv^\beta_\theta \mathfrak{N}$.
\end{theorem}
\begin{proof} We give a short sketch of a proof for the benefit of the reader:
    Let $\mu$ be a sufficiently large regular cardinal. Pick $H\prec H(\mu)$ of cardinality $\beth_{\beta+1}(\theta)$ and closed under $\theta$-sequences
    such that $\tau, \mathfrak{M}\in H$ and $T_\beta \subseteq H$. Let $N= M \cap H$ and let $\mathfrak{N}$ be the substructure of $\mathfrak{M}$ with universe $N$.
    It is straightforward to see that $\mathfrak{M} \equiv^\beta_\theta \mathfrak{N}$.
\end{proof}

Let us now observe that if $\theta$ is a cardinal, $\beta$ an ordinal, and $\mathfrak{M}$ and $\mathfrak{N}$ are two structures in the same vocabulary of size at most $\theta$, 
then 
\[
\mathfrak{M} \equiv^{2\cdot \beta}_\theta \mathfrak{N}  \Longrightarrow \mathfrak{M} \equiv^\beta_{\theta,1} \mathfrak{N}.
\] 
The reason we need $2\cdot \beta$ in the first equivalence is that in the game $\DG^\beta_{\theta,1}(\mathfrak{M},\mathfrak{N})$ Adam is allowed to play sets of size $\leq \theta$
in both models at the same time, whereas if $f\in \mathcal I^{\beta+1}_\theta(\mathfrak{M},\mathfrak{N})$ then for every $A$ which is either in $[M]^{\leq \theta}$ or $[N]^{\leq \theta}$ 
there is $g\in \mathcal I^\beta_\theta(\mathfrak{M},\mathfrak{N})$ extending $f$ which covers $A$.

Now, the upshot is that every $\equiv^\beta_{\theta,1}$-equivalence class of models in a fixed vocabulary $\tau$ of size at most $\theta$ has an element which is of cardinality at most $\beth_{2\cdot \beta+1}(\theta)$.
It follows that if $\kappa$ is a cardinal and $\alpha <\kappa$ then the logic $\mathscr L^1_{\kappa,\alpha}$ is small.
Strictly speaking we should redefine the collection of $\mathscr L^1_{\kappa,\alpha}[\tau]$ sentences, for a vocabulary $\tau$,
and the satisfaction relation as follows. A $\tau$-sentence is a class of models in some vocabulary $\tau_\phi\in [\tau]^{\leq \theta}$
closed under $\equiv^\beta_{\theta,\alpha}$ and consisting of models with universe a subset of $\beth_{2\cdot \beta}(\theta)$,
for some cardinal $\theta<\kappa$ and ordinal $\beta <\kappa$. If $\mathfrak{M}$ is a $\tau$-structure and $\phi$
is a $\tau$-sentence we let:
\[
\mathfrak{M}\models_{\mathscr L^1_{\kappa,\alpha}} \phi \hspace{3mm} 
\mbox{ if there is a structure $\mathfrak{N}\in \phi$ such that } \hspace{3mm}
\mathfrak{M}\rest \tau_\phi \equiv^{\beta_\phi}_{\theta_\phi,\alpha} \mathfrak{N}.
\]
However, we will not bother with these subtleties as they do not pose a problem. 
We now have the following.

\begin{corollary}\label{cor:regular}
    Let $\kappa$ be an infinite cardinal and $\alpha< \kappa$ an ordinal. Then the logic $\mathscr L^1_{\kappa,\alpha}$ is regular and has occurrence number $\kappa$.
    \qed
\end{corollary}

Recall that $\mathscr L_{\kappa,\lambda}$ denotes the usual infinitary logic where we start with atomic formulas
and close under negations, conjunctions and disjunctions of $<\kappa$ formulas, and quantification over $<\lambda$ variables. 

\begin{proposition}\label{prop:between-logics}
    Let $\kappa$ be a regular cardinal and $\alpha <\kappa$ an ordinal. Then 
    \[
    \mathscr L_{\kappa,\omega} \leq \mathscr L^1_{\kappa,\alpha} \leq \mathscr L_{\hspace{1pt}\beth_\kappa,\kappa}.
    \]
\end{proposition}

\begin{proof}
    To see that $\mathscr L_{\kappa,\omega} \leq \mathscr L^1_{\kappa,\alpha}$, it suffices to show that 
    $\mathscr L^1_{\kappa,\alpha}$ is closed under conjunctions of $<\kappa$ sentences. 
    Suppose $\tau$ is a vocabulary and $\gamma <\kappa$ and $\{\phi_\xi : \xi <\gamma\}$ are $\mathscr L^1_{\kappa,\alpha}[\tau]$-sentences. Let $\tau'=\bigcup \{ \tau_{\phi_\xi}: \xi <\gamma\}$, let $\theta=\sup \{ \theta_{\phi_\xi}: \xi <\gamma\}$
    and let $\beta=\sup \{ \beta_{\phi_\xi}: \xi <\gamma\}$. Since $\kappa$ is regular we have that $\theta,\beta <\kappa$. 
    We also have that $\tau'$ is of cardinality at most $\theta$.
    Then the class of $\tau'$-structures: 

    \[
    \bigwedge_{\xi <\gamma} \phi_\xi = \{ \mathfrak{M}\in {\rm Str}[\tau']: \forall \xi \, \, \mathfrak{M}\rest \tau_{\phi_\xi}\in \phi_\xi \} 
    \]
    is closed under the relation $\equiv^\beta_{\theta,\alpha}$. Hence $\bigwedge_{\xi <\gamma} \phi_\xi$ is an $\mathscr L^1_{\kappa,\alpha}$-sentence. Furthermore, its models are precisely the $\tau$-structures that satisfy all the $\phi_\xi$, for $\xi< \gamma$.

    To see that $\mathscr L^1_{\kappa,\alpha} \leq \mathscr L_{\beth_\kappa,\kappa}$, let $\phi$ be an $\mathscr L^1_{\kappa,\alpha}$-sentence in some vocabulary $\tau$. Then $\phi$ is $\equiv^\beta_{\theta,\alpha}$-closed, for some cardinal $\theta <\kappa$
    and ordinal $\beta <\kappa$.  We may assume that $\tau$ is of cardinality $\leq \theta$. 
    Then $\equiv^{2\cdot \beta}_\theta$ refines the relation $\equiv^\beta_{\theta,1}$ and hence also $\equiv^\beta_{\theta,\alpha}$.
    If we let $\gamma=\beth_{2\cdot \beta+1}(\theta)$, every $\equiv^{2\cdot \beta}_\theta$-equivalence class is definable by a formula in $\mathscr L_{\gamma,\theta^+}$
    and the number of such classes is bounded by $\exp(\gamma)$ which is less than $\beth_\kappa$.
    Thus each $\phi$ is definable by a formula in $\mathscr L_{\beth_\kappa,\kappa}$.
\end{proof}

Shelah's logic $\mathbb L^1_\kappa$ from \cite{MR2869022} is what we call $\mathscr L^1_{\kappa,\omega}$.
This logic has some very interesting properties if $\kappa=\beth_\kappa$ (see next section). From Theorem \ref{theorem:main-thm} we obtain the following. 

\begin{theorem}\label{theorem:equality-thm}
    Suppose $\kappa$ is a cardinal such that $\beth_\kappa=\kappa$. 
    Then $\mathscr L^1_{\kappa,\alpha}\equiv \mathscr L^1_{\kappa,\omega}$, for every $\omega \leq \alpha <\kappa$.
    \qed
\end{theorem}

Thus, when $\kappa=\beth_\kappa$ and  $\omega \leq \alpha <\kappa$, our new logics $\mathscr L^1_{\kappa,\alpha}$ give nothing new over $L^1_{\kappa,\omega}$. However, we emphasize that when $\kappa<\beth_\kappa$, we obtain a finer analysis, as demonstrated e.g. by Proposition~\ref{4.10}.

\section{Characterizing the logic $\mathbb L^1_{\kappa}$}

As mentioned above Shelah \cite{MR2869022} showed that his logic $\mathbb L^1_\kappa$, for $\kappa$ a fixed point of the $\beth$-function 
has a Lindstr\"{o}m type characterization and a number of other desirable properties.
In this section we give a somewhat different presentation of these results.  
The logic $\mathbb L^1_\kappa$ does not, a priori, have the Tarski union property. 
However, we will define a weaker relation $\preceq_\phi$, an elementary submodel relation relative to a formula $\phi$, 
and we will define a weakened version of the Tarski Union property, by only requiring that each model in a $\preceq_\phi$-chain
agrees on the formula $\phi$ with the union of the chain. The logic $\mathbb L^1_\kappa$ satisfies this weakened version of the Tarski Union property.

Let $\tau$ be a vocabulary and $\mathfrak{M}$ and $\mathfrak{N}$ two $\tau$-structures. For a $\tau$-sentence $\phi$
we write $\mathfrak{M} \equiv_\phi \mathfrak{N}$ if the structures $\mathfrak{M}$ and $\mathfrak{N}$ agree on $\phi$.

\begin{definition}\label{definition:preceq-relation} 
Let $\mathscr L$ be a logic and $\tau$ a vocabulary and let $\phi$ be an $\mathscr L[\tau]$-sentence. 
A relation $\preceq_\phi$ on pairs of $\tau$-structures  is a $\phi$-{\em submodel relation} if 
for any $\tau$-structures $\mathfrak{M}$, $\mathfrak{N}$ and $\mathfrak{P}$ we have the following: 
\begin{enumerate}
    \item If $\mathfrak{M}\preceq_\phi \mathfrak{N}$ then $\mathfrak{M}$ is a substructure of $\mathfrak{N}$ and
    $\mathfrak{M}\equiv_\phi \mathfrak{N}$. 
    \item If $\mathfrak{M}\preceq_\phi \mathfrak{N}\preceq_\phi \mathfrak{P}$ then $\mathfrak{M}\preceq_\phi \mathfrak{P}$.
    \item (Tarski-Vaught Property) If $\mathfrak{M}\preceq_\phi \mathfrak{P}, \mathfrak{N}\preceq_\phi \mathfrak{P}$ and $M\subseteq N$
    then $\mathfrak{M}\preceq_\phi \mathfrak{N}$.
    \item (Isomorphism correctness) suppose $f:\mathfrak{M}\to \mathfrak{P}$ is an isomorphism and $\mathfrak{N}$
    is a substructure of $\mathfrak{M}$, then 
    \[
    \mathfrak{N} \preceq_\phi \mathfrak{M}   \hspace{3mm} \Longleftrightarrow \hspace{3mm} f[\mathfrak{N}] \preceq_\phi \mathfrak{P}.
    \]
    \item (Extending first order $\preceq_{\omega,\omega}$) If $\mathfrak{M}\preceq_\phi \mathfrak{N}$ then 
    $\mathfrak{M}\preceq_{\omega,\omega} \mathfrak{N}$.
\end{enumerate}
\noindent Assume that $\mathscr L$ is a logic which has a $\phi$-submodel relation $\preceq_\phi$, for each formula $\phi$ in each signature $\tau$. We define two additional properties that the relation might satisfy.
\begin{itemize}
    \item[(i)] (Union Property) Suppose that 
    \[
    \mathfrak{M}_0 \preceq_\phi \mathfrak{M}_1 \preceq_\phi \ldots \preceq_\phi \mathfrak{M}_i \preceq_\phi \ldots
    \]
    is a countable chain of $\tau$-structures and let $\mathfrak{M} = \bigcup_i \mathfrak{M}_i$. Then $\mathfrak{M}_i$ and $\mathfrak{M}$
    agree on $\phi$, for every $i$.
    \item[(ii)] (L\"{o}wenheim-Skolem Number) There is a cardinal $\lambda$ such that, for every structure $\mathfrak{M}$
    in the vocabulary of $\phi$ and every subset $A$ of $M$ of cardinality $\leq \lambda$ there is a substructure 
    $\mathfrak{N}$ of $\mathfrak{M}$ of cardinality $\leq \lambda$ such that $A\subseteq N$ and 
    \[
    \mathfrak{N}\preceq_\phi \mathfrak{M}.
    \]
    It if exists, this $\lambda$ is called the {\em L\"{o}wenheim-Skolem number} of $\phi$.
\end{itemize}
\end{definition}

\begin{definition}\label{definition:phi-submodel-relation} Let $\kappa$ be a cardinal and $\alpha <\kappa$ an ordinal. 
    Let $\phi$ be an $\mathscr L^1_{\kappa,\alpha}[\tau]$-sentence in vocabulary $\tau$. 
    Let $\theta=\theta_\phi$ and $\beta=\beta_\phi$, as in Definition~\ref{definition:theta-alpha-logic}. We may assume that $\tau$ is of cardinality $\leq \theta$. 
    Suppose $\mathfrak{M}$ and $\mathfrak{N}$ are $\tau$-structures. We define the relation $\preceq_\phi$
    by letting $\mathfrak{M}\preceq_\phi \mathfrak{N}$ if $\mathfrak{M}$ is a substructure of $\mathfrak{N}$
    and, for every sequence $\bar{a}\in M^\theta$,
    \[
    (\mathfrak{M},\bar{a}) \equiv^{2\cdot \beta}_{\theta,\alpha}(\mathfrak{N},\bar{a}).
    \]
\end{definition}

\begin{lemma}[The Union Lemma] Let $\kappa$ be a cardinal. 
Let $\phi$ be an $\mathscr L^1_{\kappa,\omega}[\tau]$-sentence, for some vocabulary $\tau$. 
Let $(\mathfrak{M}_i)_i$ be an $\preceq_\phi$-chain of length $\omega$, and let $\mathfrak{M}$ be its union.
Then $\mathfrak{M}_0\equiv_\phi \mathfrak{M}$.
\end{lemma}

\begin{proof}
    Let $\theta=\theta_\phi$ and $\beta=\beta_\phi$. We may assume that the cardinality of $\tau$ is at most $\theta$. 
    It suffices to show that Eve has a winning strategy in $\DG^\beta_{\theta,\omega}(\mathfrak{M}_0,\mathfrak{M})$.
    To avoid confusion let $\mathfrak{N}$ be a $\tau$-structure isomorphic to $\mathfrak{M}_0$ but disjoint from 
    $\mathfrak{M}$.  
    
    Before we start let us observe that if $\mathfrak{A}$, $\mathfrak{B}$, and $\mathfrak{C}$
    are three structures in the same vocabulary of size $\leq \theta$, $\bar{a}$, $\bar{b}$, and $\bar{c}$ are sequences of
    the same length $\leq\theta$ in $A$, $B$, and $C$ respectively, and $\gamma$ is an ordinal, then:
\[
[(\mathfrak{A},\bar{a}) \equiv^\gamma_\theta (\mathfrak{B},\bar{b}) \hspace{2mm} \& \hspace{2mm} (\mathfrak{B},\bar{b}) \equiv^\gamma_\theta (\mathfrak{C},\bar{c})]
\Longrightarrow (\mathfrak{A},\bar{a}) \equiv^\gamma_\theta (\mathfrak{C},\bar{c}).
\]
Next observe that if $g\in {\mathcal I}^{2\cdot \gamma}_\theta(\mathfrak{A}, \mathfrak{B})$,
and we let $0_g$ be the constant function $0$ on ${\rm dom}(g)\cup {\rm ran}(g)$, 
then 
\[
(\gamma,{\rm dom}(g), {\rm ran}(g),g,0_g)
\]
\noindent  is a winning position for Eve in $\DG^\gamma_{\theta,1}(\mathfrak{A},\mathfrak{B})$.

We now describe a winning strategy for Eve in $\DG^\beta_{\theta,\omega}(\mathfrak{N},\mathfrak{M})$.
To begin, let $h:M\to \omega$ be defined by letting 
\[
h(x)= \min \{ n: x \in M_n\}.
\]
Suppose Adam starts by playing $(\beta_0,B^0_0,B^0_1)$, where $\beta_0<\beta$, $B^0_0\in [N]^{\leq \theta}$,
and $B^0_1\in [M]^{\leq \theta}$. Let $\pi$ be the isomorphism between $\mathfrak{N}$ and $\mathfrak{M}_0$,
and let $A^0_0$ be the substructure of $\mathfrak{N}$ generated by $B^0_0 \cup \pi^{-1}[B^0_1\cap M_0]$. 
Let $A^0_1=\pi [A^0_0]\cup B^0_1$, let $g_0= \pi \rest A^0_0$, and let $h_0$ be the function on $A^0_0 \cup A^0_1$
defined by: 
\[
h_0(x)= \begin{cases}
			0, & \text{if $x\in A^0_0$}\\
           h(x), & \text{if $x\in A^0_1$}
\end{cases}
\]
Note that if $x\in A^0_1$ and $h(x)=0$ then $x\in {\rm ran}(g_0)$.
Eve then plays $(\beta_0, A^0_0,A^0_1,g_0,h_0)$.
Let $\bar{a}_0$ be a sequence enumerating $A^0_0$ and let $\bar{b}_0$ be the sequence of the same length
such that $\bar{b}_0(\xi)=g_0(\bar{a}_0(\xi))$, for all $\xi$. Observe that
\[
(\mathfrak{N},\bar{a}_0)\equiv^{2\cdot \beta_0}_\theta (\mathfrak{M}_0,\bar{b}_0).
\]
At stage $n$ we will have a position $(\beta_n,A^n_0,A^n_1,g_n,h_n)$.
We will have that $A^n_0\subseteq {\rm dom}(g_n)$ and $A^n_1\cap M_n \subseteq {\rm ran}(g_n)$.
We will have that $h_n$ is equal to $0$ on $A^n_0$ and on $A^n_1\cap M_n$.
For $x\in A^n_1\setminus M_n$ we will have 
\[
h_n(x)=h(x)-n.
\]
Moreover, there will be an enumeration $\bar{a}_n$ of ${\rm dom}(g_n)$ and 
an enumeration $\bar{b}_n$ of ${\rm ran}(g_n)$, such that $\bar{b}_n(\xi)=g_n(\bar{a}_n(\xi))$, for all $\xi$,
and
\[
(\mathfrak{N},\bar{a}_n)\equiv^{2\cdot \beta_n}_\theta (\mathfrak{M}_n,\bar{b}_n).
\]
By transitivity we will also have that: 
\[
(\mathfrak{N},\bar{a}_n)\equiv^{2\cdot \beta_n}_\theta (\mathfrak{M}_{n+1},\bar{b}_n).
\]
Suppose at stage $n+1$ Adam plays  $(\beta_{n+1},B^{n+1}_0,B^{n+1}_1)$, where $\beta_{n+1} <\beta_n$, 
$B^{n+1}_0\in [N]^{\leq \theta}$, and $B^{n+1}_1\in [M]^{\leq \theta}$. We may assume that $B^n_1\subseteq B^{n+1}_1$.
Eve finds some $g_{n+1}\in {\mathcal I}^{2\cdot \beta_{n+1}}_\theta(\mathfrak{N},\mathfrak{M}_{n+1})$
such that $B^{n+1}_0\subseteq {\rm dom}(g_{n+1})$ and $B^{n+1}_1\cap M_{n+1}\subseteq {\rm ran}(g_{n+1})$. 
Let $A^{n+1}_0={\rm dom}(g_{n+1})$. We may assume that $A^{n+1}_0$ is a substructure of $\mathfrak{N}$.
Let $A^{n+1}_1=g_{n+1}[A^{n+1}_0]\cup B^{n+1}_1$.
%and let $h_1:A^1_0\cup A^1_1\to \omega$ be defined by:
%\[
%h_1(x)= \begin{cases}
%			0, & \text{if $x\in A^1_0$}\\
%           h(x)\dot -1, & \text{if $x\in A^1_1$}
%\end{cases}
%\]
Observe that if $x\in A^{n+1}_1$ and $h_{n+1}(x)=0$ then $x\in {\rm ran}(g_{n+1})$.
Let $\bar{a}_{n+1}$ be a sequence enumerating $A^{n+1}_0$ and let $f_n:\theta\to \theta$ be such that 
$\bar{a}_n \sqsubseteq_{f_n}\bar{a}_{n+1}$. Let $\bar{b}_{n+1}$ be defined by letting 
$\bar{b}_{n+1}(\xi)=g_{n+1}(\bar{a}_{n+1}(\xi))$, for all  $\xi$. Then we have that $\bar{b}_n \sqsubseteq_{f_n}\bar{b}_{n+1}$. 
We also have that:  
\[
(\mathfrak{N},\bar{a}_{n+1})\equiv^{2\cdot \beta_{n+1}}_\theta (\mathfrak{M}_{n+1},\bar{b}_{n+1}).
\]
In her next move Eve plays $(\beta_{n+1},A^{n+1}_0,A^{n+1}_1,g_{n+1},h_{n+1})$.
This completes the inductive step and finishes the proof of the lemma.     
\end{proof}

In \cite{MR2869022} Shelah gave a Lindstr\"{o}m type characterization of his logic $\mathbb L^1_\kappa$,
for $\kappa$ which is a fixed point of the $\beth$-function. In this characterization compactness is weakened
to the property called {\em strong undefinability of well-orders} at $\kappa$ and denoted by ${\rm SUDWO}_\kappa$.
This property roughly says that if a sentence $\phi$ has a standard model of the form $(H(\lambda),\in,\kappa,\ldots)$
then it has a non-standard one in which the ordinals are ill-founded, and in addition the model is the union of a countable
chain of small submodels which belong to it. 
The following is an abstract version of Conclusion 3.12 from \cite{MR2869022} stated for a logic having the $\phi$-submodel relation. See also \cite[Theorem 3.10]{KVV}.

\begin{lemma}[The SUDWO$_\kappa$ Lemma]\label{lemma-sudwo}
 Let $\kappa$ be a strong limit cardinal. Assume $\mathscr L$ is a regular logic which has the $\phi$-submodel relation
    with the Union Property and the L\"{o}wenheim-Skolem number strictly below $\kappa$, for every sentence $\phi$.
    Let $\tau$ be a vocabulary of cardinality $<\kappa$ containing the symbol $\in$, and a unary predicate symbol $K$. 
    Suppose, for some $\lambda \geq \kappa$,  a sentence $\phi\in \mathscr L[\tau]$ has a model of the form
    \[
    \mathfrak{A}= (H(\lambda),\in, \kappa, \ldots)
    \]
    where the interpretation of $\in$ is the standard $\in$-relation on $H(\lambda)$ and the predicate $K$ is interpreted as the set $\kappa$. Then $\phi$ has a model of the form 
    \[
    \mathfrak{B}= (B,\in^ \mathfrak{B}, K^\mathfrak{B}, \ldots)
    \]
    with some $b\in K^\mathfrak{B}$ and a sequence $\{ B_n : n <\omega \}$ of elements of $B$ such that: 
    \begin{enumerate}
    \item $\in^\mathfrak{B}$ is ill-founded below $b$, 
    \item for all $x\in B$ there is $n$ such that $x\in ^\mathfrak{B} B_n$, 
    \item $\mathfrak{B}\models| B_n | \leq b$, for all $n$.
    \end{enumerate}
\end{lemma}
\begin{proof}
    Suppose $\phi$ is an $\mathscr L[\tau]$-sentence which has a model of the form 
    \[
    \mathfrak{A}= (H(\lambda),\in, \kappa, \ldots).
    \]
   Since $\mathscr L_{\omega,\omega} \leq \mathscr L$ we may assume that $\phi$ implies a sufficient (first order) fragment
   of ${\rm ZFC}^-$. Let $\theta <\kappa$ be the L\"{o}wenheim-Skolem number of $\phi$. We must have that $| \tau | \leq \theta$.
    We recursively build models $(P_n,\ldots)\preceq_\phi \mathfrak{A}$ of size at most $\theta$ together with 
    embeddings $e_n: P_n \to P_{n+1}$ such that $e_n[P_n]\preceq_\phi P_{n+1}$. Moreover, for each $n$ there is a distinguished element $c_n\in P_n$
    such that $c_{n+1} \in e_n(c_n)$, and $e_n[P_n]\in P_{n+1}$, for all $n$.
    We will then take the direct limit of this system which will give us the desired model $\mathfrak{B}$.

\medskip

\noindent {\bf Step $0$}. For every $\beta \in (2^\theta)^+ \setminus \theta$, 
we choose a model $B^0_\beta\preceq_\phi \mathfrak{A}$ of size $\leq \theta$ such that $\beta\in B^0_\beta$.
There are at most $2^\theta$ isomorphism classes of models $(B^0_\beta,\in,\beta,\ldots)$, 
so we can fix an unbounded subset $X_0$ of $(2^\theta)^+$ such that for all $\beta,\beta' \in X_0$
there is an isomorphism 
\[
(B^0_\beta,\in,\beta,\ldots) \simeq (B^0_{\beta'},\in, \beta',\ldots).
\]
Pick a representative of this class, call it  $(P_0,\in, c_0,\ldots)$. Thus, for every $\beta \in X_0$
there is an isomorphism of $B^0_\beta$ to $P_0$ which maps $\beta$ to $c_0$.

\noindent {\bf Step $1$}. For each $\beta \in (2^\theta)^+\setminus \theta$ pick some $x_0(\beta)\in X_0$ above $\beta$,
and then apply the L\"{o}wenheim-Skolem Property to choose some $B^1_\beta \preceq_\phi \mathfrak{A}$ of size $\leq \theta$
such that
\[
B^0_{x_0(\beta)} \cup \{ \beta \} \cup \{B^0_{x_0(\beta)}\} \subseteq B^1_\beta.
\]
By the Tarski-Vaught Property we have that $B^0_{x_0(\beta)}\preceq_\phi B^1_\beta$.
Consider the structures $(B^1_\beta,\in,\beta,B^0_{x_0(\beta)},\ldots)$
By the Pigeon Hole Principle we can find an unbounded subset $X_1$ of $(2^\theta)^+$ such that, for all $\beta,\beta' \in X_1$
there is an isomorphism 
\[
(B^1_\beta,\in,\beta, B^0_{x_0(\beta)},\ldots) \simeq (B^1_{\beta'}, \in, \beta', B^0_{x_0(\beta')},\ldots).
\]
which maps $\beta \mapsto \beta'$ and $B^0_{x_0(\beta)} \mapsto B^0_{x_0(\beta')}$.
Let $(P_1,\in,c_1,R_0,\ldots)$ be a representative of this isomorphism class. 
Note that $P_1\models | R_0 |\leq c_1$. 
Moreover, we have an isomorphic embedding $e_0:P_0\to P_1$ such that $e_0[P_0]=R_0$,
and $c_1 \in e_0(c_0)$. By Isomorphism Correctness it follows that $(R_0,\in,\ldots)\preceq_\phi (P_1,\in,\ldots)$.

\medskip

\noindent {\bf Step $n+1$}. Suppose $n\geq 1$ and we have done the construction up to $n$. 
In particular, we have a structure $(P_n,\in,c_n,R_{n-1},\ldots)$ of size $\leq \theta$, and an unbounded subset $X_n$ of $(2^\theta)^+\setminus \theta$
such that for each $\beta\in X_n$ there is an isomorphic copy of $(P_n,\in,c_n,R_{n-1},\ldots)$ of the form 
\[
(B^n_\beta,\in,\beta,B^{n-1}_{x_{n-1}(\beta)},\ldots),
\]
where $x_{n-1}(\beta)\in X_{n-1}$ and $\beta < x_{n-1}(\beta)$.
We now repeat the construction from Step 1 to find a structure $(P_{n+1},\in, c_{n+1},R_n,\ldots)$ and an unbounded subset $X_{n+1}$ of $(2^\theta)^+\setminus \theta$,
such that for every $\beta \in X_{n+1}$ there is $B^{n+1}_\beta\preceq_\phi \mathfrak{A}$, of cardinality $\leq \theta$ 
such that, for some $x_n(\beta)\in X_n$ with $\beta <x_n(\beta)$ we have that $\beta, B^n_{x_n(\beta)}\in B^{n+1}_\beta$ and 
\[
(B^{n+1}_\beta,\in,\beta,B^n_{x_{n}(\beta)},\ldots) \simeq (P_{n+1},\in, c_{n+1},R_n,\ldots).
\] 
Then $R_n$ must be isomorphic to $P_n$. Let $e_n$ be the unique isomorphism. 
Finally, since 
\[
(B^n_{x_n(\beta)},\in,\ldots) \preceq_\phi (B^{n+1}_\beta,\in,\ldots),
\]
for all $\beta \in X_{n+1}$, by Isomorphism Correctness we must have:
\[
(R_n, \in,\ldots)\preceq_\phi (P_{n+1},\in,\ldots).
\]

Suppose we have completed the construction for all $n$. We obtain a directed system $(P_n,e_n ; n<\omega)$.
We know that if $R_n=e_n[P_n]$, then $R_n\in P_{n+1}$ and $R_n \preceq_\phi P_{n+1}$.
%where $e_n:P_n \to P_{n+1}$ is an embedding, and if  $R_n=e_n[P_n]$ then 
%\[
%(R_n,\in,\ldots) \preceq_\phi (P_{n+1},\in,\ldots).
%\]
We also have elements $c_n\in P_n$ satisfying the predicate $K^{P_n}$, such that $c_{n+1} \in e_n(c_n)$, for each $n$. 
Finally, we have that $P_{n+1}\models | R_n | \leq c_{n+1}$.

Let $\mathfrak{B}= (B,\in^{\mathfrak{B}},\ldots)$ be the direct limit of this system. The interpretation of the symbol $\in$ is now 
some binary relation $\in^\mathfrak{B}$ on $B$. Let $j_n: P_n \to B$ be the limit embedding, for all $n$. 
Let $B_n =j_{n+1}(R_n)$, and let $b_n=j_n(c_n)$. Then we must have $K^{\mathfrak{B}}(b_n)$,  
and $b_{n+1}\in^{\mathfrak{B}}b_n$, for all $n$. 
Hence the relation $\in^{\mathfrak{B}}$ is ill-founded below $b=b_0$.
Moreover, $P_{n+1}\models | R_n | \leq c_{n+1}$, hence we must have $B\models | B_n | \leq b_{n+1}$.
Since $b_{n+1} \in ^B \ldots \in ^B b_0 = b$ and $B$ satisfies a sufficient fragment of ${\rm ZFC}^-$, it follows that 
$B \models |B_n| \leq b$, for all $n$.
Every element $z$ of $B$ is of the form $j_n(x)$, for some $x\in P_n$ and some $n$.
Note that then $e_n(x)\in R_n$ and hence $j_{n+1}(e_n(x))\in^{\mathfrak{B}} B_n$ and $z=j_{n+1}(e_n(x))$.

It remains to check that $\mathfrak{B}\models_{\mathscr L}\phi$. 
First note that $P_n\models_{\mathscr L}\phi$, for all $n$, since $P_n$ is isomorphic to $B^n_\beta$, for any $\beta\in X_n$.
We already observed that $R_n\preceq_\phi P_{n+1}$, for all $n$.
By the Isomorphism Correctness we must have $B_n\preceq_\phi B_{n+1}$, for all $n$, as well. 
Finally, by the Union Property, we have that $B$ and $B_n$ agree on $\phi$, for any $n$.
It follows that $\mathfrak{B}\models_{\mathscr L}\phi$, as required. 

\end{proof}

Shelah \cite{MR2869022} used the SUDWO to show that his logic $\mathbb L^1_\kappa$, for $\kappa=\beth_\kappa$,
satisfies a version of Lindstr\"{o}m's Theorem as well as the Craig Interpolation Theorem. 
We now formulate a property that implies both of these results. 
Let $\mathscr L$ be a regular logic and let $\tau$ be a vocabulary. 
Recall that a class $\mathcal K$ of $\tau$-structures is called {\em projective} (\cite{MR0066302}) in $\mathscr L$ if there is a larger vocabulary $\sigma\supseteq \tau$ 
and an $\mathscr L[\sigma]$-sentence $\phi$ such that a $\tau$-structure $\mathfrak{M}$ belongs to $\mathcal K$
iff there it has an expansion to a $\sigma$-structure $\mathfrak{N}$ which satisfies $\phi$.

We will also need a slightly more general notion of a {\em relativized projective class} of models (\cite{MR0170817}). 
Suppose $\sigma$ is a vocabulary containing a distinguished unary predicate symbol $P$ and $\tau$ is a subvocabulary of 
$\sigma \setminus \{ P\}$.  Suppose $\mathfrak{N}$ is a $\sigma$-structure and let $P^{\mathfrak{N}}$
be the interpretation of $P$ in $\mathfrak{N}$. We say that $P^{\mathfrak{N}}$ is $\tau$-{\em closed}
if $c^{\mathfrak{N}}\in P^{\mathfrak{N}}$, for all constant symbols $c\in \tau$, and $P^{\mathfrak{N}}$
is closed under $f^{\mathfrak{N}}$, for all function symbols $f\in \tau$. 
If this is the case restricting $\mathfrak{M}$ to $P^{\mathfrak{N}}$ and the vocabulary $\tau$
we obtain a $\tau$-structure that we denote by 
\[
\mathfrak{N}\rest (P,\tau).
\]
If  $\mathfrak{M}=\mathfrak{N}\rest (P,\tau)$, then we will say that 
$\mathfrak{N}$ is a {\em generalized expansion} of $\mathfrak{M}$.
Let $\mathscr L$ be a regular logic and let $\tau$ be a vocabulary.  
We will say that a class $\mathcal K$ of $\tau$-structures is {\em relativized projective} in $\mathscr L$
if there is a unary predicate symbol $P\notin \tau$, a larger vocabulary $\sigma \supseteq \tau \cup \{ P\}$ and  $\phi\in \mathscr L[\sigma]$
such that a $\tau$-structure $\mathfrak{M}$ belongs
to $\mathcal K$ iff it has a generalized expansion to a $\sigma$-structure $\mathfrak{N}$ which satisfies $\phi$.
Recall that a class $\mathcal K$ of models in some vocabulary $\tau$ 
is definable in a logic $\mathscr L$ if there is an $\mathscr L$ formula $\phi$
such that $\mathcal K$ is the set of $\tau$-structures satisfying $\phi$.

\begin{lemma}[The Separation Lemma]\label{lemma:separation}
    Suppose $\kappa$ is a cardinal such that $\kappa=\beth_\kappa$ and $\alpha <\kappa$ is an ordinal.
  Let $\mathscr L$ be a regular logic which has occurrence number at most $\kappa$ 
  and satisfies the ${\rm SUDWO}_\kappa$ Lemma and 
  such that $\mathscr L_{\delta,\omega}\leq \mathscr L$, for all $\delta <\kappa$.
  Let $\mathcal K_0$ and $\mathcal K_1$ be two disjoint classes of structures in some vocabulary $\tau$
  which are both relativized projective classes in $\mathscr L$. 
  Then $\mathcal K_0$ and $\mathcal K_1$ can be separated by a class of $\tau$-structures
  definable in $\mathscr L^1_{\kappa,\alpha}$.
\end{lemma}
\begin{proof}
    Fix a unary predicate symbol $P\notin \tau$ and a vocabulary $\sigma\supseteq \tau \cup \{ P \}$ 
    of size $<\kappa$
    such that for some $\mathscr L[\sigma]$ sentences $\phi_0$ and $\phi_1$, for $i=0,1$,
    \[
    \mathcal K_i= \{ \mathfrak{M}\rest (P,\tau): \mathfrak{M}\models_{\mathscr L}\phi_i \}.
    \]
    We need to find an $\mathscr L^1_{\kappa,\alpha}[\tau]$-sentence $\psi$ such that 
    \[
    \mathcal K_0 \subseteq {\rm Mod}^\tau_{\mathscr L^1_{\kappa,\alpha}}(\psi) \hspace{2mm} \mbox{ and } \hspace{2mm}
     \mathcal K_1 \cap {\rm Mod}^\tau_{\mathscr L^1_{\kappa,\alpha}}(\psi) = \emptyset.
    \]
    Suppose no such $\psi$ exists. Then for each $\beta <\kappa$ we can find two $\sigma$-structures
    $\mathfrak{M}_{\beta,0}$ and $\mathfrak{M}_{\beta,1}$ such that, for $i=0,1$,
    \[
    \mathfrak{M}_{\beta,i} \models_{\mathscr L}\phi_i
    \]
    but 
    \[
    \mathfrak{M}_{\beta,0}\rest (P,\tau) \equiv^\beta_{| \beta |, \alpha}\mathfrak{M}_{\beta,1}\rest (P,\tau).
    \]
    Recall that this means that there is a finite sequence of $\tau$-structures $\mathfrak{N}_\beta^0, \ldots, \mathfrak{N}_\beta^k$,
    such that $\mathfrak{M}_{\beta,0}\rest (P,\tau)=\mathfrak{N}_\beta^0$, $\mathfrak{M}_{\beta,1} \rest (P,\tau)=\mathfrak{N}_\beta^k$,
    and such that Eve has a winning strategy in $\DG^\beta_{| \beta |, \alpha}(\mathfrak{N}_\beta^i, \mathfrak{N}_\beta^{i+1})$,
    for all $i<k$. Let $\gamma$ be the natural sum of $\omega$ copies of $\alpha$. 
    By Proposition \ref{proposition:natural-sum} we have that Eve has a winning strategy in 
    \[
    \DG^\beta_{| \beta |, \gamma}(\mathfrak{M}_{\beta,0} \rest (P,\tau), \mathfrak{M}_{\beta,1} \rest (P,\tau)).
    \]
    Let $\lambda$ be a sufficiently large regular cardinal such that all the structures and winning strategies are in $H(\lambda)$.
    We now define a structure 
    \[
           \mathfrak{A} = (H(\lambda), \in,\kappa, \ldots)
    \]
    inside which we code the structures $\mathfrak{M}_{\beta,0}$ and $\mathfrak{M}_{\beta,1}$ with the help of a large signature. 
    We first introduce a unary symbol $K$ and interpret it as $\kappa$. 
    Let $\mathcal M$  
    be a ternary relation symbol. 
    We define its interpretation as follows:
\[
{\mathcal M}^{\mathfrak{A}}(\beta,i,a)  \hspace{1mm} \Longleftrightarrow \hspace{1mm} \beta \in \kappa, i\in \{ 0,1\}, \mbox{ and } a\in M_{\beta,i}.
\]
   Let $\sigma_+$ be the vocabulary obtained by adding two to the arity of each symbol in $\sigma$. 
   Define the interpretation of symbols $c_+,f_+,R_+ \in \sigma_+$ as follows:
    \[ c^{\mathfrak{A}}_+(\beta,i)= \begin{cases}
        c^{\mathfrak{M}_{\beta,i}}, \text{ if $\beta <\kappa$ and  $i\in \{ 0,1\}$,} \\
        \emptyset,   \text{ otherwise.} 
    \end{cases}
    \]
    \[ 
    f^{\mathfrak{A}}_+(\beta,i)(\bar{a})= \begin{cases}
        f^{\mathfrak{M}_{\beta,i}}(\bar{a}), \text{ if $\beta <\kappa$ and  $i\in \{ 0,1\}$,} \\
        \emptyset,   \text{ otherwise.} 
    \end{cases}
   \]
    \[ R^{\mathfrak{A}}_+(\beta,i)(\bar{a})  \Longleftrightarrow \beta <\kappa, i \in \{ 0,1\} \mbox{ and } R^{\mathfrak{M}_{\beta,i}}(\bar{a}).
   \]
   With this vocabulary we can define the structures $\mathfrak{M}_{\beta,i}$ using as parameters $\beta$ and $i$.
   We introduce symbols $\dot{\tau}$, $\dot{\sigma}$ and $\dot{\sigma}_+$ and we interpret them 
   as $\tau$, $\sigma$ and $\sigma_+$. We also introduce a symbol $\dot{\gamma}$ which we interpret as the ordinal $\gamma$. 
   Let  $\delta=\max(| \sigma_+|, | \gamma|)$.
   Let $\Phi$ be the conjunction of the following sentences in the above vocabulary:
   \begin{enumerate}
        \item an $\mathscr L_{\delta^+,\omega}$-sentence:
       \[
        \forall x [ x\in \dot{\sigma} \longleftrightarrow \bigvee_{s \in \sigma}x= s]
       \]
       and similarly for $\tau$ and $\sigma_+$.
       \item an $\mathscr L_{\delta^+,\omega}$-sentence describing $\dot{\gamma}$ as the ordinal $\gamma$.
       \item a finite fragment of ${\rm ZFC}^-$ sufficient to express that, for $\beta \in K$ and $i\in \{ 0,1\}$,
       \begin{enumerate}
           \item $\mathfrak{M}_{\beta,i}$ is a $\dot{\sigma}$-structure, 
           \item $\mathfrak{M}_{\beta,i}\rest (P,\dot{\tau})$ is $\dot{\tau}$-closed, 
           \item Eve has a winning strategy in 
      \[
      {\DG}^\beta_{|\beta |,\dot{\gamma}}(\mathfrak{M}_{\beta,0}\rest (P,\dot{\tau}),\mathfrak{M}_{\beta,1}\rest (P,\dot{\tau})).
      \]
       \end{enumerate}
       \item an $\mathscr L$-sentence saying that 
       \[
       \forall \beta \in K \, \forall i \in \{ 0,1\} \, \mathfrak{M}_{\beta,i}\models_{\mathscr L}\phi_i.
       \]
       Such a sentence exists since the logic $\mathscr L$ satisfies atomic substitution. 
   \end{enumerate}
Since $\mathscr L_{\delta^+,\omega}\leq \mathscr L$ we conclude that $\Phi$ is an $\mathscr L$-sentence.
Since $\Phi$ holds in $\mathfrak{A}$, by the ${\rm SUDWO}_\kappa$-Lemma,  $\Phi$ has a model of the form:
    \[
    \mathfrak{B}= (B,\in^ \mathfrak{B}, K^\mathfrak{B}, \ldots)
    \]
    with some $b\in K^\mathfrak{B}$ and a sequence $\{ B_n : n <\omega \}$ of elements of $B$ such that: 
    \begin{enumerate}
    \item $\in^\mathfrak{B}$ is ill-founded below $b$, 
    \item for all $x\in B$ there is $n$ such that $x\in ^\mathfrak{B} B_n$, 
    \item $\mathfrak{B}\models| B_n | \leq b$, for all $n$.
    \end{enumerate}
Now, consider the structures $\mathfrak{M}^{\mathfrak{B}}_{b,i}$, for $i=0,1$, defined in the model $\mathfrak{B}$.
They can be seen as $\sigma$-structures. We then have
\[
 \mathfrak{M}^{\mathfrak{B}}_{b,i} \models_{\mathscr L} \phi_i, \mbox{ for } i =0,1.
\]
Let $\mathfrak{P}_i= \mathfrak{M}^{\mathfrak{B}}_{b,i}\rest (P,\tau)$.
Then the $\mathfrak{P}_i$ are $\tau$-structures and  we have $\mathfrak{P}_i\in \mathcal K_i$, for $i=0,1$. 
We will write $P_i$ for the domain of the structure $\mathfrak{P}_i$.
We show that $\mathfrak{P}_0$ and $\mathfrak{P}_1$ are isomorphic, which contradicts the fact that $\mathcal K_0$ and $\mathcal K_1$ are disjoint. In order to see this, observe that: 
\[
\mathfrak{B}\models E^b_{|b|,\dot{\gamma}}(\mathfrak{P}_0, \mathfrak{P}_1).
\]
Now, the point is that $\dot{\gamma}^{\mathfrak{B}}$ is isomorphic to $\gamma$ and is well-ordered, 
while $\in^{\mathfrak{B}}$ is ill-founded below $b$. 
For simplicity let us assume that the well-founded part of $\mathfrak{B}$ is transitive. 
Let $\Sigma$ be a winning strategy for Eve in the game 
\[
{\DG}^b_{| b|, \gamma}(\mathfrak{P}_0,\mathfrak{P}_1)
\]
in the model $\mathfrak{B}$.
We fix an $\in^{\mathfrak{B}}$-decreasing sequence $(b_n)_n$ with $b_0\in^{\mathfrak{B}}b$. 
We take the role of Adam and generate a run of the game by playing at stage $n$
\[
(b_n,B_n\cap P_0,B_n\cap P_1)
\] 
as defined in the structure $\mathfrak{B}$. 
We can interpret the move of the strategy  $\Sigma$ at stage $n$ as a tuple of the form
\[
(b_n,A^0_n,A^1_n,g_n,h_n),
\]
where $h_n$ is a height function on $A^0_n\cup A^1_n$ taking values in $\gamma$ and  $
g_n$ is a partial isomorphism between $h_n^{-1}(0)\cap A^0_n$ and $h_n^{-1}(0)\cap A^1_n$.
We have that:
\[
g_0\subseteq g_1 \subseteq g_2 \sqsubseteq\ldots
\]
and the height functions $h_n$ are decreasing at every element until they reach $0$.
Since $\gamma$ is well-ordered, for every element $x\in A^0_n\cup A^1_n$ its height $h_k(x)$ must reach $0$, at some stage $k\geq n$,
and then $x$ is either in the domain or range of $g_k$.
Now, $B= \bigcup_n B_n$, so every element $x$ of $P_0\cup P_1$
appears at some stage. 
It follows that $g= \bigcup_n g_n$ is an isomorphism between $\mathfrak{P}_0$ and $\mathfrak{P}_1$.
This is the desired contradiction and finishes the proof of the lemma.     
\end{proof}

We now have the following characterization which is analogous to the one given by Shelah \cite{MR2869022}.

\begin{theorem}[Lindstr\"{o}m's Theorem]\label{thm:Lindstrom}
    Let $\kappa$ be a fixed point of the $\beth$-function. 
    The logic $\mathbb L^1_\kappa$ is the maximal regular logic which is above $\mathscr L_{\delta,\omega}$, for all $\delta <\kappa$, 
    has occurrence number $\leq \kappa$, and has a $\psi$-submodel relation $\preceq_\psi$ with the Union Property and the 
    L\"{o}wenheim-Skolem number below $\kappa$, for each sentence $\psi$.
    \qed
\end{theorem}
We also have Craig's interpolation theorem for the logic $\mathbb L^1_\kappa$, for $\kappa$ such that $\kappa=\beth_\kappa$.
Suppose $\mathscr L$ is a logic, $\phi_0$ and $\phi_1$ are $\mathscr L$-sentences in vocabularies $\tau_0$ and $\tau_1$.
We write 
\[
\phi_0\models \phi_1
\]
if, for every $(\tau_0\cup \tau_1)$-structure $\mathfrak{M}$,
if $\mathfrak{M}\models_{\mathscr L} \phi_0$ then $\mathfrak{M}\models_{\mathscr L} \phi_1$.

%JV
\begin{theorem}[Craig Interpolation Theorem \cite{MR2869022}]\label{thm:Craig}
Suppose $\kappa=\beth_\kappa$. Suppose $\tau_0$ and $\tau_1$ are vocabularies and let $\tau=\tau_0 \cap \tau_1$.
Let $\phi_0\in \mathbb L^1_\kappa[\tau_0]$ and $\phi_1\in \mathbb L^1_\kappa[\tau_1]$.
Suppose $\phi_0 \models \phi_1$. Then there is $\psi\in \mathbb L^1_\kappa[\tau]$
such that $\phi_0\models \psi$ and $\psi \models \phi_1$.
\end{theorem}
\begin{proof}
    Let $\mathcal K_0$ be the class of $\tau$-structures that have a $\tau_0$-expansion which satisfies $\phi_0$
    and let $\mathcal K_1$ be the class of $\tau$-structures that have a $\tau_1$-expansion that satisfies $\neg \phi_1$.
    By our assumption $\mathcal K_0$ and $\mathcal K_1$ are disjoint. 
    By Lemma \ref{lemma:separation} there is a sentence $\psi\in \mathbb L^1_\kappa[\tau]$
    such that 
    \[
    \mathcal K_0 \subseteq {\rm Mod}^\tau_{\mathbb L^1_\kappa}(\psi) \hspace{2mm} \mbox{ and } \hspace{2mm}
     \mathcal K_1 \cap {\rm Mod}^\tau_{\mathbb L^1_\kappa}(\psi)=\emptyset.
     \]
     It follows that $\phi_0\models \psi$ and $\psi \models \phi_1$, as required. 
\end{proof}

 \section{Open problems}\label{openproblems}
 %\todo{An old list!}}
 
\begin{enumerate}
\item The logic $\mathbb L^1_\kappa$ is defined via a semantic game, not by giving recursive rules of its syntax. This is in sharp contrast to logics such as $\mathscr L_{\kappa\omega}$ and $\mathscr L_{\kappa\kappa}$. Is there a nice recursive syntax for $\mathbb L^1_\kappa$?  For  partial positive answers, see \cite{MR4357456} and \cite{KVV}.
\item We know that any valid implication in $\mathscr L_{\kappa,\omega}$, $\kappa$ regular, has an  interpolant  in $\mathscr L_{(2^{<\kappa})^+,\kappa}$ \cite{MR290943}. 
Is there always for  regular $\kappa$ a logic between $\mathscr L_{\kappa,\omega}$ and $\mathscr L_{(2^{<\kappa})^+,\kappa}$ which has the Interpolation Property? By Theorem~\ref{thm:Craig} such a logic exists when $\kappa=\beth_\kappa$.
\item Is the relation $E^\beta_{\theta,\omega}(\mathfrak{M}, \mathfrak{N})$ transitive? Is the relation $\equiv^{\beta}_{\theta,<\omega^2}$ the transitive closure of $E^\beta_{\theta,\omega}(\mathfrak{M}, \mathfrak{N})$?
\item Is there an algebraic characterization of elementary equivalence in $\mathbb L^1_\kappa$, as there is for strongly compact $\kappa$
\cite{MR4358271}.
\item Can we express in $\mathscr L^1_{\theta,\alpha}$ the freeness of an abelian group? It can be expressed in $\mathscr L_{\kappa^+,\kappa}$ for strongly compact $\kappa$ \cite{MR0453516}. 
\item  Does $\alpha$ give rise to a proper hierarchy among the logics $\mathscr L^1_{\theta,\alpha}$? Proposition~\ref{minus a point1} manifests the difference between $\alpha=\omega$ and $\alpha=\omega+1$.
\item  Does $\theta$ give rise to a proper hierarchy among the logics $\mathscr L^1_{\theta,\alpha}$? Proposition~\ref{phiA} manifests the difference between $\theta=\aleph_0$ and $\theta=\aleph_1$.
\item Find a logic between cofinality logic and stationary logic which has the Craig Interpolation Property, and even perhaps a  Lindstr\"om Theorem. 
\end{enumerate}

%\listoftodos

\section{Acknowledgements}

We are indebted to Andr\'es Villaveces for many useful discussions on the results of this paper. We are also indebted to the anonymous referee for their careful reading of the paper and many good suggestions. This project has received funding from the European Research Council (ERC) under the
European Union’s Horizon 2020 research and innovation programme (grant agreement No
101020762) and from the Academy of Finland (grant number 368671).

\bibliographystyle{plain}

\bibliography{lonekappa}

\end{document}